\documentclass[pdflatex,sn-mathphys-num]{sn-jnl}
\usepackage{amsmath,amssymb,amsfonts,amsthm}
\usepackage{graphicx}%
\usepackage{multirow}%
\usepackage{amsmath,amssymb,amsfonts}%
\usepackage{amsthm}%
\usepackage{mathrsfs}%
\usepackage[title]{appendix}%
\usepackage{xcolor}%
\usepackage{textcomp}%
\usepackage{manyfoot}%
\usepackage{booktabs}%
\usepackage{algorithm}%
\usepackage{algorithmicx}%
\usepackage{algpseudocode}%
\usepackage{listings}%

\usepackage{enumitem}
\usepackage[mathscr]{euscript}
\usepackage{dsfont}
\usepackage{stmaryrd}
\usepackage{tikz}
\usepackage{tikz-cd}
\usepackage{mathtools}
\usepackage{bbm}
\usepackage{upgreek}

\makeatletter
\@ifundefined{th@thmstyleone}{%
  \newtheoremstyle{thmstyleone}{}{}{\itshape}{}{\bfseries}{.}{ }{}%
}{}
\@ifundefined{th@thmstylethree}{%
  \newtheoremstyle{thmstylethree}{}{}{}{}{\bfseries}{.}{ }{}%
}{}
\makeatother
\providecommand{\fnm}[1]{#1}
\providecommand{\sur}[1]{#1}

\providecommand{\orgname}[1]{#1}
\providecommand{\orgaddress}[1]{#1}
\providecommand{\country}[1]{#1}
\providecommand{\city}[1]{#1}
\providecommand{\state}[1]{#1}

\providecommand{\pacs}[1]{\par\noindent\textbf{MSC:} #1}

\def\bord{\partial\Omega}
\def\Tr{\operatorname{Tr}\,}
\def\Ker{\operatorname{Ker}}

\def\B{\mathcal{B}}

\def\Hb{\mathcal{B}(\bord)}
\def\Hbp{\mathcal{B}(\bord)}
\def\Hbm{\mathcal{B}'(\bord)}

\def\Hbc{\boldsymbol{B}(\bord)}
\def\Hbcm{\boldsymbol{B}'(\bord)}
\def\Dc{\boldsymbol{\mathcal D}}
\def\Hbcp{\boldsymbol{B}_\perp(\bord)}
\def\Hbct{\boldsymbol{B}_T(\bord)}
\def\Hbcpm{\boldsymbol{B}'_\perp(\bord)}
\def\Hbctm{\boldsymbol{B}'_T(\bord)}
\def\Hbctan{\boldsymbol{B}_t(\bord)}
\def\Hbcn{\boldsymbol{B}_n(\bord)}
\def\Hconevec{\vec{\boldsymbol H}\vphantom{H}^1(\Omega)}
\def\Cb{\mathcal{B}^2(\bord)}
\def\Cbm{(\mathcal{B}^2(\bord))'}

\def\Trc{\boldsymbol{\operatorname{Tr}}\,}

\def\D{\mathcal{D}}

\def\del{\partial}
\def\eps{\varepsilon}
\def\R{\mathbb{R}}

\def\N{\mathbb{N}}

\def\dx{\,{\rm d}x}

\def\dy{\,{\rm d}y}

\def\dps{\displaystyle}

\renewcommand{\div}{{\rm div}\,}
\newcommand{\divpw}{{\rm div}_{pw}\,}
 
\newcommand{\divnesp}{{\rm div}} 
\newcommand{\rots}{{\rm curl}\,}
\newcommand{\rotv}{{\bf curl}\,}
\newcommand{\rotvnesp}{{\bf curl}}

\newcommand{\divT}{\operatorname{div}_T}
\newcommand{\curlcT}{\operatorname{\bf curl}_T}
\newcommand{\curlT}{\operatorname{curl}_T}

\def\uc{{\boldsymbol{u}}}

\def\vc{{\boldsymbol{v}}}
\def\Vc{{\boldsymbol{V}}}
\def\Kc{{\boldsymbol{K}}}
\def\Lc{{\boldsymbol{L}}}
\def\Hc{{\boldsymbol{H}}}

\def\wc{{\boldsymbol{w}}}
\def\nc{{\boldsymbol{n}}}
\def\fc{{\boldsymbol{f}}}
\def\gc{{\boldsymbol{g}}}
\def\hc{{\boldsymbol{h}}}
\def\lc{{\boldsymbol{l}}}
\def\zc{{\boldsymbol{z}}}
\def\phic{{\boldsymbol{\phi}}}
\def\psic{{\boldsymbol{\psi}}}
\def\trn{\operatorname{Tr}_\nc}
\def\Trn{\operatorname{Tr}_\nc}
\def\trno{\operatorname{Tr}_{\nc_0}}

\def\TrT{\operatorname{{\bf Tr}}_T}
\def\trT{\operatorname{{tr}_T}}

\def\piT{\boldsymbol{\pi}_T}

\def\ucd{{\mathbf{u}}}

\def\Dcd{{\mathbf{D}}}
\def\Hcd{{\mathbf{H}}}
\def\Lcd{{\mathbf{L}}}

\newcommand{\Hdiv}{\Hc({\rm div},\Omega)}
\newcommand{\Hdivz}{\Hc_0({\rm div},\Omega)}
\newcommand{\Hdivs}{\Hcd({\rm div},\Omega)}
\newcommand{\Hdivsz}{\Hcd_0({\rm div},\Omega)}
\newcommand{\Hcurl}{\Hc({\bf curl},\Omega)}
\newcommand{\Hcurlz}{\Hc_0({\bf curl},\Omega)}
\newcommand{\Hcurls}{\Hcd({\rm curl},\Omega)}
\newcommand{\Hcurlsz}{\Hcd_0({\rm curl},\Omega)}

\theoremstyle{thmstyleone}
\newtheorem{theorem}{Theorem}[section]

\newtheorem{proposition}[theorem]{Proposition}

\newtheorem{corollary}[theorem]{Corollary}
\newtheorem{lemma}[theorem]{Lemma}

\newtheorem{c/p}[theorem]{Conjecture/Proposition}

\theoremstyle{thmstylethree}
\newtheorem{definition}[theorem]{Definition}
\newtheorem{remark}[theorem]{Remark}
\newtheorem{example}[theorem]{Example}
\newtheorem{examples}[theorem]{Examples}

\newtheorem{assumption}[theorem]{Assumption}

\newtheorem{open question}[theorem]{Open Question}

\usetikzlibrary{cd}

\makeatletter
\newsavebox{\@brx}
\newcommand{\llangle}[1][]{\savebox{\@brx}{\(\m@th{#1\langle}\)}%
  \mathopen{\copy\@brx\kern-0.5\wd\@brx\usebox{\@brx}}}
\newcommand{\rrangle}[1][]{\savebox{\@brx}{\(\m@th{#1\rangle}\)}%
  \mathclose{\copy\@brx\kern-0.5\wd\@brx\usebox{\@brx}}}
\makeatother

\newcommand{\norm}[1]{\left\lVert#1\right\rVert}
\title{Tangential and normal traces  for extension domains with non-Lipschitz boundaries}

\author[1]{\fnm{Romain} \sur{Cervera}}\email{romain.cervera@centralesupelec.fr}
\author[2]{\fnm{Patrick} \sur{Ciarlet}}\email{patrick.ciarlet@ensta.fr}
\author*[1]{\fnm{Anna} \sur{Rozanova-Pierrat}}\email{anna.rozanova-pierrat@centralesupelec.fr}
\author[3,4]{\fnm{Alexander} \sur{Teplyaev}}\email{a.teplyaev@rochester.edu}

\affil[1]{\orgname{CentraleSup\'elec, Universit\'e Paris-Saclay}, \orgaddress{\city{Gif-sur-Yvette}, \country{France}}}
\affil[2]{\orgname{ENSTA, Institut Polytechnique de Paris}, \orgaddress{\city{Palaiseau}, \country{France}}}
\affil[3]{\orgname{University of Connecticut}, \orgaddress{\city{Storrs}, \state{CT}, \country{USA}}}
\affil[4]{\orgname{University of Rochester}, \orgaddress{\city{Rochester}, \state{NY}, \country{USA}}}
 
\begin{document}
 
\abstract{
\unboldmath
We generalize the classical vector-valued tangential and normal trace theory on Lipschitz domains to the setting of $H^1$-extension domains, which could be non-Lipschitz and include domains with fractal boundaries such as the Koch snowflake. We define and study these operators using the surjectivity of the trace operator on $H^1(\Omega)$ and the Hilbert structure of the associated trace space. The normal trace operator is defined on $\Hdiv$ in $\mathbb{R}^n$. A generalized Stokes formula allows us to introduce the tangential trace operator on $\Hcurl$ and $\Hc^1(\Omega)$ in two and three dimensions. We generalize the results of Buffa, Costabel, and Sheen (2002) for Lipschitz domains: we define two abstract tangential boundary spaces as the images of the trace operators, establish their Hilbert structure, and construct an abstract rotation operator that links them, in place of the usual geometric rotation based on the normal vector. We also show that the full trace space splits orthogonally into a normal and a tangential trace space, thereby extending the classical decomposition of a trace into its normal and tangential components. We obtain a Costabel-type coercivity identity on $\boldsymbol H^1_0(\Omega)$ without using pointwise boundary geometry; on its divergence-free subspace, the weak formulations of $\operatorname{\textbf{curl}}\operatorname{\textbf{curl}}+\textbf{I}$ and $-\Delta+\textbf{I}$ coincide and are well posed. Finally, the trace framework developed for $H^1$-extension domains yields Green formulas for the Hodge--Dirac operator on $(H^1(\Omega))^8$.
}
 
 
\maketitle



\section{Introduction}\label{secIntro}

The study of Maxwell's equations in non-Lipschitz domains, including those with fractal boundaries, presents significant mathematical challenges. Traditional approaches, as detailed in works of Assous, Ciarlet, and Labrunie~\cite{Assous_2018}, Monk~\cite{Monk_2003} and Chicaud~\cite{CHICAUD-2021,CHICAUD-2021-1,CHICAUD-2023}, have primarily focused on Lipschitz domains, leveraging the regularity to effectively define and analyze boundary traces and associated operators. However, extending these methodologies to non-Lipschitz domains necessitates the development of novel mathematical tools.

In this article, we adapt and generalize the mathematical framework detailed in~\cite{Assous_2018} to accommodate non-Lipschitz domains, specifically the $H^1$-extension domains introduced in Section~\ref{sec:ext_domains}. This class goes back to the extension theorems of Calder\'on~\cite{CALDERON-1961} and Stein~\cite{STEIN-1970} for Lipschitz domains and to Jones~\cite{JONES-1981}, who characterized the simply-connected planar $H^1$-extension domains as the $(\eps,\infty)$-domains and thereby brought fractal boundaries such as the von Koch snowflake into the theory, together with the measure-density description of Haj\l asz, Koskela and Tuominen~\cite{HAJLASZ-2008} in $\R^n$ with $n\ge 2$ and the universal extension operators of Rogers~\cite{ROGERS-2006}. This class was recently used as the natural setting for boundary value problems on domains, including both the regular and irregular boundaries, for the well-posedness and shape optimization of acoustic and wave problems by Hinz, Rozanova-Pierrat, and Teplyaev~\cite{HINZ-2021-1,HINZ-2023}, and for transmission and layer-potential problems by Claret, Hinz, Rozanova-Pierrat, and Teplyaev~\cite{Claret_2026,Claret_2024}. Our approach differs from traditional methods that rely heavily on $L^2$ spaces on the boundary. There are two main ways to define a trace operator: one based on the boundary measure and another independent of any boundary measure. The $L^2$-type trace constructions related with a $d$-upper regular boundary Borel measure can be found in~\cite{HINZ-2021-1,HINZ-2021,HINZ-2023,CLARET-2025,ROZANOVA-PIERRAT-2021,ROZANOVA-PIERRAT-2026-1} or, with a more restrictive, $d$-measure  in~\cite{ARFI-2019,ROZANOVA-PIERRAT-2026, DEKKERS-2022, DEKKERS-2022-1}. At the same time, for $d$-set boundary measures, it is possible to work in the associated Besov spaces, as in the study of boundary value problems on Koch-type interfaces by Lancia~\cite{LANCIA-2002} and by Capitanelli and Vivaldi~\cite{CAPITANELLI-2011}. The other way to define the trace operator on the boundary is the measure-free framework initially introduced by Biegert~\cite{BIEGERT-2009}, and recently developed in~\cite{Claret_2026,Claret_2024}. We adopt it here and do not use any boundary measure, using only the intrinsic properties of the trace operator on $H^1(\Omega)$. This allows for a more flexible analysis that covers domains with both regular and irregular (non-rectifiable) boundaries.

A cornerstone of our work is the redefinition of boundary trace operators -- namely, the normal trace, tangential trace, and the tangential component of the trace -- in the context of this class of $H^1$-extension domains. This redefinition relies on the Hilbert space structure of the image of the trace operator of elements of  $H^1(\Omega)$, denoted by $\B(\partial \Omega)$. The transition from Lipschitz to extension domains removes the geometric tools that underlie classical vector trace theory. The classical vector-valued trace theory on smooth and Lipschitz domains, developed by Girault and Raviart~\cite{GIRAULT-1986}, N\'ed\'elec~\cite{Nedelec2001}, and Monk~\cite{Monk_2003}, and refined for $\Hcurl$ by Buffa, Costabel, and Sheen~\cite{BUFFA-2002}, splits a boundary field pointwise in a local frame composed of tangential and normal vectors (Figure~\ref{fig:3D_vectors}).
\begin{figure}[htbp]
\centering
\includegraphics[width=0.7\textwidth]{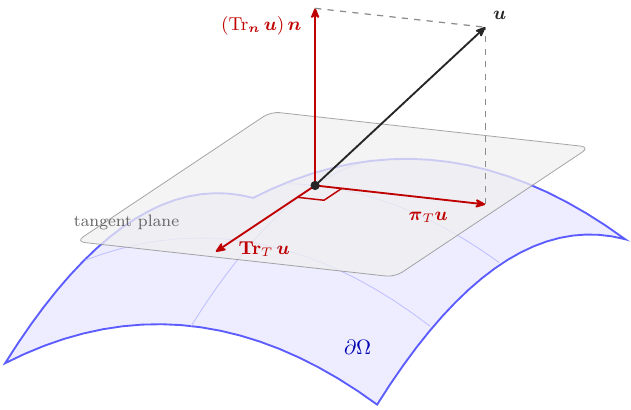}
\caption{Pointwise decomposition of a field $\uc$ at a Lipschitz boundary $\partial\Omega$ into its normal component $(\trn\uc)\,\nc$ and its tangential component $\piT\uc$, with the tangential trace $\TrT\uc$ obtained from $\piT\uc$ by rotation in the tangent plane (see~\cite{BUFFA-2002}). This vector framework does not exist in the fractal case.}
\label{fig:3D_vectors}
\end{figure}
Because these vectors do not exist on fractal boundaries, we define the normal and tangential traces globally as continuous linear functionals via the generalized Green and Stokes formulas. To recover the boundary spaces without a tangent plane, we rely on the orthogonal decomposition of $\Hc^1(\Omega)$ to select unique optimal lifts of the boundary data. This construction provides abstract tangential trace spaces and an abstract rotation operator that replaces the classical pointwise cross product; Figure~\ref{fig:tangential_spaces} collects these spaces and the operators between them.
Beyond the construction of trace operators, the lack of boundary regularity invalidates classical coercivity bounds, such as those established by Costabel~\cite{COSTABEL-1991} for regular domains, and prevents the use of the Gaffney inequality since fields with mixed boundary conditions generally lack global $\Hc^1(\Omega)$-regularity on irregular domains~\cite{Amrouche_1998,Costabel_2000}. We circumvent these analytical barriers by establishing an exact coercivity identity restricted to $\Hc^1_0(\Omega)$, where boundary contributions completely vanish. This variational approach thereby allows us to rigorously establish the weak well-posedness of boundary value problems associated with the operator $\rotv\rotv + \mathbf{I}$ on the divergence-free subspace of $\Hc^1_0(\Omega)$.

While weak trace definitions for vector fields have been previously introduced to study magnetostatic problems on specific fractal geometries, such as the Koch snowflake cylinder (see, for instance, Creo et al.~\cite{Creo_2018}), our work broadens this perspective by providing a systematic and unified trace theory valid for an arbitrary bounded $H^1$-extension domain. Such measure-free trace operators could also serve boundary integral methods on irregular geometries, such as the multi-trace formulations of Claeys and Hiptmair~\cite{claeys-2012} for electromagnetic scattering at composite objects. However, because the surjectivity of the tangential trace operators from $\Hcurl$ remains an open problem for domains with fractal boundaries, the complete resolution of Maxwell's equations is deferred to a forthcoming work. One possible approach to circumvent this, as proposed by Schulz et al.~\cite{Schulz_2022} for Lipschitz domains, is to use the Hodge-Dirac operator, which accommodates both Maxwell and Helmholtz boundary value problems. As a first step in this direction, we conclude this article by showing that the trace framework developed here directly yields the generalized Green formulas for the Hodge-Dirac operator on $(H^1(\Omega))^8$.

The structure of this article is as follows: in Section~\ref{sec:ext_domains}, we introduce $H^1$-extension domains and give some fundamental trace results, laying the groundwork for our study. Section~\ref{sec:normal_trace} defines the normal trace operator in a similar way as in the regular case and details its properties. Section~\ref{sec:tangential_trace} defines the tangential trace operator on $\Hcurl$ and $\Hc^1(\Omega)$, introduces the first abstract boundary space as its image, and establishes its Hilbert structure. Section~\ref{sec:tangential_component}, which contains the core contribution of the article, introduces the tangential component trace operator. It plays the same role as in Buffa et al.~\cite{BUFFA-2002}, but is built by a different, purely variational method, without pointwise boundary geometry. Here, we define the second abstract boundary space, construct the abstract rotation operator linking the two boundary spaces, and establish refined integration-by-parts formulas that are the abstract analogs of their classical counterparts. With both tangential traces in hand, we define in Section~\ref{SecTangOp} the tangential differential operators -- tangential gradient, divergence, and curl. Each operator is introduced in two forms: a general version relying on the broader trace definition, and a restricted version within the $\Hc^1(\Omega)$-framework. Finally, Section~\ref{SecHodgeDirac} applies this theory to establish the generalized Green formulas for the Hodge-Dirac operator.

\section{\texorpdfstring{$H^1$-extension domains}{H1-extension domains}}\label{sec:ext_domains}

\subsection{Main definitions}

We assume that $\Omega$ is a bounded domain (\textit{i.e.} a bounded open connected set) in $\R^n$, $n\ge2$, with boundary $\bord$, and is an \textit{$H^1$-extension domain}~\cite{JONES-1981,ROGERS-2006,HAJLASZ-2008}, $i.e.$ there exists a bounded linear extension operator $E:H^1(\Omega)\to H^1(\R^n)$:
\begin{enumerate}
\item[(i)] $\forall u\in H^1(\Omega),\; \left.Eu\right|_\Omega=u$,
\item[(ii)] $\forall u\in H^1(\Omega),\; \norm{Eu}_{H^1(\R^n)}\le C\norm{u}_{H^1(\Omega)}$,
\end{enumerate}
where $C>0$ depends only on $\Omega$.

We recall the notion of $n$-set~\cite{JONSSON-1984}, which will be used for both open and closed subsets.

\begin{definition}\label{DefNset}
A Borel set $A \subset \R^n$ is called an \emph{$n$-set} if it satisfies the measure density condition:
\begin{equation}\label{EqNset}
    \exists c > 0 \quad \forall x \in A, \;
    \forall r \in {]0,1]}, \quad
    \lambda^{(n)}(B_r(x) \cap A) \geq c\, r^n,
\end{equation}
where $\lambda^{(n)}$ is the $n$-dimensional Lebesgue measure and $B_r(x)$ is the open ball of $\R^n$ centered at $x$ of radius $r > 0$.
\end{definition}

By~\cite{HAJLASZ-2008}, the $H^1$-extension property of $\Omega$ implies that $\Omega$ is an $n$-set. Hence $\Omega$ satisfies the uniform measure-density condition~\eqref{EqNset}.
This ensures in particular that $H^1$-extension domains are `non-collapsing' close to the boundary. If $d_H$ is the Hausdorff dimension of $\del \Omega$, then  $d_H > n - 2 \geq 0$. In particular, domains with cusps or fractal tree~\cite{ACHDOU-2014}--type geometries fail to be $H^1$-extension domains. On the other hand, domains whose boundary is given by a composition of $d$-sets with parameters $d \in (n-1, n)$ do satisfy this extension property~\cite{WALLIN-1991}.

The typical examples of $H^1$-extension domains are $(\eps,\infty)$-domains (or uniform domains if bounded)~\cite{JONES-1981}. In $\R^2$, by~\cite{JONES-1981}, a simply connected domain is  an $H^1$-extension domain if and only if it is a uniform domain. In $\R^3$, there is a {H}ajłasz-{K}oskela-{T}uominen characterization of the Sobolev extension domains. The equivalence condition~\cite{HAJLASZ-2008} is given for $n$-sets on which the canonical norm of $H^1(\Omega)$ is equivalent to the $C_2^1$-norm of the space of fractional sharp maximal functions. The particular examples of the $H^1$-extension domains are Lipschitz domains, von Koch snowflake-type domains and NTA domains~\cite{NYSTROM-1996}.

If $\Omega$ is a bounded domain, then its boundary $\partial\Omega$ has positive capacity. Here, capacity is understood with respect to $H^1(\R^n)$ (see, for example,~\cite[Section~2.1]{Fukushima_2011}). Accordingly, the notions of ``quasi-everywhere'' (q.e.) and ``quasi-continuous'' are understood in this sense. By~\cite[Theorem~2.1.3]{Fukushima_2011}, every $u\in H^1(\R^n)$ admits a quasi-continuous representative $\tilde u$ such that $u=\tilde u$ a.e. In the framework of bounded $H^1$-extension domains, we define the trace operator on the boundary $\del \Omega$ in the same way as in~\cite[Section 6]{BIEGERT-2009} (see also \cite{HINZ-2021-1,Claret_2026,ROZANOVA-PIERRAT-2026-1}).

\begin{definition}[Trace operator]\label{DefTraceInt}
Let $\Omega$ be an $H^1$-extension domain of $\R^n$, $n\ge2$. The trace operator $\Tr:H^1(\Omega)\to \Hb$ is defined quasi-everywhere (q.e.) by
\begin{equation}\label{E:pointwiseredef}
\Tr u(x)=\dps\lim_{r\to 0^+}\frac{1}{\lambda^{(n)}(\Omega\cap B_r(x))}\int_{\Omega\cap B_r(x)}{u(y)\,\dy},\quad x\in \del \Omega,
\end{equation}
where $\lambda^{(n)}$ denotes the Lebesgue measure on $\R^n$ and $B_r(x)$ is the open ball of $\R^n$ of center $x$ and radius $r$.
Here  $\mathcal{B}(\partial\Omega)$ is the vector space defined as the image of the trace operator consisting of all q.e. equivalence classes of pointwise restrictions $\tilde u|_{\partial\Omega}$ of quasi-continuous representatives $\tilde u$ of classes $u \in H^1(\R^n)$.
\end{definition}

Note that  $\Tr u=\widetilde{\mathrm E_\Omega u}|_{\partial\Omega}$, $u\in H^1(\Omega)$, for any bounded linear extension operator $\mathrm{E}_\Omega:H^1(\Omega)\to H^1(\R^n)$, see \cite[Corollary 6.3]{BIEGERT-2009}. Hence, the definition does not depend on the way we extend an element of $H^1(\Omega)$ to $H^1(\R^n)$~\cite[Theorem 6.1]{BIEGERT-2009}. A similar construction was used for domains with $d$-set boundaries in~\cite[Theorem 1]{WALLIN-1991}. Here, there is no assumption on the boundary measure. The boundaries with variable Hausdorff dimension, as unions of different $d$-sets, could be also cited as examples of boundaries covered by this construction.

For non-Lipschitz boundaries, such as von Koch fractals, the normal vector does not exist for all $x\in \del \Omega$, preventing us from considering the usual (strong) normal derivative. We define the weak normal derivative as a linear and continuous functional on the image of the trace operator (see~\cite{LANCIA-2002} for a particular case), using the Green formula for elements in the space
\begin{equation}\label{H1D}
H^1_\Delta(\Omega) := \bigl\{v \in H^1(\Omega) \;\big|\; \Delta v \in L^2(\Omega)\bigr\}.
\end{equation}
Equipped with the norm
\begin{equation}\label{EqNormH1Delta}
\|v\|^2_{H^1_\Delta(\Omega)} := \|v\|^2_{H^1(\Omega)} + \|\Delta v\|^2_{L^2(\Omega)},
\end{equation}
this space is a Hilbert space (see~\cite[p.~25]{Grisvard_1992}).

Let us denote by $\D(\Omega)$ the set of infinitely differentiable functions with compact support in $\Omega$ and $\D'(\Omega)$ the space of distributions on $\Omega$.
We recall now the trace operator properties~\cite{Claret_2026,HINZ-2023}.

\begin{theorem}[Trace theorem]\label{ThTrisom}
Let $\Omega$ be an $H^1$-extension domain of $\R^n$, $n\ge2$.
Then
\begin{enumerate}
    \item\label{Trisom1} the image of the trace operator $\Hb:=\Tr(H^1(\Omega))$ endowed with the norm
    \begin{equation}\label{normTri}
    \left\|f\right\|_{\Hb}:=\min\{ \left\|v\right\|_{H^1(\Omega)}\mid f=\mathrm{Tr}\:v\},
    \end{equation}
    is a Hilbert space. 
    
    \item\label{Trisom2} let $H^1_0(\Omega)=\overline{\D(\Omega)}^{H^1}$. Then it holds that $\operatorname{Ker}\Tr=H^1_0(\Omega)$.
    
    \item\label{Trisom3} the trace operator $\operatorname{Tr}: H^1(\Omega) \to \Hb$ is a partial isometry with operator norm equal to $1$.
    
    \item\label{Trisom4} $\Tr|_{V_1(\Omega)}: V_1(\Omega) \to \Hb$ defines an isometry, where 
    \begin{equation}\label{V1}
    V_1(\Omega)=\big\{v\in H^1(\Omega)\;\big|\; (-\Delta+1)v=0 \mbox{ in }\D'(\Omega)\big\}
    \end{equation}
    is the space of  $1$-harmonic functions endowed with the standard norm $\|\cdot\|_{H^1}$:
    \begin{equation}\label{EqV1andOthers}
    	V_1(\Omega)=\operatorname{Ker}(\Tr)^\perp =\big(H^1_0(\Omega)\big)^\perp.
    	\end{equation}
    	Therefore, 
    	\begin{equation}\label{EqBnorm}		
    	\forall u\in V_1(\Omega),\quad \|u\|_{H^1(\Omega)}=\|\!\Tr u\|_{\Hbp}.
    \end{equation}
\end{enumerate}
\end{theorem}

\begin{remark}
    The trace operator in Theorem~\ref{ThTrisom} does not depend on a boundary measure on $\del \Omega$.
\end{remark}

As a direct consequence of Theorem~\ref{ThTrisom}, the trace space $\Hb$ is a Hilbert space endowed with the inner product
\begin{equation}\label{DefHbInnerProd}
(f, g)_{\Hb} := (w_f, w_g)_{H^1(\Omega)},
\end{equation}
where $w_f, w_g \in V_1(\Omega)$ are the unique optimal ($1$-harmonic) lifts of $f$ and $g$.

We now define the weak normal derivative.

\begin{definition}[Weak normal derivative]\label{DefNormDerInt}
Let $\Omega$ be a bounded $H^1$-extension domain of $\R^n$, $n\ge2$. The weak normal derivative of $u\in H^1_\Delta(\Omega)$ is the continuous linear functional on $\Hb$, $\psi\in\Hbm$, such that for all $v\in H^1(\Omega)$, it holds
\begin{equation}\label{EqGreenInt}
\langle\psi,\Tr v\rangle_{\Hbm,\,\Hb}= \int_\Omega{(\Delta u)v\dx}+\int_\Omega{\nabla u\cdot\nabla v\dx},
\end{equation}
and we denote $\dps\frac{\del u}{\del \nc}=\psi$ or, by   analogy with the regular case, $(\nabla u) \cdot  \nc=\psi$.
\end{definition}

\begin{remark}\label{RmkContNormDer}
The right-hand side of~\eqref{EqGreenInt} is a bounded linear functional of $v\in H^1(\Omega)$, and it vanishes for $v\in H^1_0(\Omega)$ by the definition of the distributional Laplacian, hence it depends only on $\Tr v$. Since $\Tr:H^1(\Omega)\to\Hb$ is continuous and surjective by Theorem~\ref{ThTrisom}, this defines $\psi$ as a continuous linear functional on $\Hb$. This is the standard construction of the weak normal derivative, which we formulate here on $H^1$-extension domains (see~\cite{Doob_1962,Constantinescu_1963,Maeda_1964,Maeda_1980}, but also more recent works~\cite{LANCIA-2002,HINZ-2021-1,HINZ-2021,HINZ-2023,CLARET-2025,ROZANOVA-PIERRAT-2021,ROZANOVA-PIERRAT-2026-1}).
As the functional depends only on $\Tr v$, one may evaluate~\eqref{EqGreenInt} on any $v\in H^1(\Omega)$ with prescribed trace, and a canonical choice is the $1$-harmonic extension, which maps $w\in\Hb$ to the unique weak solution $v\in H^1(\Omega)$ of
\begin{equation*}
    (-\Delta + 1)v = 0 \quad \text{in } \D'(\Omega), \qquad \Tr\, v = w,
\end{equation*}
that is, $v\in V_1(\Omega)$ as defined in~\eqref{V1}. This operator realizes the infimum in the trace norm~\eqref{normTri} and defines an isometric right inverse of the trace (Theorem~\ref{ThTrisom}, item~\ref{Trisom4}). For its construction and properties in the non-Lipschitz setting, we refer to~\cite[Corollary~1]{HINZ-2021-1} and~\cite[Proposition~2.2]{HINZ-2023}.
\end{remark}

In what follows, we denote by $\D(\overline\Omega)$ the set of restrictions to $\Omega$ of functions in $\D(\R^n)$:
\begin{equation}
    \D(\overline\Omega) = \{ u|_{\Omega} \mid u \in \D(\R^n) \}.
\end{equation}
In an $H^1$-extension domain $\Omega$ of $\R^n$, the set $\D(\overline\Omega)$ is dense in $H^1(\Omega)$. This density follows from the $H^1$-extension property of $\Omega$: one can extend elements to a larger regular domain, apply the standard density result there, and then restrict back to $\Omega$.

\subsection{\texorpdfstring{$H^1_0(\Omega)$-extension property}{H10(Omega)-extension property}}
We provide here a proof of the $H^1_0(\Omega)$-extension property (Proposition~\ref{PropExtH10}) in our framework. While it was established in ~\cite[Theorem~5.29]{ADAMS-2003} for domains with the segment property, a different argument is required here since the segment property is purely geometric. Actually, the key is to have an extension property, and hence, to ensure that $\Omega$ and its complement are $n$-sets of $\R^n$ (or, equivalently, satisfy the measure density condition~\cite{HAJLASZ-2008}).
Let us recall the notion of two-sided admissible domains from~\cite{Claret_2026}.

\begin{definition}\label{DefTwoSidedH1Ext}
Let $\Omega$ be a domain of $\R^n$, $n\ge2$, and set $\Omega^e:=\R^n\setminus\overline\Omega$. We say that $\Omega$ is a \textit{two-sided $H^1$-extension domain} if $\Omega$ and $\Omega^e$ are $H^1$-extension domains, and if $\bord=\partial\Omega^e$.
\end{definition}

Let us provide now a characterization of elements in $H^1_0(\Omega)$ in terms of their zero extension to $\R^n$. For any element $u \in H^1(\R^n)$, we recall that $\tilde u$ denotes its quasi-continuous representative, i.e. a quasi-continuous element of $H^1(\R^n)$ such that $u(x) = \tilde u(x)$ a.e. We use the following characterization of $H^1_0(\Omega)$ given in~\cite[Example~2.3.1]{Fukushima_2011}:
\begin{equation}\label{H1CapH10}
    H^1_0(\Omega) = \{ u_{|\Omega} \; : \; u \in H^1(\R^n), \; \tilde u = 0 \text{ q.e. on } \R^n \setminus \Omega \}.
\end{equation}
The equality $\operatorname{Ker}\Tr = H^1_0(\Omega)$ of Theorem~\ref{ThTrisom}, item~\ref{Trisom2}, rests on this characterization, see~\cite[Theorem~2.9(i) and Appendix~A, p.~38]{Claret_2026}.

Let $\Omega^e := \R^n \setminus \overline\Omega$. Since $\Omega$ is two-sided in the sense of Definition~\ref{DefTwoSidedH1Ext}, the set $\Omega^e$ is an $H^1$-extension domain with $\partial\Omega^e = \bord$. We write $\operatorname{Tr}_e : H^1(\Omega^e) \to \Hb$ for the corresponding trace operator, which takes its values in the same space $\Hb$ as $\Tr$ because $\partial\Omega^e = \bord$. By Theorem~\ref{ThTrisom}, applied to $\Omega$ and to $\Omega^e$, the restrictions $\Tr|_{V_1(\Omega)}$ and $\operatorname{Tr}_e|_{V_1(\Omega^e)}$ to the spaces of $1$-harmonic functions defined in~\eqref{V1} are bijections onto $\Hb$, isometric for the norm~\eqref{normTri} associated with the corresponding side. Only the injectivity is used below, and it comes, for a given $f \in \Hb$, from the uniqueness of the weak solution of the interior and exterior problems
\begin{equation}
    \begin{cases}
        -\Delta w + w = 0 & \text{in } \Omega,\\
        \Tr w = f,
    \end{cases}
    \qquad\text{respectively}\qquad
    \begin{cases}
        -\Delta w + w = 0 & \text{in } \Omega^e,\\
        \operatorname{Tr}_e w = f.
    \end{cases}
\end{equation}
We now prove the extension property.

\begin{proposition}\label{PropExtH10}
Let $\Omega$ be a bounded two-sided $H^1$-extension domain of $\R^n$, $n \ge 2$. Let $u \in H^1(\Omega)$. We denote by $\mathcal{Z}u$ the extension by zero of $u$ outside of $\Omega$. If $\mathcal{Z}u \in H^1(\R^n)$, then $u \in H^1_0(\Omega)$.
\end{proposition}

\begin{proof}
We argue by contradiction and assume that $u \notin H^1_0(\Omega)$. By Theorem~\ref{ThTrisom}, item~\ref{Trisom2}, this means that $\Tr u \neq 0$ in $\Hb$.
Since $\mathcal{Z}u \in H^1(\R^n)$, its restrictions to $\Omega$ and to $\Omega^e$ give $\mathcal{Z}u \in H^1(\R^n \setminus \bord)$, and we denote by $z := (\mathcal{Z}u)_{|\Omega^e} \in H^1(\Omega^e)$ its exterior part. The element $\mathcal{Z}u$ is a global $H^1(\R^n)$-extension of $u$ and of $z$ at the same time, and by~\cite[Theorem~6.1 and Corollary~6.3]{BIEGERT-2009} the trace does not depend on the chosen extension, so, using $\partial\Omega^e = \bord$,
\begin{equation}\label{EqTrExtInt}
    \operatorname{Tr}_e z = \widetilde{\mathcal{Z}u}_{|\bord} = \Tr u \neq 0 .
\end{equation}
We now decompose $z$ according to~\eqref{EqV1andOthers} applied to $\Omega^e$, as $z = z_0 + z_1$ with $z_0 \in H^1_0(\Omega^e)$ and $z_1 \in V_1(\Omega^e)$. Since $H^1_0(\Omega^e) = \operatorname{Ker}\operatorname{Tr}_e$, we have $\operatorname{Tr}_e z_0 = 0$, hence $\operatorname{Tr}_e z_1 = \operatorname{Tr}_e z \neq 0$ by~\eqref{EqTrExtInt}. As $\operatorname{Tr}_e$ is injective on $V_1(\Omega^e)$, this forces $z_1 \neq 0$, and therefore $z \neq 0$ in $H^1(\Omega^e)$ by the uniqueness of the decomposition.
This contradicts the definition of $\mathcal{Z}u$, which vanishes a.e. on $\Omega^e \subset \R^n \setminus \Omega$, hence giving $z = 0$. Consequently $u \in H^1_0(\Omega)$.
\end{proof}

\section{\texorpdfstring{Normal trace operator on $\Hdiv$}{Normal trace operator on H(div,Omega)}}\label{sec:normal_trace}

We adopt the same notations as in~\cite{Ciarlet_2024}.
We use bold symbols for vector-valued functions and vector-valued function spaces.
 In what follows we work in the vector-valued space $\Hc^1(\Omega)$:
$\uc\in \Hc^1(\Omega)$ means that $\uc=(u_1,\ldots,u_n)$ and for all $j=1,\ldots,n$ $u_j\in H^1(\Omega)$, $i.e.$ the notation $\uc\in \Hc^1(\Omega)$ needs to be read by components. Therefore, $\Trc \uc$ is also understood by components, and we note 
$\Trc: \Hc^1(\Omega) \to \Hbc$. We consider  the spaces
\begin{align}\label{EqHdiv}
    \Hdiv = \{ \wc \in \Lc^2(\Omega)  \mid \div\wc \in L^2(\Omega) \}
\end{align}
and, for $n=3$,
\begin{align}\label{EqHcurl}
    \Hcurl = \{ \wc \in \Lc^2(\Omega)  \mid \rotv \wc \in \Lc^2(\Omega) \},
\end{align}
where divergence and curl are understood in the sense of distributions. Endowed with the norms
\begin{align*}
	&\|\wc\|_{\Hdiv} = \left(\|\wc\|_{\Lc^2(\Omega)}^2 + \|\div \wc\|_{L^2(\Omega)}^2\right)^\frac{1}{2},\\
	&\|\wc\|_{\Hcurl} = \left(\|\wc\|_{\Lc^2(\Omega)}^2 + \|\rotv \wc\|_{\Lc^2(\Omega)}^2\right)^\frac{1}{2},
\end{align*}
both of these spaces are Hilbert spaces, this property does not depend on the geometry of the boundary.
In particular, from~\cite{Ciarlet_2024} we know that on $\Hc^1(\Omega)$ the $\Hdiv$-norm and $\Hcurl$-norm are weaker than the canonical $\Hc^1(\Omega)$-norm:
\begin{align}
	\forall \wc\in \Hc^1(\Omega), \quad \|\wc\|_{\Hdiv}\le \sqrt{n} \,\|\wc\|_{\Hc^1(\Omega)},\label{IneqCdiv}\\
    \forall \wc\in \Hc^1(\Omega), \quad \|\wc\|_{\Hcurl}\le \sqrt{2} \, \|\wc\|_{\Hc^1(\Omega)},\label{IneqCcurl}
\end{align}
but not conversely.

To define the normal trace operator in non-Lipschitz domains, we directly generalize~\cite[Theorem~2.5, p.~27]{GIRAULT-1986}, based on the integration by parts formula. It suffices to ensure that the normal trace operator has all the required properties. To begin with, we define the weak version of $\uc \cdot  \nc$ in the same way as in Definition~\ref{DefNormDerInt}.

\begin{definition}[Generalized Green formula]\label{DefNormDotProd}
Let $\Omega$ be a bounded $H^1$-extension domain of $\R^n$, $n\ge2$. 
The weak normal dot product of $\uc\in \Hdiv$ is the continuous linear functional on $\Hb$, $\phi\in\Hbm$, such that for all $v\in H^1(\Omega)$, it holds
\begin{equation}\label{EqGreenIntBis}
\langle\phi,\Tr v\rangle_{\Hbm,\,\Hb}= \int_\Omega{(\div \uc)v\dx}+\int_\Omega{\uc\cdot\nabla v\dx},
\end{equation}
and we denote, by   analogy with the regular case, $\uc \cdot  \nc := \phi$.
\end{definition}

\begin{remark}
    As in Definition~\ref{DefNormDerInt}, equation~\eqref{EqGreenIntBis} defines a continuous linear functional $L_n$ on $\Hb$: for all $\uc \in \Hdiv$,
    \begin{equation*}
        L_n: w \in \Hb \mapsto L_n\, w = \int_{\Omega} (\div\uc) Ew \dx + \int_{\Omega} \uc \cdot \nabla(Ew) \dx.
    \end{equation*}
    With a similar proof, we obtain
    \begin{align*}
    |L_n\,w| \leq C_E \left(\|\div \uc\|_{L^2(\Omega)} + \|\uc\|_{\Lc^2(\Omega)}\right)\|w\|_{\Hb}.
    \end{align*}
\end{remark}

We now use the generalized Green formula to define the generalized normal trace operator. For simplicity, we will now refer to it simply as the normal trace operator.

\begin{theorem}[Normal trace]\label{ThNormalTrace}
Let $\Omega$ be a bounded $H^1$-extension domain of $\R^n$, $n\ge2$. Then the (generalized) normal trace operator, defined by
    \begin{align*}\label{EqNormalTrace}
        \trn \colon \Hdiv &\to \Hbm\\
    \uc & \mapsto \trn \uc = \uc \cdot  \nc,
    \end{align*}
    is a linear, continuous and surjective mapping,
    given by Definition~\ref{DefNormDotProd}:
    \begin{equation}\label{EqItByPn}
    \forall v \in H^1(\Omega), \quad \left\langle \uc \cdot  \nc,\Tr v\right\rangle_{\Hbm,\,\Hb}= \int_{\Omega} (\div\uc)v\dx + \int_{\Omega} \uc \cdot \nabla v \dx.
\end{equation}
\end{theorem}

\begin{proof}
   Let us first prove the continuity.
    
    Let $\uc \in \Hdiv$ and $v \in H^1(\Omega)$. By definition of the normal trace, we have
    \begin{align}
        \left\langle \uc \cdot  \nc,\Tr v \right\rangle_{\Hbm,\,\Hb} = \int_{\Omega} (\div \uc) v \dx + \int_{\Omega} \uc \cdot (\nabla v)\dx .
    \end{align}
    Therefore, using the Cauchy-Schwarz inequality in $\mathbb{R}^2$, we find
    \begin{align*}
        |\left\langle \uc \cdot  \nc,\Tr v \right\rangle_{\Hbm,\,\Hb}| &\leq \|\div \uc\|_{L^2(\Omega)} \|v\|_{L^2(\Omega)} + \|\uc\|_{\Lc^2(\Omega)} \|\nabla v\|_{\Lc^2(\Omega)}  \\
        &\leq \|\uc\|_{\Hdiv} \|v\|_{H^1(\Omega)}. 
    \end{align*}
    We then have for all $v\in H^1(\Omega)$ such that $w = \Tr v$:
    \begin{equation*}
        |\left\langle \uc \cdot  \nc,w \right\rangle_{\Hbm,\,\Hb}| \leq \|\uc\|_{\Hdiv} \|v\|_{H^1(\Omega)}.
    \end{equation*}
    Hence, the previous inequality implies by~\eqref{normTri} that
    \begin{align*}
        \forall \uc \in \Hdiv,\, \forall w \in \Hb, \quad |\left\langle \uc \cdot  \nc,w \right\rangle_{\Hbm,\,\Hb}| \leq \|\uc\|_{\Hdiv} \|w\|_{\Hb}.   
    \end{align*}
    Therefore, the continuity of $\trn$ follows:
    \begin{equation}
    \|\trn \uc\|_{\Hbm} \leq \|\uc\|_{\Hdiv}.
    \end{equation}

    Next, we prove the surjectivity. Let $\mu$ $\in$ $\Hbm$. We want to find $\uc \in \Hdiv$ such that
    \begin{equation*}
        \uc \cdot  \nc = \mu\in \Hbm.
    \end{equation*}
    Consider the following problem:
    $$\begin{cases}\label{EqDp1}
    \text{Find $\phi \in H^1(\Omega)$ such that}\\
    -\Delta \phi + \phi = 0 & \text{ weakly on }  \Omega,\\
    \frac{\partial \phi}{\partial \nc} = \mu & \text{ in } \Hbm.
    \end{cases}$$
    Let $\mu\in\B'(\partial\Omega)$. Consider the bounded linear functional on $H^1(\Omega)$ given by
\[
F(v):=\langle\mu,\Tr v\rangle_{\B'(\partial\Omega),\B(\partial\Omega)}.
\]
By the Riesz representation theorem applied to the $H^1(\Omega)$ inner product, there exists a unique $\varphi\in H^1(\Omega)$ such that
\[
\int_\Omega \nabla\varphi\cdot\nabla v\,dx
+\int_\Omega \varphi v\,dx
=\langle\mu,\Tr v\rangle_{\B'(\partial\Omega),\B(\partial\Omega)}
\qquad\text{for all }v\in H^1(\Omega).
\]
Taking $v\in\mathcal D(\Omega)$ shows that $-\Delta\varphi+\varphi=0$ in $\mathcal D'(\Omega)$; hence $\Delta\varphi=\varphi\in L^2(\Omega)$. Therefore $\boldsymbol u:=\nabla\varphi$ belongs to $H(\operatorname{div},\Omega)$. Finally, Definition 3.1 gives, for every $v\in H^1(\Omega)$,
\[
\langle\Trn\boldsymbol u,\Tr v\rangle
=\int_\Omega (\Delta\varphi)v\,dx
 +\int_\Omega\nabla\varphi\cdot\nabla v\,dx
=\langle\mu,\Tr v\rangle.
\]
Since $\Tr:H^1(\Omega)\to\B(\partial\Omega)$ is surjective, $\Trn\boldsymbol u=\mu$. Thus $\Trn$ is surjective.
\end{proof}

In what follows, we will sometimes need to consider the restriction of the normal trace to a subset of the boundary $\partial\Omega$. Since the normal trace is defined as a functional rather than as a pointwise function, its restriction must also be understood in a functional-analytic sense.

\begin{definition}[Restriction of the normal trace]\label{DefResNormalTrace}
Let $\Omega$ be a bounded $H^1$-extension domain of $\R^n$, $n \ge 2$. Let $\uc \in \Hdiv$ and let $\Gamma$ be a relatively open subset of the boundary $\partial\Omega$.\\
We define the space of trace functions vanishing outside $\Gamma$ as the subspace
\begin{equation}\label{EqDefBGamma}
    \B_\Gamma(\partial\Omega) := \bigl\{ \varphi \in \B(\partial\Omega) \; : \; \varphi = 0 \ \text{q.e. on } \partial\Omega \setminus \Gamma \bigr\}.
\end{equation}
The restriction of the normal trace of $\uc$ to $\Gamma$, denoted $(\trn\uc)|_\Gamma$, is the functional belonging to the dual space $\B'_\Gamma(\partial\Omega)$ whose action is defined by
\begin{equation}\label{EqActionResTrace}
    \forall \phi \in \B_\Gamma(\partial\Omega), \quad \left\langle (\trn\uc)|_\Gamma, \phi \right\rangle_{\B'_\Gamma(\partial\Omega), \B_\Gamma(\partial\Omega)} := \left\langle \trn \uc, \phi \right\rangle_{\Hbm,\, \Hb}.
\end{equation}
Equivalently, for $\phi = \Tr v \in \B_\Gamma(\partial\Omega)$, where $v \in H^1(\Omega)$ is such that $\Tr v = 0$ q.e. on $\partial\Omega \setminus \Gamma$, the action is given via the generalized Green identity:
\begin{equation}\label{EqActionResTraceInt}
    \left\langle (\trn\uc)|_\Gamma, \phi \right\rangle_{\B'_\Gamma, \B_\Gamma} = \int_{\Omega} (\div\uc)v \dx + \int_{\Omega} \uc \cdot \nabla v \dx.
\end{equation}
\end{definition}

\begin{remark}
Elements of $\B(\partial\Omega)$ are defined only up to sets of zero $H^1$-capacity, so the condition in~\eqref{EqDefBGamma} is stated quasi-everywhere and does not depend on the choice of representative. The subspace $\B_\Gamma(\partial\Omega)$ is closed in $\B(\partial\Omega)$, since convergence in $\B(\partial\Omega)$ gives, along a subsequence, quasi-everywhere convergence of the corresponding quasi-continuous representatives~\cite[Theorem~2.1.4]{Fukushima_2011}, hence it is itself a Hilbert space.
\end{remark}

The definition of the normal trace operator clearly depends on the domain $\Omega$ on which it is defined. While working within a single, fixed domain, we will use the default notation $\trn$. However, in proofs involving multiple domains or subdomains (e.g., $\Omega$, $\Omega_0$, $\Omega'$), it will be crucial to distinguish between them. To emphasize this dependence without making the notation too cumbersome, we will subscript the normal vector symbol to indicate the associated domain. For example, the normal trace operator for a domain $\Omega_0$ will be denoted by $\operatorname{Tr}_{\nc_0}$, and the operator for a domain $\Omega'$ will be denoted by $\operatorname{Tr}_{\nc'}$.

Using the definition of the normal trace, we can now establish a fundamental property for vector fields defined on nested domains. As a matter of fact, in many applications, a global field is constructed by gluing together separate fields defined on such domains. For the resulting field to belong globally to $\Hc(\operatorname{div})$, its distributional divergence must not contain any singularities on the common boundary. The following proposition formalizes this idea, showing that this is equivalent to a ``zero jump'' condition on the normal traces, which acts as a transmission condition across the interface. It is worth noting that an alternative approach exists, for instance in~\cite{Claret_2026}, where the sign change reflecting the opposite orientation of the normal vectors is introduced directly into a separate definition for the exterior normal derivative. Our approach differs: we rely solely on a single, unified definition of the normal trace for any domain. As a consequence, the resulting opposite-sign relationship on an interface becomes a derived property rather than a definitional convention.

\begin{proposition}\label{PropOppNormTr}
Let $\Omega$ and $\Omega'$ be two bounded $H^1$-extension domains of $\R^n$ such that $\overline{\Omega} \subset \Omega'$. Define $\Omega_0 := \Omega' \setminus \overline{\Omega}$ and assume that $\Omega_0$ is an $H^1$-extension domain with $\bord = \partial(\R^n \setminus \overline\Omega)$, which holds in particular if $\Omega$ is a two-sided $H^1$-extension domain in the sense of Definition~\ref{DefTwoSidedH1Ext} (see Figure~\ref{fig:domain_decomp} for an illustration).\\
Let $\uc \in \Hdiv$ and $\uc_0 \in \Hc(\div,\Omega_0)$. We define a piecewise vector field $\uc' \in \Lc^2(\Omega')$ by
$$ \uc'(x) := \begin{cases} \uc(x) & \text{if } x \in \Omega \\ \uc_0(x) & \text{if } x \in \Omega_0 \end{cases}, $$
and we denote by $(\trno\uc_0)_{|\partial\Omega}$ the restriction of the normal trace functional from $\Omega_0$ to the interface $\partial\Omega$ as given in Definition~\ref{DefResNormalTrace}. 
Then the field $\uc'$ belongs to the global space $\Hc(\div,\Omega')$ if and only if the sum of the normal traces on the interface $\bord$ is zero, i.e.,
\begin{equation}\label{EqJumpCondition}
    \trn\uc + (\trno\uc_0)_{|\bord} = 0 \quad \text{in } \B'(\bord).
\end{equation}
\end{proposition}

\begin{figure}[htbp]
\centering
\includegraphics[width=0.4\textwidth]{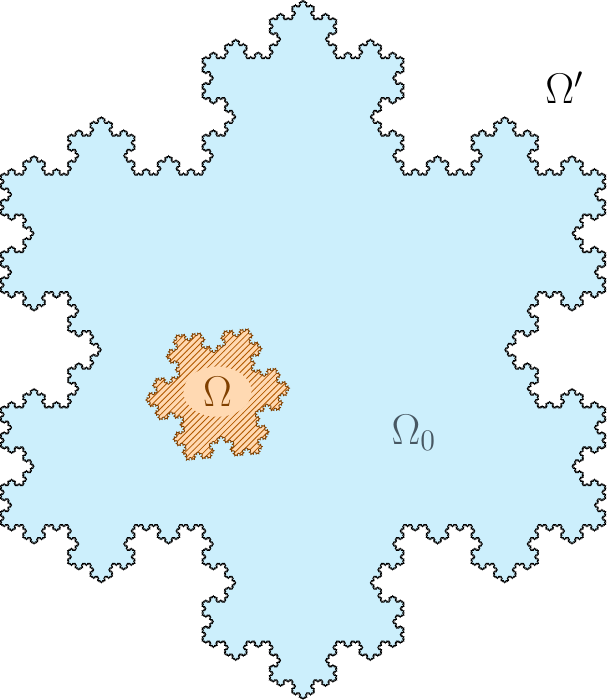}
\caption{The inner domain $\Omega$, the outer domain $\Omega'$, and the region $\Omega_0 = \Omega' \setminus \overline\Omega$ between them}
\label{fig:domain_decomp}
\end{figure}

\begin{proof}
By construction, since $\uc \in \Lc^2(\Omega)$ and $\uc_0 \in \Lc^2(\Omega_0)$, the piecewise field $\uc'$ belongs to $\Lc^2(\Omega')$. For $\uc'$ to be in $\Hc(\div,\Omega')$, we must show that its distributional divergence $\div \uc'$ is an element in $L^2(\Omega')$.\\
Let $\phi \in \D(\Omega')$ be an arbitrary test function. The distributional divergence is defined by its action on $\phi$:
$$ \langle \div \uc', \phi \rangle = - \int_{\Omega'} \uc' \cdot \nabla\phi \dx. $$
We split the integral over the subdomains $\Omega$ and $\Omega_0$:
$$ \langle \div \uc', \phi \rangle = - \int_{\Omega} \uc \cdot \nabla\phi \dx - \int_{\Omega_0} \uc_0 \cdot \nabla\phi \dx.$$
For each integral, we apply the definition of the respective normal trace operator (Theorem~\ref{ThNormalTrace}), which is possible since by hypothesis $\uc \in \Hdiv$ and $\uc_0 \in \Hc(\div,\Omega_0)$:
\begin{align*}
    - \int_{\Omega} \uc \cdot \nabla\phi \dx &= \int_{\Omega} (\div \uc) \phi \dx - \langle \trn\uc, \operatorname{Tr}_{\bord}(\phi) \rangle_{\Hbm,\, \Hb} \\
    - \int_{\Omega_0} \uc_0 \cdot \nabla\phi \dx &= \int_{\Omega_0} (\div \uc_0) \phi \dx - \langle \trno\uc_0, \operatorname{Tr}_{\partial\Omega_0}(\phi) \rangle_{\B'(\partial\Omega_0),\, \B(\partial\Omega_0)}
\end{align*}
Summing these two identities, we obtain the action of the distributional divergence:
\begin{multline}
    \langle \div \uc', \phi \rangle = \int_{\Omega'} (\divpw \uc') \phi \dx \\ - \langle \trn\uc, \operatorname{Tr}_{\bord}(\phi) \rangle_{\Hbm,\, \Hb} - \langle \trno\uc_0, \operatorname{Tr}_{\partial\Omega_0}(\phi) \rangle_{\B'(\partial\Omega_0),\,\B(\partial\Omega_0)},
\end{multline}
where $\divpw \uc'$ is the piecewise-defined function equal to $\div\uc$ in $\Omega$ and $\div\uc_0$ in $\Omega_0$. Then $\divpw \uc' \in L^2(\Omega')$.

Since the test function $\phi \in \D(\Omega')$ has compact support in the open set $\Omega'$, its trace is zero on the outer boundary $\partial\Omega'$. Since $\Omega_0 = \Omega' \setminus \overline{\Omega} = \Omega' \cap (\R^n \setminus \overline\Omega)$, its boundary satisfies $\partial\Omega_0 \subset \partial\Omega' \cup \partial(\R^n \setminus \overline\Omega) = \partial\Omega' \cup \bord$. Conversely, because $\bord$ is strictly contained within the open set $\Omega'$, the hypothesis $\bord = \partial(\R^n \setminus \overline\Omega)$ directly ensures that $\bord \subset \partial\Omega_0$. Consequently, $\partial\Omega_0$ is the disjoint union of the compact sets $\bord$ and $\partial\Omega_0 \cap \partial\Omega'$, meaning $\bord$ is a relatively open subset of $\partial\Omega_0$ to which Definition~\ref{DefResNormalTrace} legitimately applies. The trace $\operatorname{Tr}_{\partial\Omega_0}(\phi)$ vanishes q.e. on $\partial\Omega'$, hence belongs to $\B_{\bord}(\partial\Omega_0)$. We claim that $\Hb = \B_{\bord}(\partial\Omega_0)$ with equivalent norms. Indeed, by the intrinsic characterization of the trace (see the paragraph after Definition~\ref{DefTraceInt}), for any $v \in H^1(\Omega)$ or $v_0 \in H^1(\Omega_0)$ the traces $\operatorname{Tr}_{\bord} v$ and $\operatorname{Tr}_{\partial\Omega_0} v_0$ are both computed as the restriction to the relevant boundary of a global $H^1(\R^n)$-extension, independently of the domain. Because $\bord$ and $\partial\Omega'$ are strictly disjoint, the two trace spaces consist of the exact same elements on $\bord$. The identity map between them is continuous in both directions (by composition of the respective bounded extension and trace operators), which yields equivalent norms. Consequently, their dual spaces coincide, $\Hbm = \B'_{\bord}(\partial\Omega_0)$, and Definition~\ref{DefResNormalTrace} directly gives
$$ \langle \trno\uc_0, \operatorname{Tr}_{\partial\Omega_0}(\phi) \rangle_{\B'(\partial\Omega_0),\,\B(\partial\Omega_0)} = \langle (\trno\uc_0)_{|\bord}, \operatorname{Tr}_{\bord}(\phi) \rangle_{\Hbm,\, \Hb}. $$
The expression for the distributional divergence thus becomes
\begin{equation}
    \langle \div \uc', \phi \rangle = \int_{\Omega'} (\divpw \uc') \phi \dx - \langle \trn\uc + (\trno\uc_0)_{|\bord}, \operatorname{Tr}_{\bord}(\phi) \rangle_{\Hbm,\, \Hb}.
\end{equation}
This shows that the distribution $\div \uc'$ is the sum of a regular part (the $L^2$ element $\divpw \uc'$) and a singular part concentrated on the interface $\bord$.\\
For $\div \uc'$ to be in $L^2(\Omega')$, this singular part must be zero. This requires the functional $F := \trn\uc + (\trno\uc_0)_{|\bord}$ to be zero. The identity above shows that $F$ needs to vanish when tested against the set of traces 
$$T := \{ \operatorname{Tr}_{\bord}(\phi) \mid \phi \in \D(\Omega') \}.$$
Since $\Omega$ is an $H^1$-extension domain, this set of traces is dense in $\B(\bord)$.
To see why, let $g \in \B(\bord)$ be an arbitrary trace. By definition of the trace space, there exists a function $v \in H^1(\Omega)$ such that $\operatorname{Tr}_{\bord}(v) = g$. Since $\Omega$ is an $H^1$-extension domain, there exists a global extension $\tilde{v} \in H^1(\R^n)$ such that $\tilde{v}_{|\Omega} = v$. Furthermore, because $\overline{\Omega}$ is a compact subset of the open set $\Omega'$, the distance between $\overline{\Omega}$ and $\partial\Omega'$ is strictly positive. Consequently, there exists a smooth cut-off function $\eta \in \D(\Omega')$ such that $\eta = 1$ on an open neighborhood of $\overline{\Omega}$. 

The space $\D(\R^n)$ is dense in $H^1(\R^n)$, meaning there exists a sequence $(\psi_k)_{k\in\N} \subset \D(\R^n)$ such that $\psi_k \to \tilde{v}$ in the $H^1(\R^n)$ norm. We now define the sequence of truncated functions
$$ \phi_k := \eta \psi_k. $$
Because $\eta \in \D(\Omega')$, each $\phi_k$ is a smooth function with compact support strictly contained within $\Omega'$, so $\phi_k \in \D(\Omega')$. Moreover, since $\eta = 1$ on $\overline{\Omega}$, the restriction of $\phi_k$ to $\Omega$ is simply $\psi_k|_\Omega$. Thus, we have
$$ \|\phi_k - v\|_{H^1(\Omega)} = \|\psi_k - v\|_{H^1(\Omega)} \leq \|\psi_k - \tilde{v}\|_{H^1(\R^n)} \xrightarrow[k \to \infty]{} 0. $$
By continuity of the trace operator $\operatorname{Tr}_{\bord} : H^1(\Omega) \to \B(\bord)$, this implies that
$$ \operatorname{Tr}_{\bord}(\phi_k) \to \operatorname{Tr}_{\bord}(v) = g \quad \text{in } \B(\bord). $$
This proves that the set $T$ is indeed dense in $\B(\bord)$.\\
Therefore, $\uc' \in \Hc(\div,\Omega')$ if and only if $F = 0$, which is precisely the condition
$$ \trn\uc + (\trno\uc_0)_{|\bord} = 0 \quad \text{in } \B'(\bord). $$
\end{proof}

Let us present now density results using the new definition of the normal trace. These results are analogous to those established in the Lipschitz setting, see for example~\cite{GIRAULT-1986}. In the regular case, density is mainly used as a tool to define the operators. Here, the situation is different: since the normal trace is defined directly on $\Hdiv$, the density results arise as a consequence of this definition. We show a first density result.
\begin{proposition}\label{DensHdiv}
Let $\Omega$ be a bounded two-sided $H^1$-extension domain of $\R^n$, $n\ge2$. The space $(\D(\overline\Omega))^n$ is dense in $\Hdiv$:
\begin{equation}
    \forall \wc \in \Hdiv, \exists (\wc_k)_k \in \left((\D(\overline\Omega))^n\right)^{\mathbb{N}},\quad \lim_{k \to \infty} \|\wc - \wc_k\|_{\Hdiv} = 0.
\end{equation}
\end{proposition}

\begin{proof}
We use the same approach as Girault and Raviart in~\cite[Theorem 2.4, p.~27]{GIRAULT-1986} based on the following property in Banach spaces:
    a subspace $\mathcal{V}$ of a Banach space $M$  is dense in $M$ if and only if every element of $M'$  that vanishes on $\mathcal{V}$ also vanishes on $M$.
Let $l \in (\Hdiv)'$. By the Riesz representation theorem, $l$ has an associated element $ \lc$ of $\Hdiv$ such that for all $\uc \in \Hdiv$,
\begin{equation}
    \left\langle l,\uc\right\rangle = ( \lc,\uc)_{\Lc^2(\Omega)} + (\div  \lc,\div \uc)_{L^2(\Omega)} = \sum_{i=1}^{n} (l_i,u_i)_{L^2(\Omega)} + (l_{n+1},\div \uc)_{L^2(\Omega)},
\end{equation}
where $l_{n+1} =  \div  \lc$.

Now, suppose that $l$ vanishes on $(\D(\overline\Omega))^n$ and let $\mathcal{Z}$ denote the zero-extension operator. This implies that
\begin{equation*}
    \forall \phic \in (\D(\mathbb{R}^n))^n,\quad (\mathcal{Z}{ \lc},\phic)_{\Lc^2(\mathbb{R}^n)} + (\mathcal{Z}{l}_{n+1},\div \boldsymbol{\phic})_{L^2(\mathbb{R}^n)} = 0.
\end{equation*}
We differentiate in $\D'(\mathbb{R}^n)$ to obtain
\begin{equation*}
     \forall \phic \in (\D(\mathbb{R}^n))^n,\quad(\mathcal{Z}{ \lc},\phic)_{\Lc^2(\mathbb{R}^n)} = - (\mathcal{Z}{l}_{n+1},\div\boldsymbol{\phic})_{L^2(\mathbb{R}^n)} = \langle \nabla \mathcal{Z}{l}_{n+1},\phic \rangle.
\end{equation*}
Then, in the sense of distributions on $\mathbb{R}^n$, we have $\mathcal{Z}{ \lc} =\nabla \mathcal{Z}{l}_{n+1}$, with $\mathcal{Z}{ \lc} \in (L^2(\mathbb{R}^n))^n $. Therefore, $\mathcal{Z}{l}_{n+1} \in H^1(\mathbb{R}^n)$. It follows by Proposition~\ref{PropExtH10} that $l_{n+1} \in H^1_0(\Omega)$. Then, for all $\uc \in \Hdiv$, by the Green formula we have
\begin{align*}
    \left\langle l,\uc\right\rangle &= ( \lc,\uc)_{\Lc^2(\Omega)} + ( \div  \lc,\div \uc)_{L^2(\Omega)} = (\nabla {l}_{n+1}, \uc)_{\Lc^2(\Omega)} + ({l}_{n+1},\div \uc)_{L^2(\Omega)} \\
    &= \langle \Trn \uc , \Tr l_{n+1} \rangle_{\Hbm,\,\Hb} = 0
\end{align*}
since $l_{n+1} \in H^1_0(\Omega)$.
Thus, $l = 0$ on $(\D(\overline\Omega))^n$ implies $l=0$ on $\Hdiv$, meaning that $(\D(\overline\Omega))^n$ is dense in $\Hdiv$.
\end{proof}

This second density result gives a characterization of the kernel of the normal trace operator. The proof follows the same strategy as in~\cite[Theorem~2.6, p.~29]{GIRAULT-1986}. The main difference with Proposition~\ref{DensHdiv} is that here we obtain directly an element of $H^1(\Omega)$, so there is no need to use the zero-extension and Proposition~\ref{PropExtH10}. Consequently, for this result we do not need to assume that $\Omega$ is a two-sided $H^1$-extension domain.

\begin{definition}
Let $\Omega$ be an open subset of $\R^n$, $n\ge2$. The Hilbert space $\Hc_0(\div,\Omega)$ is defined as the closure of $(\D(\Omega))^n$ in $\Hdiv$:
    \begin{align}
        \forall \wc \in \Hc_0(\div,\Omega), \exists(\wc_k)_k \in ((\D(\Omega))^n)^{\mathbb{N}}, \lim\limits_{k \rightarrow +\infty} \|\wc-\wc_k\|_{\Hdiv} = 0.
    \end{align}
\end{definition}

\begin{proposition}\label{propHdiv}
Let $\Omega$ be a bounded $H^1$-extension domain of $\R^n$, $n\ge2$. The following equality holds:
\begin{equation}\label{DensityDiv}
 \Ker(\Trn)=\Hc_0(\div,\Omega).
\end{equation}
\end{proposition}

\begin{proof}

We first establish that $(\D(\Omega))^n$ is dense in $\Ker(\Trn)$ in a similar way as in Proposition~\ref{DensHdiv}. 
Let $l \in\Ker(\Trn)'$. Once again, $l$ has an associated element $ \lc$ of $\Ker(\Trn)$ such that for all $\uc \in\Ker(\Trn)$,
\begin{equation}
    \left\langle l,\uc\right\rangle = ( \lc,\uc)_{\Lc^2(\Omega)} + ( \div  \lc,\div \uc)_{L^2(\Omega)} = \sum_{i=1}^{n} (l_i,u_i)_{L^2(\Omega)} + (l_{n+1},\div \uc)_{L^2(\Omega)}.
\end{equation}
Now, suppose that $l$ vanishes on $(\D(\Omega))^n$. This implies that
\begin{equation*}
    \forall \phic \in (\D(\Omega))^n,\quad ( \lc,\boldsymbol{\phi})_{\Lc^2(\Omega)} + ( \div  \lc,\div\boldsymbol{\phi})_{L^2(\Omega)} = 0.
\end{equation*}
Then, in the sense of distributions on $\Omega$, we have
\begin{equation*}
     \lc=\nabla l_{n+1}.
\end{equation*}
Since $ \lc \in (L^2(\Omega))^n $, we have $l_{n+1} \in H^1(\Omega)$. It also holds that
\begin{equation*}
    \forall \uc \in\Ker(\Trn),\quad \left\langle l,\uc\right\rangle = (\nabla l_{n+1},\uc)_{\Lc^2(\Omega)} + (l_{n+1},\div \uc)_{L^2(\Omega)}.
\end{equation*}
According to our definition of the normal trace,
\begin{multline*}\label{DefNormTrace2}
    \forall \uc \in \Hdiv,\, \forall v \in  H^1(\Omega),\\
    (\uc,\nabla v)_{\Lc^2(\Omega)} + (\div \uc,v)_{L^2(\Omega)} = \left\langle \trn \uc,\Tr v\right\rangle_{\Hbm,\,\Hb}.
\end{multline*}
We apply this to $v = l_{n+1}$ knowing that $\Ker(\Trn) \subset \Hdiv$, which gives
\begin{equation*}
    \forall \uc \in\Ker(\Trn),\quad\left\langle l,\uc\right\rangle=\left\langle\trn \uc,\Tr l_{n+1}\right\rangle_{\Hbm,\,\Hb} = 0.
\end{equation*}
Hence, $(\D(\Omega))^n$ is dense in $\Ker(\Trn)$. Thus, the closure of $(\D(\Omega))^n$ is equal to $\Ker(\Trn)$ and by definition $\Hc_0(\div,\Omega)$ is the closure of $(\D(\Omega))^n$ in $\Hdiv$.
We conclude that $\Ker(\Trn)=\Hc_0(\div,\Omega)$.
\end{proof}

\section{Tangential trace operator}\label{sec:tangential_trace}

In what follows, we work in $\R^3$ and use the notations 
\begin{align}
    \Lc^2(\Omega)=(L^2(\Omega))^3, \text{ } \Hc^1(\Omega)=(H^1(\Omega))^3, \text{ } \Hbc = (\Hb)^3, \textit{ etc.}
\end{align}

We must therefore give separate weak meanings to the two classical boundary quantities $\boldsymbol u\times\boldsymbol n$ and $(\boldsymbol n\times\boldsymbol u)\times\boldsymbol n$. The first is represented below by the generalized tangential trace $\widetilde{\TrT}\boldsymbol u$, while the second is represented in Section 5 by the abstract tangential-component trace $\widetilde\piT\boldsymbol u$, with their respective boundary spaces specified in the corresponding definitions. In the literature dealing with Lipschitz domains, it is common to refine the trace spaces by decomposing fields into explicit normal and tangential components, such as the spaces $\boldsymbol{L}_t^2(\partial\Omega)$ and $\boldsymbol{H}^{1/2}_{\parallel}(\partial \Omega)$ detailed in~\cite[Section 2.2.2]{Ciarlet_2024}. However, in our irregular setting, the absence of a pointwise well-defined outward normal vector makes such geometric decompositions impossible. Consequently, our first approach aligns with the fundamental, classical definition of the tangential trace (as presented, for instance, in~\cite[Theorem~2.11, p.~29]{GIRAULT-1986}). In this classical framework, the trace is defined globally as a continuous linear functional acting on the full boundary trace space. To clarify, we adopt different notation for the (generalized) tangential trace $\uc \times \nc$, depending on the regularity of the field $\uc$: we denote this functional $\widetilde{\TrT} \uc$ for $\uc \in \Hcurl$ and $\TrT\uc$ for $\uc \in \Hc^1(\Omega)$. We begin with a generalization of the Stokes formula.

\begin{theorem}[Generalized Stokes formula]\label{ThGStokesF}
Let $\Omega$ be a bounded $H^1$-extension domain of $\mathbb{R}^3$. For any $\uc \in \Hcurl$ and $\vc \in \Hc^1(\Omega)$, let the bilinear form $b(\uc,\vc)$ be defined by the Stokes integral:
\begin{equation}\label{EqDefBilinearForm_b}
    b(\uc,\vc) := \int_{\Omega}(\uc\cdot\rotv\vc - \vc\cdot\rotv\uc)\dx.
\end{equation}
This form depends on $\vc$ only through its trace $\Trc\vc$. It therefore induces a unique continuous linear functional on the trace space $\Hbc$, denoted $\uc \times \nc \in \Hbcm$. This gives the generalized Stokes formula:
\begin{equation}\label{EqStokesTrT}
    \forall \vc \in \Hc^1(\Omega), \quad \langle \uc \times \nc, \Trc\vc \rangle_{\Hbcm,\, \Hbc} = b(\uc,\vc).
\end{equation}
\end{theorem}

\begin{proof}
For any $\vc \in \Hc^1(\Omega)$, let $\wc = \Trc \vc \in \Hbc$. Given $\uc \in \Hcurl$, we define the following linear functional:
\begin{equation}
    L_u : \wc \in \Hbc \mapsto L_u(\wc) = \int_{\Omega} \uc \cdot \rotvnesp(\boldsymbol{Ew}) \dx - \int_{\Omega} (\boldsymbol{Ew}) \cdot \rotv \uc \dx,
\end{equation}
where $\boldsymbol{E}: \Hbc \to \Hc^1(\mathbb{R}^3)$ is the 3D version of the extension operator introduced before (extension by component). Let $\vc = \left. \boldsymbol{E}\wc\right|_\Omega \in \Hc^1(\Omega)$. Then, $\Trc \vc$ exists and we have $\Trc \vc = \wc$. Thus, since the extension operator is continuous, and with the inequality~\eqref{IneqCcurl} applied to $\vc \in \Hc^1(\Omega)$, $L_u$ is continuous: 
\begin{align*}
|L_u \wc| &\leq \|\uc\|_{\Lc^2(\Omega)} \|\rotv \vc\|_{\Lc^2(\Omega)} + \|\rotv \uc\|_{\Lc^2(\Omega)}\|\vc\|_{\Lc^2(\Omega)}\\
&\leq \left(\|\uc\|^2_{\Lc^2(\Omega)} + \|\rotv\uc\|^2_{\Lc^2(\Omega)}\right)^{\frac{1}{2}} \left( \|\vc\|^2_{\Lc^2(\Omega)} + \|\rotv\vc\|^2_{\Lc^2(\Omega)} \right)^{\frac{1}{2}}\\
&\leq \sqrt{2} \|\uc\|_{\Hcurl} \|\vc\|_{\Hc^1(\Omega)}\\
&\leq \sqrt{2} C_{\boldsymbol{E}} \|\uc\|_{\Hcurl} \|\wc\|_{\Hbc},
\end{align*}
where the constant $C_{\boldsymbol{E}}$ is the norm of $\boldsymbol{E}$. We now verify that $L_u(\wc)$ is independent of the choice of $\boldsymbol{E}$. Let $\boldsymbol{E}'$ be another extension operator, associated with
\begin{equation*}
    L_u' : \wc \in \Hbc \mapsto L_u'(\wc) = \int_{\Omega} \uc \cdot \rotvnesp(\boldsymbol{E}'\boldsymbol{w}) \dx - \int_{\Omega} (\boldsymbol{E}'\boldsymbol{w}) \cdot \rotv \uc \dx.
\end{equation*}
Set $\vc' = \boldsymbol{E}'\wc|_\Omega$. Since $\Trc\vc = \Trc\vc' = \wc$, the difference $\boldsymbol\psi := \vc - \vc'$ belongs to $\Hc^1_0(\Omega)$, and
$$ L_u(\wc) - L_u'(\wc) = b(\uc, \boldsymbol\psi). $$
For any $\boldsymbol\phi \in \Dc(\Omega)$, the definition of the distributional curl gives $b(\uc, \boldsymbol\phi) = 0$. Since $\Dc(\Omega)$ is dense in $\Hc^1_0(\Omega)$ and the continuity bound above shows that $b(\uc, \cdot)$ is continuous on $\Hc^1(\Omega)$, passing to the limit gives $b(\uc, \boldsymbol\psi) = 0$ for every $\boldsymbol\psi \in \Hc^1_0(\Omega)$. In particular, $L_u(\wc)$ depends only on $\wc$, not on the extension operator. Therefore, as $\Trc \vc = \wc$, it is sufficient to define $L_u \wc = \left\langle \uc \times  \nc,\wc\right\rangle_{\Hbcm,\,\Hbc}$ to obtain the generalized Stokes formula~\eqref{EqStokesTrT}. 
\end{proof}

The Stokes formula gives us two different versions of the tangential trace operator. 
The first, defined on $\Hcurl$, is the natural one since the Stokes formula involves $\rotv$, although it is also   more delicate to handle. The second, defined on the smaller space $\Hc^1(\Omega)$, gives access to more structure, in particular to the tangential component trace and to the decomposition of the boundary field it makes possible. A comparable distinction is made in~\cite{BUFFA-2002}, although our definition of tangential traces differs significantly: all known constructions for regular or Lipschitz boundaries heavily rely on the system $(\boldsymbol{\tau}_1,\boldsymbol{\tau}_2,\boldsymbol{n})$ composed of two tangential vectors and a normal vector at a boundary point, which do not exist in the non-Lipschitz case.

To overcome this geometric obstruction, our strategy is to rely exclusively on global variational identities. A crucial step in characterizing the properties of our trace operators is to establish a norm equivalence between the $\Hcurl$ and $\Hc^1(\Omega)$ topologies on appropriate functional subspaces. In classical theory, proving such coercivity results (e.g., \cite{COSTABEL-1991}) requires manipulating tangential operators and normal vector identities. In our abstract setting, we must instead identify a subspace where the boundary contributions naturally vanish, allowing us to link the $\rotvnesp$ and $\operatorname{div}$ operators directly to the full gradient.

Motivated by this, we now look for a space on which the problems $\rotv \rotv + \textbf{I}$ and $-\Delta + \textbf{I}$ are equivalent. We recall that $\Hc^1(\Omega)$ is the Hilbert space endowed with the canonical inner product
\begin{equation}\label{EqInnerProductH1}
	\forall \uc,\vc\in \Hc^1(\Omega), \quad (\uc,\vc)_{\Hc^1(\Omega)} = (\nabla \uc,\nabla \vc)_{\Lc^2(\Omega)} + (\uc,\vc)_{\Lc^2(\Omega)},
\end{equation}
where $\nabla \uc$ denotes the tensor $(\partial_j u_i)_{1 \le i,j \le 3}$, and the $\Lc^2(\Omega)$ inner product for tensors is defined by
\begin{equation*}
    (\nabla \uc, \nabla \vc)_{\Lc^2(\Omega)} := \int_\Omega \nabla \uc : \nabla \vc \, \mathrm{d}x = \int_\Omega \sum_{1 \le i,j \le 3} \partial_j u_i \, \partial_j v_i \dx.
\end{equation*}
We first establish two technical lemmas which connect the $\rotvnesp$ and $\divnesp$ operators on $\Hc^1_0(\Omega)$.

\begin{lemma}\label{InclH10TrtTrn}
Let $\Omega$ be a bounded $H^1$-extension domain of $\R^3$. Then
$$
\Hc^1_0(\Omega) \subset \Ker(\TrT) \quad \text{and} \quad \Hc^1_0(\Omega) \subset \Ker(\Trn).
$$
\end{lemma}

\begin{proof}
Let $\vc\in\Hc^1_0(\Omega)$. For the first inclusion, we use the antisymmetric property of the Stokes formula: for all $\wc \in \Hc^1(\Omega)$, 
$$\langle \TrT\vc, \Trc \wc \rangle_{\Hbcm,\, \Hbc} = - \langle \TrT\wc, \Trc \vc \rangle_{\Hbcm,\, \Hbc} = 0.$$
For the second inclusion, we approximate $\vc\in\Hc^1_0(\Omega)$ by a sequence $(\psic_k) \subset \Dc(\Omega)$ and use the fact that the $\Hdiv$-norm is weaker than the $\Hc^1(\Omega)$-norm:
$$
\|\vc - \psic_k\|_{\Hdiv} \leq \sqrt{3} \|\vc - \psic_k\|_{\Hc^1(\Omega)} \xrightarrow[k \to \infty]{} 0,
$$
which means that $\vc \in \Hdivz = \Ker(\Trn)$ by Proposition~\ref{propHdiv}.
\end{proof}

\begin{lemma}\label{LemRotDivLapl}
Let $\Omega$ be a bounded $H^1$-extension domain of $\R^3$. For all $\uc \in \Hc^1(\Omega)$ and all $\vc \in \Hc^1_0(\Omega)$,
\begin{equation}\label{RotDivLapl}
    (\rotv\uc, \rotv\vc)_{\Lc^2(\Omega)} + (\div\uc, \div\vc)_{L^2(\Omega)} = (\nabla\uc, \nabla\vc)_{\Lc^2(\Omega)}.
\end{equation}
\end{lemma}

\begin{proof}
Let $\uc \in \Hc^1(\Omega)$ and $\vc \in \Hc^1_0(\Omega)$. By the density of $\Dc(\Omega)$ in $\Hc^1_0(\Omega)$ and the definition of the distributional \textbf{curl}, we have
\begin{equation*}
    \langle\rotvnesp(\rotv\uc), \vc\rangle = (\rotv\uc, \rotv\vc)_{\Lc^2(\Omega)}.
\end{equation*}
Moreover, by the same density argument and the definition of the distributional gradient,
\begin{equation*}
    \langle\nabla(\div\uc), \vc\rangle = -(\div\uc, \div\vc)_{L^2(\Omega)}.
\end{equation*}
Using the vector identity $\Delta\uc = \nabla(\div\uc) - \rotvnesp(\rotv\uc)$, we obtain
\begin{equation*}
    \langle\Delta\uc, \vc\rangle = -(\div\uc, \div\vc)_{L^2(\Omega)} - (\rotv\uc, \rotv\vc)_{\Lc^2(\Omega)}.
\end{equation*}
On the other hand, using the density of $\Dc(\Omega)$ in $\Hc^1_0(\Omega)$ and the definition of the distributional Laplacian, we find
$$\langle\Delta\uc, \vc\rangle = -(\nabla\uc, \nabla\vc)_{\Lc^2(\Omega)},$$ 
which yields~\eqref{RotDivLapl}.
\end{proof}

This identity immediately yields a powerful coercivity and equivalence result for this specialized subspace, extending classical results to non-Lipschitz geometries.

\begin{theorem}\label{ThmCoercivityZ}
Let $\Omega$ be a bounded $H^1$-extension domain of $\R^3$.
\begin{enumerate}
    \item The bilinear form 
    $$a(\uc,\vc) := (\rotv\uc,\rotv\vc)_{\Lc^2(\Omega)} + (\div\uc,\div\vc)_{L^2(\Omega)}$$ satisfies, for all $\uc \in \Hc^1_0(\Omega)$:
    \begin{equation}
        a(\uc,\uc) + \|\uc\|_{\Lc^2(\Omega)}^2 = \|\uc\|_{\Hc^1(\Omega)}^2.
    \end{equation}
    Consequently, the form $a + (\cdot,\cdot)_{\Lc^2}$ is coercive on $\Hc^1_0(\Omega)$ with respect to the $\Hc^1(\Omega)$-norm.
    \item Let $\Vc := \{\uc \in \Hc^1_0(\Omega) : \div\uc = 0\}$, a closed subspace of $\Hc^1_0(\Omega)$. For all $\uc, \vc \in \Vc$,
    \begin{equation*}
        (\rotv\uc,\rotv\vc)_{\Lc^2(\Omega)} + (\uc,\vc)_{\Lc^2(\Omega)} = (\nabla\uc,\nabla\vc)_{\Lc^2(\Omega)} + (\uc,\vc)_{\Lc^2(\Omega)} = (\uc,\vc)_{\Hc^1(\Omega)}.
    \end{equation*}
    In particular, the bilinear form $(\uc,\vc) \mapsto (\rotv\uc,\rotv\vc)_{\Lc^2(\Omega)} + (\uc,\vc)_{\Lc^2(\Omega)}$ restricted to $\Vc \times \Vc$ coincides with the inner product of $\Vc$. Hence, for every $\fc \in \Vc'$ the problem of finding $\uc \in \Vc$ with
    \begin{equation*}
        (\rotv\uc,\rotv\vc)_{\Lc^2(\Omega)} + (\uc,\vc)_{\Lc^2(\Omega)} = \langle \fc, \vc \rangle_{\Vc',\Vc} \qquad \text{for all } \vc \in \Vc
    \end{equation*}
    has a unique solution, which is the weak formulation of $\rotvnesp(\rotv\uc) + \uc = \fc$ on $\Vc$, and equally the weak formulation of $-\Delta\uc + \uc = \fc$ on $\Vc$.
\end{enumerate}
\end{theorem}

\begin{remark}\label{RemCostabel}
For a bounded $C^{1,1}$ domain, Costabel~\cite[Section 2]{COSTABEL-1991} establishes a
G\aa rding-type inequality on the larger spaces $\Hc^1(\Omega)\cap \Ker(\TrT)$ and $\Hc^1(\Omega)\cap \Ker(\Trn)$, with an additional lower-order
$\Lc^2$ term.  Theorem~\ref{ThmCoercivityZ} gives the corresponding exact identity for the
augmented form on $\Hc_0^1(\Omega)$ without using pointwise boundary geometry. In contrast with the Lipschitz case, we must work in the strictly smaller space $\Hc^1_0(\Omega)$ instead of the larger spaces where only the tangential or normal trace vanishes, reflecting the additional constraints imposed by the non-Lipschitz setting.
\end{remark}

\subsection{\texorpdfstring{$\Hcurl$-framework}{H(curl,Omega)-framework}}\label{SubsecTrTHcurl}

First, we give some properties of the broader version of the tangential trace, which is a direct adaptation of what is done in~\cite[Section 2.2.2]{Ciarlet_2024}.

\begin{theorem}\label{ThmTrTHcurl}
    Let $\Omega$ be a bounded $H^1$-extension domain of $\mathbb{R}^3$. Then the tangential trace operator is defined in the sense of Theorem~\ref{ThGStokesF} by
    \begin{align*}
        \widetilde{\TrT} \colon \Hcurl &\to \Hbcm\\
    \uc & \mapsto \widetilde{\TrT} \uc = \uc \times  \nc.
    \end{align*}
    It is a linear and continuous mapping, which satisfies by~\eqref{EqStokesTrT}:
    \begin{equation}
     \forall \vc \in \Hc^1(\Omega), \left\langle \widetilde{\TrT} \uc,\Trc \vc\right\rangle_{\Hbcm,\,\Hbc} =
     \int_{\Omega}(\uc\cdot\rotv\vc - \vc\cdot\rotv\uc)\dx.
    \end{equation}
\end{theorem}

\begin{proof}
Linearity comes from the Stokes formula~\eqref{EqStokesTrT}, by linearity of the integral and of the $\rotv$ operator.\newline
To show continuity, we proceed in the same way as we did in Theorem~\ref{ThNormalTrace}.
Let $\uc \in \Hcurl$ and $\vc \in \Hc^1(\Omega)$. By definition of the tangential trace, we have
    \begin{align}
        \left\langle \uc \times  \nc,\Trc \vc\right\rangle_{\Hbcm,\,\Hbc} =
        \int_{\Omega}(\uc\cdot\rotv\vc - \vc\cdot\rotv\uc)\dx.
    \end{align}
    Therefore, using the Cauchy-Schwarz inequality in $\mathbb{R}^2$, we find
    \begin{align*}
        |\left\langle \uc \times  \nc,\Trc \vc \right\rangle_{\Hbcm,\,\Hbc}| &\leq \|\uc\|_{\Lc^2(\Omega)} \|\rotv \vc\|_{\Lc^2(\Omega)} + \|\vc\|_{\Lc^2(\Omega)} \|\rotv \uc\|_{\Lc^2(\Omega)}  \\
        &\leq \|\uc\|_{\Hcurl} \|\vc\|_{\Hcurl}\\
        &\leq \sqrt{2}  \|\uc\|_{\Hcurl} \|\vc\|_{\Hc^1(\Omega)}.
    \end{align*}
        We then have for all $\vc \in \Hc^1(\Omega)$ such that $\wc = \Trc \vc$:
    \begin{equation*}
        |\left\langle \uc \times  \nc, \wc \right\rangle_{\Hbcm,\,\Hbc}| \leq  \sqrt{2} \|\uc\|_{\Hcurl} \|\vc\|_{\Hc^1(\Omega)}.
    \end{equation*}
    Hence, the previous inequality implies by~\eqref{normTri} that
    \begin{multline*}
    \forall \uc \in \Hcurl,\; \forall \wc \in \Hbc,\quad \\
    |\left\langle \uc \times \nc,\wc \right\rangle_{\Hbcm,\,\Hbc}| 
    \leq \sqrt{2} \|\uc\|_{\Hcurl} \|\wc\|_{\Hbc}.
    \end{multline*}
    Therefore, we have the continuity of $\widetilde{\TrT}$: 
    \begin{equation}
        \|\widetilde{\TrT}\, \uc\|_{\Hbcm} \leq \sqrt{2} \|\uc\|_{\Hcurl}.
    \end{equation}    
\end{proof}

\begin{remark}
Here, we do not claim that the tangential trace operator is surjective. In fact, even in the case of a Lipschitz domain, it is known not to be surjective, see~\cite{BUFFA-2002} for more details.
\end{remark}

The next step is to characterize the kernel of the tangential trace operator. For general domains in $\R^3$, this characterization depends on the geometric regularity of the domain. We therefore first state this characterization
as a formal assumption in $\R^3$ and then, in the next subsection, we will prove that this assumption is in fact satisfied by all bounded $H^1$-extension domains in the special case of $\R^2$.

\begin{definition}
Let $\Omega$ be an open subset of $\R^3$. The Hilbert space $\Hcurlz$ is defined as the closure of $\Dc(\Omega)$ in $\Hcurl$:
    \begin{align}
        \forall \wc \in \Hcurlz, \exists(\wc_k)_k \in (\Dc(\Omega))^{\mathbb{N}}, \lim\limits_{k \rightarrow +\infty} \|\wc-\wc_k\|_{\Hcurl} = 0.
    \end{align}
\end{definition}

It is clear from the generalized Stokes formula~\eqref{EqStokesTrT} and the continuity of the tangential trace that any element in $\Hcurlz$ has a zero tangential trace, which gives the inclusion $\Hcurlz \subset \Ker(\widetilde{\TrT})$. The reverse inclusion is more difficult. We formalize it as an assumption on the domain's regularity.

\begin{assumption}\label{AssumpHcurlz3D}
A bounded $H^1$-extension domain $\Omega \subset \R^3$ is called Stokes-regular if the following characterization holds: a field $\uc \in \Hcurl$ belongs to $\Hcurlz$ if and only if it satisfies the variational condition
\begin{equation}\label{EqVarCond}
    \forall \boldsymbol{\phi} \in \Dc(\overline\Omega), \quad \int_\Omega (\boldsymbol{u} \cdot \rotv \boldsymbol{\phi} - \boldsymbol{\phi} \cdot \rotv \boldsymbol{u}) \dx = 0.
\end{equation}
\end{assumption}

This assumption, which holds for instance for finite unions of star-shaped domains \cite[Lemma 2.4, p.~33]{GIRAULT-1986}, allows us to fully characterize the kernel of the tangential trace.

\begin{proposition}\label{propDensKerTrT}
Let $\Omega \subset \R^3$ be a bounded, Stokes-regular $H^1$-extension domain. Then:
\begin{equation}\label{DensityCurl}
  \Hcurlz = \Ker(\widetilde{\TrT}).
\end{equation}
\end{proposition}

\begin{proof}
Let $\uc \in \Hcurl$. Since $\Omega$ is Stokes-regular, we have the following equivalence:
\begin{equation*}
    \uc \in \Hcurlz \iff \forall \phic \in \Dc(\overline\Omega), \int_\Omega (\uc \cdot \rotv \phic - \phic \cdot \rotv \uc)\dx = 0.
\end{equation*}
Now we use the fact that $\Dc(\overline\Omega)$ is dense in $\Hc^1(\Omega)$. This gives
\begin{align*}
    \uc \in \Hcurlz &\iff \forall \vc \in \Hc^1(\Omega), \int_\Omega (\uc \cdot \rotv \vc - \vc \cdot \rotv \uc)\dx = 0 \\
    &\iff \widetilde{\TrT} \uc = 0 \\
    &\iff \uc \in \Ker(\widetilde{\TrT}).\tag*{\qedhere}
\end{align*}
\end{proof}

\begin{remark}
Our approach, which defines $\Hcurlz$ via smooth function approximation and then assumes the variational characterization \eqref{EqVarCond}, follows the classical structure of~\cite[Chapter I]{GIRAULT-1986}. We note that an alternative, equivalent approach is to take the kernel characterization, $\Hcurlz := \Ker(\widetilde{\TrT})$, as the primary definition of the space, as is done for instance in~\cite{Creo_2024}. Our assumption essentially states that these two common definitions coincide for the class of domains we consider.
\end{remark}

A further consequence of the Stokes-regularity assumption is the density of smooth functions in the full $\Hcurl$ space, a result analogous to ~\cite[Theorem 2.10, p.~34]{GIRAULT-1986}.

\begin{proposition}\label{DensHrot}
Let $\Omega$ be a bounded, Stokes-regular $H^1$-extension domain of $\R^3$. Then the space $\Dc(\overline\Omega)^3$ is dense in $\Hc(\rotv,\Omega)$.
\end{proposition}

\begin{proof}
The proof is similar to that of Proposition~\ref{DensHdiv}, with some slight changes at the end. Let $l$ be a continuous linear functional on $\Hcurl$ that vanishes on $\Dc(\overline\Omega)^3$. By the Riesz representation theorem, there exists $\lc \in \Hcurl$ such that for all $\uc \in \Hcurl$, 
$$\langle l, \uc \rangle = (\lc, \uc)_{\Hcurl}.$$
We define $\lc_{n+1} := \rotv \lc \in \Lc^2(\Omega)$. Because $l$ vanishes on $\Dc(\overline\Omega)^3$, the same logic as in previous proofs yields, in the sense of distributions on $\R^3$, $\mathcal{Z} \lc = - \rotv (\mathcal{Z}\lc_{n+1})$. This implies $\mathcal{Z}\lc_{n+1} \in \Hc(\rotv,\mathbb R^3)$, and thus $\lc_{n+1} \in \Hcurl$.\\
Since $l$ vanishes on $\Dc(\overline\Omega)^3$, we have the following equality:
\begin{align*}
    \forall \boldsymbol{\phi} \in \Dc(\overline\Omega)^3, \quad 0 &= \langle l,\boldsymbol{\phi}\rangle = (\lc, \boldsymbol{\phi})_{\Lc^2(\Omega)} + (\rotv \lc, \rotv \boldsymbol{\phi})_{\Lc^2(\Omega)}\\
    &= -(\rotv \lc_{n+1} , \boldsymbol{\phi})_{\Lc^2(\Omega)} + (\lc_{n+1}, \rotv \boldsymbol{\phi})_{\Lc^2(\Omega)}.
\end{align*}
It follows directly from Assumption~\ref{AssumpHcurlz3D} that $\lc_{n+1} \in \Hcurlz$.
According to Proposition~\ref{propDensKerTrT}, we can consider a sequence $(\boldsymbol\psi_k)_{k\in\mathbb N} \subset \Dc(\Omega)$ which converges to $\lc_{n+1}$ in the $\Hcurl$ norm. Using the definition of the weak curl for $\boldsymbol\psi_k \in \Dc(\Omega)$, taking the limit $k \to \infty$, and using the continuity of the inner product, we find:
\begin{equation*}
    \forall \boldsymbol{u} \in \Hcurl, \quad \langle l,\boldsymbol{u}\rangle = -(\rotv \lc_{n+1} , \boldsymbol{u})_{\Lc^2(\Omega)} + (\lc_{n+1}, \rotv \boldsymbol{u})_{\Lc^2(\Omega)} = 0.\qedhere
\end{equation*}
\end{proof}

\subsection{The Special Case of 2D Domains}
\label{sec:2D_tangential_trace}

We now show that for planar domains, Assumption~\ref{AssumpHcurlz3D} is not an assumption but a theorem that holds for all bounded $H^1$-extension domains. This is due to the special structure of the 2D curl operators. In what follows, we adopt a dedicated notation for vectors and vector-valued function spaces in $\R^2$, to distinguish them from their three-dimensional counterparts. For instance, we write $\ucd \in \Lcd^2(\Omega)$ in place of $\uc \in \Lc^2(\Omega)$.

\begin{definition}
Let $\Omega \subset \R^2$ be an open set. For a vector field $\ucd = (u_1, u_2)^T \in \Lcd^2(\Omega)$ and a scalar field $ v \in L^2(\Omega)$, we define their weak curls in the sense of distributions by:
\begin{align*}
    \rots\ucd &= \partial_1 u_2 - \partial_2 u_1, \\
    \rotv v &= \begin{pmatrix} \partial_2  v \\ -\partial_1  v \end{pmatrix}.
\end{align*}
We then define the space of square-integrable vector fields with square-integrable scalar curl as:
\begin{equation*}
    \Hcurls = \{ \ucd \in \Lcd^2(\Omega) \mid \rots\ucd \in L^2(\Omega) \}.
\end{equation*}
\end{definition}

The space $\Hcurls$ is the two-dimensional analogue of the curl energy space defined in~\eqref{EqHcurl}, and $\Hcurlsz := \overline{\Dcd(\Omega)}^{\Hcurls}$ is the closure of compactly supported smooth test functions.

It is essential to note that the 2D operators $\rots$ and $\rotv$ have a special structure that is not present in 3D. Both are directly related to the gradient and divergence via the 90° clockwise rotation matrix
$$ \mathbf{M} = \begin{pmatrix} 0 & 1 \\ -1 & 0 \end{pmatrix}. $$
For any scalar field $v$, the vector curl is a rotated gradient:
\begin{equation}\label{EqCurlEqualNabla}
    \rotv v = (\partial_2  v, -\partial_1  v)^T = \mathbf{M} \nabla v.
\end{equation}
Since $\mathbf{M}$ is an orthogonal matrix, we have the pointwise identity $|\rotv v(x)| = |\nabla v(x)|$, which gives the equality of function spaces
\begin{equation}\label{EqHcurlEqualH1}
    H(\rotv,\Omega) = H^1(\Omega).
\end{equation}
Similarly, for any vector field $\ucd \in \Lcd^2(\Omega)$, the scalar curl is the opposite of the divergence of the rotated field:
\begin{equation}\label{EqCurlEqualDiv}
    \rots\ucd = -\div(\mathbf{M}^T\ucd).
\end{equation}
Since $|\mathbf{M}^T\ucd| = |\ucd|$ and $|\div(\mathbf{M}^T\ucd)| = |\rots\ucd|$ pointwise, the map $\ucd \mapsto \mathbf{M}^T\ucd$ is an isometric isomorphism
\begin{equation}\label{EqCurlDivIso}
    \mathbf{M}^T : \Hcurls \rightarrow{} \Hdivs, \qquad \|\mathbf{M}^T\ucd\|_{\Hdivs} = \|\ucd\|_{\Hcurls}.
\end{equation}
Since $\mathbf{M}^T$ maps $\Dcd(\Omega)$ onto $\Dcd(\Omega)$ (being a constant invertible matrix), it also maps $\Hcurlsz$ onto $\Hdivsz$ isometrically. This reduction to the divergence setting allows us to obtain all the results of this subsection from their analogues in Section~\ref{sec:normal_trace}, as we now show.

\begin{theorem}\label{ThStokesTrT2D}
Let $\Omega$ be a bounded $H^1$-extension domain of $\mathbb{R}^2$.
For every $\ucd \in \Hcurls$, the map
\begin{equation}\label{EqDefTrT2D}
    l_u : w \in \Hb \mapsto \int_{\Omega} \ucd \cdot \rotvnesp(E w) \dx - \int_{\Omega} (E w)\,\rots\ucd \dx
    \;\in \R,
\end{equation}
where $E\colon \Hb \to H^1(\R^2)$ is any bounded linear extension operator, defines a bounded linear functional on $\Hb$ that is independent of the choice of $E$.  We denote this functional $\trT\ucd \in \Hbm$. It satisfies the generalized Stokes formula: for all $ v \in H^1(\Omega)$,
\begin{equation}\label{EqStokesTrT2D}
    \langle \trT\ucd,\Tr v\rangle_{\Hbm,\,\Hb} = \int_{\Omega} \ucd \cdot \rotv v\dx - \int_{\Omega}  v\,\rots\ucd\dx,
\end{equation}
and the operator $\trT\colon\Hcurls\to\Hbm$ is linear and continuous.
\end{theorem}

\begin{proof}
Let $\ucd \in \Hcurls$ and $v \in H^1(\Omega)$. Using $\ucd = \mathbf{M}\mathbf{M}^T\ucd$, $\rotv v = \mathbf{M}\nabla v$, and $\rots\ucd = -\div(\mathbf{M}^T\ucd)$,
\begin{align*}
    \int_{\Omega} \ucd \cdot \rotv v\dx - \int_{\Omega} v\,\rots\ucd\dx &= \int_{\Omega} (\mathbf{M}\mathbf{M}^T\ucd) \cdot (\mathbf{M}\nabla v)\dx + \int_{\Omega} v\,\div(\mathbf{M}^T\ucd)\dx \\
    &= \int_{\Omega} (\mathbf{M}^T\ucd) \cdot \nabla v\dx + \int_{\Omega} (\div(\mathbf{M}^T\ucd))\, v\dx,
\end{align*}
where the first equality uses $\mathbf{M}^T\mathbf{M} = I$ (orthogonality of $\mathbf{M}$). Since $\mathbf{M}^T\ucd \in \Hdivs$ by~\eqref{EqCurlDivIso}, the right-hand side is exactly $\langle \trn(\mathbf{M}^T\ucd), \Tr v \rangle_{\Hbm,\,\Hb}$ by Theorem~\ref{ThNormalTrace}. Hence $\trT\ucd = \trn(\mathbf{M}^T\ucd)$, and the well-definedness, independence of $E$, linearity and continuity of $\trT$ follow from the corresponding properties of $\trn$ established in Theorem~\ref{ThNormalTrace}.
\end{proof}

The following proposition provides the full characterization of the kernel of this operator, proving that the density property holds without extra geometric assumptions.

\begin{proposition}\label{PropKerTrT2D}
Let $\Omega$ be a bounded $H^1$-extension domain of $\mathbb{R}^2$. Then
\begin{equation}\label{EqKerTrT2D}
    \Ker(\trT) = \Hcurlsz.
\end{equation}
\end{proposition}

\begin{proof}
Since $\trT\ucd = \trn(\mathbf{M}^T\ucd)$ by the proof of Theorem~\ref{ThStokesTrT2D}, $\ucd \in \Ker(\trT)$ if and only if $\mathbf{M}^T\ucd \in \Ker(\trn)$. By Proposition~\ref{propHdiv}, $\Ker(\trn) = \Hdivsz$. Since $\mathbf{M}^T : \Hcurls \to \Hdivs$ is an isometric isomorphism mapping $\Hcurlsz$ onto $\Hdivsz$ (as observed above),
\begin{equation*}
    \Ker(\trT) = \mathbf{M}(\Hdivsz) = \Hcurlsz. \qedhere
\end{equation*}
\end{proof}

\begin{remark}
As shown in Proposition~\ref{propDensKerTrT}, equality~\eqref{EqKerTrT2D} is equivalent to the Stokes regularity of $\Omega$, since $\Omega$ is an $H^1$-extension domain. It means that every $H^1$-extension domain in $\R^2$ is Stokes-regular.
\end{remark}

As we did for the space $\Hdiv$, we finally give a density result for the whole $\Hcurls$. The two-sided hypothesis is necessary here, in contrast to Proposition~\ref{PropKerTrT2D} which required only the one-sided condition — mirroring exactly the structural difference between Proposition~\ref{DensHdiv} and Proposition~\ref{propHdiv}.

\begin{proposition}\label{PropDensCurl2D}
Let $\Omega$ be a bounded two-sided $H^1$-extension domain of $\mathbb{R}^2$. Then, the space $\Dcd(\overline\Omega)$ is dense in $\Hcurls$.
\end{proposition}

\begin{proof}
Since $\mathbf{M}^T$ maps $\Dcd(\overline\Omega)$ onto $\Dcd(\overline\Omega)$ and $\Hcurls$ isometrically onto $\Hdivs$, the density of $\Dcd(\overline\Omega)$ in $\Hcurls$ is equivalent to the density of $\Dcd(\overline\Omega)$ in $\Hdivs$, which is Proposition~\ref{DensHdiv}.
\end{proof}

\subsection{\texorpdfstring{$\Hc^1(\Omega)$-framework and the space $\Hbcp$}{H1-framework and Bperp}}
\label{SubsecTrTH1}

In the Lipschitz case, Buffa, Costabel and Sheen~\cite{BUFFA-2002} define the image of the $\Hc^1(\Omega)$-tangential trace, and endow it with the associated quotient norm. Adopting the same strategy, we define the following space.

\begin{definition}\label{DefBperp}
Let $\Omega$ be a bounded $H^1$-extension domain of $\R^3$. We define
\begin{equation}\label{EqDefBperp}
    \Hbcp := \TrT\bigl(\Hc^1(\Omega)\bigr) \subset \Hbcm,
\end{equation}
endowed with the quotient norm
\begin{equation}\label{EqNormBperp}
    \|\fc\|_{\Hbcp} := \inf_{\substack{\vc \in \Hc^1(\Omega) \\ \TrT\vc = \fc}} \|\vc\|_{\Hc^1(\Omega)}.
\end{equation}
\end{definition}

We now establish the fundamental properties of this space.

\begin{theorem}\label{ThBperp}
Let $\Omega$ be a bounded $H^1$-extension domain of $\R^3$, and let $\Kc := \Ker(\TrT)$ be the kernel of the tangential trace on $\Hc^1(\Omega)$.
\begin{enumerate}

    \item For every functional $\fc \in \Hbcp$, there exists a unique potential $\wc_f \in \Kc^{\perp_{\Hc^1}}$ such that $\TrT\wc_f = \fc$. This $\wc_f$ is the unique minimizer of the $\Hc^1(\Omega)$-norm among all pre-images of $\fc$:
    \begin{equation}\label{EqMinWcf}
        \|\wc_f\|_{\Hc^1(\Omega)} = \min_{\substack{\vc\in\Hc^1(\Omega)\\\TrT\vc=\fc}} \|\vc\|_{\Hc^1(\Omega)}.
    \end{equation}

    \item The space $\Hbcp$, endowed with the inner product
    \begin{equation}\label{EqInnerBperp}
        (\fc,\gc)_{\Hbcp} := (\wc_f, \wc_g)_{\Hc^1(\Omega)},
    \end{equation}
    where $\wc_f$ and $\wc_g$ are the unique optimal lifts described above, is a Hilbert space. The norm induced by this inner product coincides with the quotient norm from \eqref{EqNormBperp}.

    \item The restriction of the trace operator $\TrT$ to the subspace $\Kc^{\perp_{\Hc^1}}$ is a bijective isometry from $(\Kc^{\perp_{\Hc^1}}, \|\cdot\|_{\Hc^1(\Omega)})$ to $(\Hbcp, \|\cdot\|_{\Hbcp})$:
    \begin{equation}\label{EqIsomTrT}
        \forall \wc \in \Kc^{\perp_{\Hc^1}}, \quad 
        \|\TrT\wc\|_{\Hbcp} = \|\wc\|_{\Hc^1(\Omega)}.
    \end{equation}

    \item The space of optimal lifts is characterized by a weak boundary value problem:
    \begin{equation}\label{EqV1T}
    \begin{aligned}
    \Kc^{\perp_{\Hc^1}}
    = \Bigl\{ \wc \in \Hc^1(\Omega) \;\Bigm|\;&
    -\Delta \wc + \wc = \mathbf{0} \text{ weakly in } \Omega, \\
    &\forall \vc \in \Kc,\; \left\langle \frac{\partial \wc}{\partial \nc}, \Trc \vc \right\rangle_{\Hbcm,\,\Hbc} = 0 \Bigr\}.
    \end{aligned}
    \end{equation}
\end{enumerate}
\end{theorem}

\begin{remark}\label{RemMinBperp}
We note that the infimum defining the quotient norm of $\Hbcp$ is achieved. This minimum is uniquely realized by the optimal lift $\wc_f \in \Kc^{\perp_{\Hc^1}}$.
\end{remark}

\begin{proof}
\noindent\emph{Points 1, 2, and 3.} 
Since the trace operator $\TrT: \Hc^1(\Omega) \to \Hbcm$ is continuous, its kernel $\Kc = \Ker(\TrT)$ is a closed subspace of the Hilbert space $\Hc^1(\Omega)$. Points 1, 2, and 3 are then direct consequences of standard quotient Hilbert space theory. Specifically, the orthogonal decomposition $\Hc^1(\Omega) = \Kc \oplus \Kc^{\perp_{\Hc^1}}$ guarantees that for each trace $\fc \in \Hbcp$, there is a unique optimal lift $\wc_f \in \Kc^{\perp_{\Hc^1}}$ realizing the minimum in \eqref{EqMinWcf}. Furthermore, the restriction $\TrT|_{\Kc^{\perp_{\Hc^1}}}$ is a bijective isometry from the complete space $\Kc^{\perp_{\Hc^1}}$ onto $(\Hbcp, \|\cdot\|_{\Hbcp})$ by construction under the quotient norm, which immediately transfers the Hilbert space structure to $\Hbcp$.

\medskip
\noindent\emph{Point 4.} 
Let $\uc \in \Kc^{\perp_{\Hc^1}}$, so $(\uc,\vc)_{\Hc^1(\Omega)} = 0$ for all $\vc \in \Kc$. Since $\Dc(\Omega) \subset \Hc^1_0(\Omega) \subset \Kc$ by Lemma~\ref{InclH10TrtTrn}, we also have $(\uc,\vc)_{\Hc^1(\Omega)} = 0$ for all $\vc \in \Dc(\Omega)$. Expanding the $\Hc^1$-inner product yields:
\begin{equation*}
    0 = (\uc,\vc)_{\Hc^1(\Omega)} = (\nabla\uc,\nabla\vc)_{\Lc^2(\Omega)} + (\uc,\vc)_{\Lc^2(\Omega)} = \langle -\Delta\uc+\uc,\vc\rangle.
\end{equation*}
This holds for all $\vc\in\Dc(\Omega)$, hence $-\Delta\uc+\uc=\mathbf{0}$ weakly in $\Omega$. Consequently, $\Delta\uc=\uc\in\Lc^2(\Omega)$, and the weak normal derivative $\frac{\partial\uc}{\partial\nc}\in\Hbcm$ is well-defined via Definition~\ref{DefNormDerInt}. 

Applying the Green formula for any $\vc \in \Kc$, and using the weak equation $-\Delta\uc+\uc=\mathbf{0}$, we obtain:
\begin{align*}
    0 &= (\uc,\vc)_{\Hc^1(\Omega)} = \langle -\Delta\uc+\uc,\vc\rangle + \left\langle \frac{\partial\uc}{\partial\nc}, \Trc\vc \right\rangle_{\Hbcm,\, \Hbc}\\
    &= \left\langle \frac{\partial\uc}{\partial\nc}, \Trc\vc \right\rangle_{\Hbcm,\, \Hbc},
\end{align*}
which proves the weak boundary condition. 

Conversely, if $\uc\in\Hc^1(\Omega)$ satisfies $-\Delta\uc+\uc=\mathbf{0}$ weakly and $\left\langle\frac{\partial\uc}{\partial\nc},\Trc\vc\right\rangle = 0$ for all $\vc \in \Kc$, the Green formula directly yields $(\uc,\vc)_{\Hc^1(\Omega)} = 0$ for all $\vc \in \Kc$, confirming that $\uc \in \Kc^{\perp_{\Hc^1}}$.
\end{proof}

\section{Tangential component trace operator}\label{sec:tangential_component}

In this section we define and study the (generalized) tangential component trace operator $\widetilde\piT$, the abstract analogue of the operator from~\cite[Section~2.2.3]{Ciarlet_2024}. We follow the strategy of~\cite{BUFFA-2002}: in the Lipschitz case, the tangential component trace operator is first defined on $\Hc^1(\Omega)$ as a geometric projection, and then extended to $\Hcurl$ using the Green formula (cf.~\cite[Eq.~(27)]{BUFFA-2002}). In our abstract setting, we cannot define $\piT$ geometrically. Instead, we take this formula as the definition, with the target space replaced by the abstract space $\Hbcp$ established in Section~\ref{SubsecTrTH1}.

\subsection{\texorpdfstring{$\Hc^1(\Omega)$-framework and the space $\Hbct$}{H1-framework and BT}}
\label{SubsecPiTH1}

Given $\uc\in\Hc^1(\Omega)$, we want to define $\piT\uc$ as a linear functional on $\Hbcp = \TrT(\Hc^1(\Omega))$.
The natural candidate, by analogy with formula~\cite[Eq.~(27)]{BUFFA-2002}, is to set for all $\vc \in \Hc^1(\Omega)$:
\begin{equation}\label{EqCandidatePiT}
    \langle \piT\uc, \TrT\vc \rangle_{\Hbcpm,\, \Hbcp} := -b(\uc, \vc),
\end{equation}
where $b$ is defined by the Stokes integral, see equation~\eqref{EqDefBilinearForm_b}. Before adopting this as a definition, we must verify that
the right-hand side depends on $\vc$ only through $\TrT\vc$, i.e., that two different
choices of pre-image give the same value. Let $\vc_1, \vc_2 \in \Hc^1(\Omega)$ satisfy
$\TrT\vc_1 = \TrT\vc_2$. Then $\vc_1 - \vc_2 \in \Kc = \Ker(\TrT)$. For any element
$\vc^0 \in \Kc$, the Stokes formula~\eqref{EqStokesTrT} gives
\begin{equation*}
    \forall \wc \in \Hc^1(\Omega), \quad b(\vc^0, \wc) = \langle \TrT\vc^0, \Trc\wc \rangle_{\Hbcm,\, \Hbc} = 0
\end{equation*}
since $\TrT\vc^0 = \mathbf{0}$. By antisymmetry of $b$, this implies
\begin{equation}\label{EqNullitybuv}
    \forall \uc \in \Hc^1(\Omega), \, \forall \vc^0 \in \Kc, \quad
    b(\uc, \vc^0) = -b(\vc^0, \uc) = 0.
\end{equation}
Applying this with $\vc^0 = \vc_1 - \vc_2$:
\begin{equation*}
    b(\uc, \vc_1) - b(\uc, \vc_2) = b(\uc, \vc_1 - \vc_2) = 0,
\end{equation*}
so the right-hand side of~\eqref{EqCandidatePiT} is indeed independent of the choice
of pre-image of $\TrT\vc$. Formula~\eqref{EqCandidatePiT} therefore defines $\piT\uc$
unambiguously as a linear functional on $\Hbcp$.

\begin{definition}\label{DefPiTH1}
Let $\Omega$ be a bounded $H^1$-extension domain of $\R^3$. For $\uc \in \Hc^1(\Omega)$,
we define $\piT\uc \in \Hbcpm$ by
\begin{equation}\label{EqDefPiTH1}
    \forall \fc \in \Hbcp, \quad
    \langle \piT\uc, \fc \rangle_{\Hbcpm,\, \Hbcp} := -b(\uc, \wc_f),
\end{equation}
where $\wc_f \in \Kc^{\perp_{\Hc^1}}$ is the unique optimal lift of $\fc$ from
Theorem~\ref{ThBperp}.
\end{definition}

The use of the optimal lift $\wc_f$ in~\eqref{EqDefPiTH1} is simply a canonical choice
of pre-image: by the well-definedness argument above, any $\vc \in \Hc^1(\Omega)$ with
$\TrT\vc = \fc$ gives the same value $-b(\uc, \vc)$.

\begin{proposition}\label{PropPiTH1}
Let $\Omega$ be a bounded $H^1$-extension domain of $\R^3$. The operator $\piT : \Hc^1(\Omega) \to \Hbcpm$ is linear and continuous, with 
$$\|\piT\uc\|_{\Hbcpm} \leq \sqrt{2} \, \|\uc\|_{\Hc^1(\Omega)}.$$ 
For all $\uc, \vc \in \Hc^1(\Omega)$:
\begin{equation}\label{EqPiTuTrTv}
    b(\uc, \vc) = -\langle \piT\uc, \TrT\vc \rangle_{\Hbcpm,\, \Hbcp} = \langle \piT\vc, \TrT\uc \rangle_{\Hbcpm,\, \Hbcp}.
\end{equation}
\end{proposition}

\begin{proof}
Linearity is immediate from linearity of $b$ and of $\fc \mapsto \wc_f$. Continuity is a consequence of the Cauchy-Schwarz inequality and the isometry~\eqref{EqIsomTrT}:
\begin{align*}
    |\langle \piT\uc, \fc \rangle_{\Hbcpm,\, \Hbcp}| &= |b(\uc, \wc_f)|\\
    &\leq \sqrt{2} \, \|\uc\|_{\Hc^1(\Omega)} \|\wc_f\|_{\Hc^1(\Omega)}\\
    &\leq \sqrt{2} \, \|\uc\|_{\Hc^1(\Omega)} \|\fc\|_{\Hbcp}.
\end{align*}
We now prove equation~\eqref{EqPiTuTrTv}. The first equality follows from Definition~\ref{DefPiTH1} and~\eqref{EqCandidatePiT}. The second follows by antisymmetry $b(\uc, \vc) = -b(\vc, \uc)$ and the first equality applied to the pair $(\vc, \uc)$.
\end{proof}

Proposition~\ref{PropPiTH1} gives the abstract analogue of equation~(27) of~\cite{BUFFA-2002} for $\uc, \vc \in \Hc^1(\Omega)$. Next, we establish that $\TrT$ and $\piT$ have the same kernel on $\Hc^1(\Omega)$, the abstract analogue of~\cite[Eq.~(8)]{BUFFA-2002}.

\begin{proposition}\label{PropSameKer}
Let $\Omega$ be a bounded $H^1$-extension domain of $\R^3$. On $\Hc^1(\Omega)$:
\begin{equation}\label{EqSameKer}
    \Ker(\TrT) = \Ker(\piT).
\end{equation}
\end{proposition}

\begin{proof}
Let $\vc \in \Kc = \Ker(\TrT)$, so $\TrT\vc = \mathbf{0}$. Let $\fc \in \Hbcp$ with optimal lift $\wc_f \in \Kc^{\perp_{\Hc^1}}$. By Definition~\ref{DefPiTH1}:
\begin{equation*}
    \langle \piT\vc, \fc \rangle_{\Hbcpm,\, \Hbcp} = -b(\vc, \wc_f).
\end{equation*}
Since $\vc \in \Kc$ and $\wc_f \in \Hc^1(\Omega)$, property~\eqref{EqNullitybuv} gives $b(\vc, \wc_f) = 0$. As this holds for all $\fc \in \Hbcp$, we conclude $\piT\vc = \mathbf{0}$.

Now, let $\vc \in \Hc^1(\Omega)$ satisfy $\piT\vc = \mathbf{0}$. By the Stokes formula~\eqref{EqStokesTrT}, for all $\uc \in \Hc^1(\Omega)$:
\begin{equation*}
    \langle \TrT\vc, \Trc\uc \rangle_{\Hbcm,\, \Hbc} = b(\vc, \uc).
\end{equation*}
By formula~\eqref{EqPiTuTrTv}:
\begin{equation*}
    b(\vc, \uc) = -\langle \piT\vc, \TrT\uc \rangle_{\Hbcpm,\, \Hbcp} = 0,
\end{equation*}
since $\piT\vc = \mathbf{0}$ by hypothesis. Hence $\langle \TrT\vc, \Trc\uc \rangle_{\Hbcm,\, \Hbc} = 0$ for all $\uc \in \Hc^1(\Omega)$. Since $\Trc : \Hc^1(\Omega) \to \Hbc$ is surjective by Theorem~\ref{ThTrisom}, this gives $\TrT\vc = \mathbf{0}$.
\end{proof}

Proposition~\ref{PropSameKer} shows that $\piT$ and $\TrT$ share the same kernel $\Kc$. The only property of $\TrT$ that was used in Theorem~\ref{ThBperp} to construct the Hilbert structure on $\Hbcp$ was the fact that $\Kc$ is a closed subspace of $\Hc^1(\Omega)$, which gives the orthogonal decomposition $\Hc^1(\Omega) = \Kc \oplus \Kc^{\perp_{\Hc^1}}$ and hence the unique optimal lifts. Since $\Ker(\piT) = \Kc$ is the same closed subspace, the entire argument of Theorem~\ref{ThBperp} applies verbatim to $\piT$, yielding a Hilbert space structure on $\Hbct := \piT(\Hc^1(\Omega))$.

\begin{definition}\label{DefBT}
We define
\begin{equation}\label{EqDefBT}
    \Hbct := \piT\bigl(\Hc^1(\Omega)\bigr) \subset \Hbcpm,
\end{equation}
endowed with the quotient norm 
\begin{equation}\label{EqQuotNormBT}
    \|\fc\|_{\Hbct} := \inf_{\substack{\vc \in \Hc^1(\Omega) \\ \piT\vc = \fc}} \|\vc\|_{\Hc^1(\Omega)}.
\end{equation}
\end{definition}

\begin{theorem}\label{ThBT}
Let $\Omega$ be a bounded $H^1$-extension domain of $\R^3$. The space $\Hbct$, endowed with the inner product
\begin{equation}\label{EqInnerBT}
    (\fc, \gc)_{\Hbct} := (\wc_f, \wc_g)_{\Hc^1(\Omega)},
\end{equation}
where $\wc_f, \wc_g \in \Kc^{\perp_{\Hc^1}}$ are the unique elements with $\piT\wc_f = \fc$ and $\piT\wc_g = \gc$, is a Hilbert space. The restriction of $\piT$ to $\Kc^{\perp_{\Hc^1}}$ is a bijective isometry:
\begin{equation}\label{EqIsomPiT}
    \forall \wc \in \Kc^{\perp_{\Hc^1}}, \quad \|\piT\wc\|_{\Hbct} = \|\wc\|_{\Hc^1(\Omega)}.
\end{equation}
Moreover, for every $\vc \in \Hc^1(\Omega)$, the unique element $\wc \in \Kc^{\perp_{\Hc^1}}$ satisfying $\piT\wc = \piT\vc$ is the same as the unique element satisfying $\TrT\wc = \TrT\vc$. In other words, the optimal lifts for $\piT$ and for $\TrT$ coincide on $\Kc^{\perp_{\Hc^1}}$.
\end{theorem}

\begin{remark}\label{RemMinBT}
Analogously to Remark~\ref{RemMinBperp}, the infimum defining the quotient norm in~\eqref{EqQuotNormBT} is reached on the optimal lift $\wc_f \in \Kc^{\perp_{\Hc^1}}$, which gives the minimum.
\end{remark}

\begin{proof}
By Proposition~\ref{PropSameKer}, $\Ker(\piT) = \Kc = \Ker(\TrT)$, so the argument is identical to that of Theorem~\ref{ThBperp}, replacing $\TrT$ by $\piT$ throughout. Furthermore, for any $\vc \in \Hc^1(\Omega)$, both optimal lifts correspond to the same element $\wc \in \Kc^{\perp_{\Hc^1}}$, which is simply the orthogonal projection of $\vc$ onto $\Kc^{\perp_{\Hc^1}}$, so they necessarily coincide.
\end{proof}

Once $\Hbct$ is defined, we can refine the target space of $\widetilde\TrT$. In the Lipschitz case, the integration by parts formula~\cite[Eq.~(27)]{BUFFA-2002} shows that the tangential trace is a functional on the image of the tangential component trace. The following theorem is the abstract analogue, and it deliberately reuses the symbol $\widetilde\TrT$ for the resulting functional $\Hcurl \to \Hbctm$ of~\eqref{EqDefTrTonBt}. For a general field of $\Hcurl$, this functional tests only against the optimal lifts in $\Kc^{\perp_{\Hc^1}}$, so its agreement with the full trace holds a priori only on $\Hc^1(\Omega)$ through~\eqref{EqbuvTrTpiT}, and it is extended to $\Hcurl$ under Stokes regularity in Lemma~\ref{LembuvHcurl}. The two readings need not agree on all of $\Hcurl$, the refined one being a priori weaker, so whenever an equality $\widetilde\TrT\uc = 0$ or a norm $\|\widetilde\TrT\uc\|$ is written, the ambient space $\Hbcm$ or $\Hbctm$ is indicated explicitly.

\begin{theorem}\label{ThmTrTtoBtm}
Let $\Omega$ be a bounded $H^1$-extension domain of $\R^3$. The operator $\widetilde\TrT : \Hcurl \to \Hbctm$, defined for $\uc \in \Hcurl$ and $\fc \in \Hbct$ with optimal lift $\wc_f \in \Kc^{\perp_{\Hc^1}}$ by
\begin{equation}\label{EqDefTrTonBt}
    \langle \widetilde\TrT\uc, \fc \rangle_{\Hbctm,\, \Hbct} := b(\uc, \wc_f),
\end{equation}
is well-defined, linear, and continuous with
\begin{equation}\label{EqContTrTBtm}
    \|\widetilde\TrT\uc\|_{\Hbctm} \leq \sqrt{2} \, \|\uc\|_{\Hcurl}.
\end{equation}
Moreover, for all $\uc \in \Hc^1(\Omega)$ and $\vc \in \Hc^1(\Omega)$:
\begin{equation}\label{EqbuvTrTpiT}
    b(\uc, \vc) = \langle \TrT\uc, \piT\vc \rangle_{\Hbctm,\, \Hbct}.
\end{equation}
\end{theorem}

\begin{proof}
Since the optimal lift $\wc_f \in \Kc^{\perp_{\Hc^1}}$ of $\fc \in \Hbct$ is unique by Theorem~\ref{ThBT}, the right-hand side of~\eqref{EqDefTrTonBt} depends only on $\fc$ and not on any choice of representative.
Using the Cauchy-Schwarz inequality and the isometry~\eqref{EqIsomPiT}:
\begin{align*}
    |\langle \widetilde\TrT\uc, \fc \rangle_{\Hbctm,\, \Hbct}| &= |b(\uc, \wc_f)|\\
    &\leq \sqrt{2} \, \|\uc\|_{\Hcurl} \|\wc_f\|_{\Hc^1(\Omega)}\\
    &\leq \sqrt{2} \, \|\uc\|_{\Hcurl} \|\fc\|_{\Hbct}.
\end{align*}
We now establish equation~\eqref{EqbuvTrTpiT}. Let $\uc, \vc \in \Hc^1(\Omega)$ and write $\vc = \vc^0 + \vc^\perp$ with $\vc^0 \in \Kc$ and $\vc^\perp \in \Kc^{\perp_{\Hc^1}}$. By Theorem~\ref{ThBT}, $\vc^\perp$ is the optimal lift of $\piT\vc$, so by definition~\eqref{EqDefTrTonBt}:
\begin{equation*}
    \langle \TrT\uc, \piT\vc \rangle_{\Hbctm,\, \Hbct} = b(\uc, \vc^\perp).
\end{equation*}
Since $\vc^0 \in \Kc$ and $\uc \in \Hc^1(\Omega)$, property~\eqref{EqNullitybuv} gives $b(\uc, \vc^0) = 0$. Therefore:
\begin{equation*}
    b(\uc, \vc) = b(\uc, \vc^\perp) + b(\uc, \vc^0) 
    = b(\uc, \vc^\perp) 
    = \langle \TrT\uc, \piT\vc \rangle_{\Hbctm,\, \Hbct}. \qedhere
\end{equation*}
\end{proof}

Since both $\TrT$ and $\piT$ are isometries from the same space $\Kc^{\perp_{\Hc^1}}$ onto $\Hbcp$ and $\Hbct$ respectively, there is a natural abstract rotation linking them. This is the abstract analogue of the rotation operator $r$ of~\cite{BUFFA-2002}.

\begin{proposition}\label{PropAbstractRotation}
Let $\Omega$ be a bounded $H^1$-extension domain of $\R^3$. The map $R : \Hbct \to \Hbcp$ defined by $R(\fc) := \TrT(\wc_f)$, where $\wc_f \in \Kc^{\perp_{\Hc^1}}$ is the unique optimal lift of $\fc \in \Hbct$, is a linear isometric isomorphism. It satisfies:
\begin{enumerate}
    \item $R \circ \piT = \TrT$ on $\Hc^1(\Omega)$.
    \item $R$ satisfies the antisymmetry relation
    \begin{multline}
    \forall \fc, \gc \in \Hbct \subset \Hbcpm,\\
    \langle \gc,R(\fc) \rangle_{\Hbcpm,\,\Hbcp} = -\langle \fc,R(\gc) \rangle_{\Hbcpm,\,\Hbcp}.\label{Rf,g}
    \end{multline}
    In particular, $\langle \fc, R(\fc) \rangle_{\Hbcpm,\, \Hbcp} = 0$ for all $\fc \in \Hbct$.
    \item The adjoint $R^* : \Hbcpm \to \Hbctm$, restricted to $\Hbct \subset \Hbcpm$, satisfies $R^*|_{\Hbct} = -R$ as operators from $\Hbct$ to $\Hbctm$.
    \item The inverse $R^{-1} : \Hbcp \to \Hbct$ is given explicitly by 
    $$R^{-1}(\fc) = \piT(\wc_f).$$ 
    It satisfies $R^{-1} \circ \TrT = \piT$ on $\Hc^1(\Omega)$, and the antisymmetry relation
    \begin{multline}\label{EqRinvAntisym}
        \forall \fc, \gc \in \Hbcp,\\
        \langle R^{-1}(\fc), \gc \rangle_{\Hbcpm,\, \Hbcp} = -\langle R^{-1}(\gc), \fc \rangle_{\Hbcpm,\, \Hbcp}.
    \end{multline}
    \item The adjoint and inverse are related by
    \begin{equation}\label{EqRstarRinverse}
        R^* = \iota_T \circ R^{-1} \circ \iota_\perp^{-1},
    \end{equation}
    where $\iota_\perp : \Hbcp \to \Hbcpm$ and $\iota_T : \Hbct \to \Hbctm$ are the Riesz isomorphisms of $\Hbcp$ and $\Hbct$.
\end{enumerate}
\end{proposition}

\begin{proof}
We first show that $R$ is a well-defined isometric isomorphism from $\Hbct$ to $\Hbcp$.
For $\fc \in \Hbct$ with optimal lift $\wc_f \in \Kc^{\perp_{\Hc^1}}$, Theorems~\ref{ThBperp}
and~\ref{ThBT} give
\begin{equation*}
    \|R(\fc)\|_{\Hbcp}
    = \|\TrT\wc_f\|_{\Hbcp}
    = \|\wc_f\|_{\Hc^1(\Omega)}
    = \|\piT\wc_f\|_{\Hbct}
    = \|\fc\|_{\Hbct},
\end{equation*}
where the second and third equalities use the isometries~\eqref{EqIsomTrT} and~\eqref{EqIsomPiT}, together with the fact that the optimal lifts for $\TrT$ and $\piT$ coincide on $\Kc^{\perp_{\Hc^1}}$ by Theorem~\ref{ThBT}. This shows that $R$ is an isometry, hence injective. To see that $R$ is surjective, let $\fc \in \Hbcp$ with optimal lift $\wc_f \in \Kc^{\perp_{\Hc^1}}$, then $\piT\wc_f \in \Hbct$ and $R(\piT\wc_f) = \TrT\wc_f = \fc$.

We now prove Point~1. For $\uc \in \Hc^1(\Omega)$, write $\uc = \uc^0 + \uc^\perp$ with $\uc^0 \in \Kc$ and $\uc^\perp \in \Kc^{\perp_{\Hc^1}}$. Since $\uc^0 \in \Kc = \Ker(\piT) = \Ker(\TrT)$, we have $\piT\uc = \piT\uc^\perp$ and $\TrT\uc = \TrT\uc^\perp$. The optimal lift of $\piT\uc \in \Hbct$ is therefore $\uc^\perp$, and so $R(\piT\uc) = \TrT\uc^\perp = \TrT\uc$.

For Point~2, let $\fc, \gc \in \Hbct$ with optimal lifts $\wc_f, \wc_g \in \Kc^{\perp_{\Hc^1}}$. By the definition of $R$ and formula~\eqref{EqPiTuTrTv}, we have
\begin{align*}
    \langle \gc,R(\fc) \rangle_{\Hbcpm,\,\Hbcp}
    &= \langle \piT\wc_g, \TrT\wc_f \rangle_{\Hbcpm,\,\Hbcp}\\
    &= -\langle \piT\wc_f, \TrT\wc_g \rangle_{\Hbcpm,\,\Hbcp}\\
    &= -\langle \fc, R(\gc) \rangle_{\Hbcpm,\,\Hbcp}.
\end{align*}

To prove Point~3, we first observe that the compatibility between the two duality pairings, established by formulas~\eqref{EqPiTuTrTv} and~\eqref{EqbuvTrTpiT}, gives for all $\fc \in \Hbct \subset \Hbcpm$ and $\gc \in \Hbcp \subset \Hbctm$:
\begin{equation}\label{EqCompatHbcpHbct}
    \langle \fc, \gc \rangle_{\Hbcpm,\, \Hbcp} = \langle \gc, \fc \rangle_{\Hbctm,\, \Hbct}.
\end{equation}
The adjoint $R^* : \Hbcpm \to \Hbctm$ is defined by 
\begin{multline}
    \forall \fc \in \Hbcpm,\, \forall \gc \in \Hbct,\\
    \langle R^*(\fc), \gc \rangle_{\Hbctm,\, \Hbct} := \langle \fc, R(\gc) \rangle_{\Hbcpm,\, \Hbcp}.
\end{multline}
For $\fc,\gc \in \Hbct \subset \Hbcpm$, applying the antisymmetry of Point~2 and the compatibility formula~\eqref{EqCompatHbcpHbct} gives
\begin{align*}
    \langle R^*(\fc), \gc \rangle_{\Hbctm,\, \Hbct}
    &= \langle \fc, R(\gc) \rangle_{\Hbcpm,\, \Hbcp}\\
    &= -\langle \gc, R(\fc) \rangle_{\Hbcpm,\, \Hbcp}\\
    &= -\langle R(\fc), \gc \rangle_{\Hbctm,\, \Hbct},
\end{align*}
which gives $R^*|_{\Hbct} = -R$.

We now show Point~4. The formula $R^{-1}(\fc) = \piT(\wc_f)$ follows directly from the surjectivity argument above: we showed that $R(\piT\wc_f) = \fc$ for every $\fc \in \Hbcp$ with optimal lift $\wc_f$. The identity $R^{-1} \circ \TrT = \piT$ on $\Hc^1(\Omega)$ then follows from Point~1 by applying $R^{-1}$ on the left. For the antisymmetry~\eqref{EqRinvAntisym}, let $\fc, \gc \in \Hbcp$ with optimal lifts $\wc_f, \wc_g \in \Kc^{\perp_{\Hc^1}}$. By the definition of $R^{-1}$ and formula~\eqref{EqPiTuTrTv}:
\begin{align*}
    \langle R^{-1}(\fc), \gc \rangle_{\Hbcpm,\, \Hbcp}
    &= \langle \piT\wc_f, \TrT\wc_g \rangle_{\Hbcpm,\, \Hbcp}\\
    &= -\langle \piT\wc_g, \TrT\wc_f \rangle_{\Hbcpm,\, \Hbcp}\\
    &= -\langle R^{-1}(\gc), \fc \rangle_{\Hbcpm,\, \Hbcp}.
\end{align*}

We finally establish Point~5. Let $\fc \in \Hbcpm$ and $\gc \in \Hbct$, and set $\hc := \iota_\perp^{-1}\fc \in \Hbcp$ with optimal lift $\wc_h \in \Kc^{\perp_{\Hc^1}}$. By Point~4, $R^{-1}(\hc) = \piT(\wc_h)$, and since $\wc_h$ is the optimal lift of $\piT(\wc_h)$ under $\piT$ by Theorem~\ref{ThBT}, formula~\eqref{EqInnerBT} gives:
\begin{equation*}
    \langle \iota_T(R^{-1}(\hc)), \gc \rangle_{\Hbctm,\, \Hbct}
    = (R^{-1}(\hc), \gc)_{\Hbct}
    = (\wc_h, \wc_g)_{\Hc^1(\Omega)}.
\end{equation*}
By definition of $R$, $R(\gc) = \TrT(\wc_g)$, and $\wc_g$ is the optimal lift of $R(\gc)$ under $\TrT$ by Theorem~\ref{ThBT}. Formula~\eqref{EqInnerBperp} then gives:
\begin{align*}
    \langle R^*(\fc), \gc \rangle_{\Hbctm,\, \Hbct}
    &= \langle \fc, R(\gc) \rangle_{\Hbcpm,\, \Hbcp}\\
    &= (\hc, R(\gc))_{\Hbcp}\\
    &= (\wc_h, \wc_g)_{\Hc^1(\Omega)}.
\end{align*}
Since the two expressions coincide for all $\gc \in \Hbct$, we conclude that
\begin{equation*}
    R^* = \iota_T \circ R^{-1} \circ \iota_\perp^{-1}.\qedhere
\end{equation*}

\end{proof}

\begin{remark}\label{RemRotation}
In~\cite{BUFFA-2002}, the rotation $r : \Vc_\pi \to \Vc_\gamma$ satisfies the single identity $r^{-1} = r^* = -r$, which follows from the pointwise identity $(\nc \wedge \cdot)^2 = -\mathrm{id}$ on tangential fields and the fact that both $\Vc_\pi$ and $\Vc_\gamma$ sit inside the common ambient space $\Lc^2_t(\partial\Omega)$. In our abstract setting, the properties established in Proposition~\ref{PropAbstractRotation} recover this identity in the following split form. Point~3 gives $R^*|_{\Hbct} = -R$, which is the analogue of $r^* = -r$, and Point~5 recovers the Lipschitz identity $r^{-1} = r^*$ in the correct abstract form: rather than a direct equality, it takes the form $R^* = \iota_T \circ R^{-1} \circ \iota_\perp^{-1}$, where the two Riesz isomorphisms $\iota_\perp$ and $\iota_T$ play the role of the $\Lc^2_t(\partial\Omega)$ identification that is available in the Lipschitz case but absent here.
\end{remark}

We collect in the following proposition the structure of the trace space $\Hbc$ that emerges from the constructions of the two previous sections.

\begin{proposition}\label{PropDecompHbc}
Let $\Omega$ be a bounded $H^1$-extension domain of $\R^3$. 
Define the normal trace space
\begin{equation}\label{EqDefBn}
    \Hbcn := \Trc(\Kc) \subset \Hbc,
\end{equation}
where $\Kc := \Ker(\TrT) = \Ker(\piT)$ is the common kernel established in Proposition~\ref{PropSameKer}. Similarly, define the tangential trace space
\begin{equation}\label{EqDefBtan}
    \Hbctan := \Trc\bigl(\Kc^{\perp_{\Hc^1}}\bigr) \subset \Hbc.
\end{equation}
Both are Hilbert spaces: $\Hbcn$ is a closed subspace of $\Hbc$ equipped with the induced norm, while $\Hbctan$ is endowed with the inner product
\begin{equation}\label{EqInnerBtan}
    \forall \fc,\gc \in \Hbctan, \quad (\fc,\gc)_{\Hbctan} := (\wc_f, \wc_g)_{\Hc^1(\Omega)},
\end{equation}
where $\wc_f, \wc_g \in \Kc^{\perp_{\Hc^1}}$ are the unique optimal lifts of $\fc$ and $\gc$.\\
The following properties hold:
\begin{enumerate}
    \item The full trace space $\Hbc$ decomposes orthogonally as
    \begin{equation}\label{EqDecompHbc}
        \Hbc = \Hbcn \oplus \Hbctan.
    \end{equation}

    \item The maps $\Phi_T \colon \Hbctan \to \Hbct$ and $\Phi_\perp \colon \Hbctan \to \Hbcp$ defined for all $\wc \in \Kc^{\perp_{\Hc^1}}$ by
    \begin{equation}\label{EqTwoDescrip}
        \Phi_T(\Trc\wc) := \piT\wc \qquad\text{and}\qquad \Phi_\perp(\Trc\wc) := \TrT\wc,
    \end{equation}
    are well-defined isometric isomorphisms. In particular, $\Hbct$ and $\Hbcp$ are two isometric Hilbert-space descriptions of the same abstract space $\Hbctan$, and the inner product~\eqref{EqInnerBtan} coincides with both $(\cdot,\cdot)_{\Hbct}$ and $(\cdot,\cdot)_{\Hbcp}$.

    \item The rotation operator $R \colon \Hbct \to \Hbcp$ satisfies
    \begin{equation}\label{EqRlinksDescriptions}
        R = \Phi_\perp \circ \Phi_T^{-1}.
    \end{equation}
\end{enumerate}
\end{proposition}

\begin{proof} As a preamble we note that, if $\Trc\wc_1 = \Trc\wc_2$ for $\wc_1, \wc_2 \in \Kc^{\perp_{\Hc^1}}$, then $\wc_1 - \wc_2 \in \Hc^1_0(\Omega) \subset \Kc = \Ker(\piT) = \Ker(\TrT)$ by Lemma~\ref{InclH10TrtTrn}, so $\piT\wc_1 = \piT\wc_2$ and $\TrT\wc_1 = \TrT\wc_2$.\\
We first define the mapping $\Psi \colon \Hbc \to \Hbct$ by $\Psi(\Trc\vc) := \piT\vc^\perp$, with $\vc = \vc^0 + \vc^\perp$, $\vc^0 \in \Kc$ and $\vc^\perp \in \Kc^{\perp_{\Hc^1}}$. This map is well-defined because $\Ker(\textbf{Tr}) = \Hc^1_0(\Omega) \subset \Ker(\piT)$ (by Lemma~\ref{InclH10TrtTrn} and Proposition~\ref{PropSameKer}). Furthermore, $\Psi$ is continuous by continuity of $\piT \colon \Hc^1(\Omega) \to \Hbct$ and the definition of the quotient trace norm~\eqref{normTri}. By definition, $\Hbcn = \Ker(\Psi)$, making it a closed subspace of $\Hbc$ and thus a Hilbert space under the induced norm.\\
Regarding $\Hbctan$, recall from Theorem~\ref{ThBperp} that every $\wc \in \Kc^{\perp_{\Hc^1}}$ weakly satisfies $-\Delta\wc + \wc = \mathbf{0}$, meaning $\Kc^{\perp_{\Hc^1}}$ is a subspace of the 1-harmonic space $\boldsymbol{V}_1(\Omega)$. By the trace isometry of Theorem~\ref{ThTrisom}, the restriction $\Trc|_{\Kc^{\perp_{\Hc^1}}}$ is therefore an isometry. Consequently, the mapping $\Trc \colon \Kc^{\perp_{\Hc^1}} \to \Hbctan$ is an isometric isomorphism. This immediately endows the image space $\Hbctan$ with a Hilbert space structure and the inner product defined in~\eqref{EqInnerBtan}.

To prove Point 1, let $\zc_n \in \Hbcn$ and $\zc_t \in \Hbctan$. By definition of these spaces, there exist $\vc^0 \in \Kc$ and $\vc^\perp \in \Kc^{\perp_{\Hc^1}}$ such that $\zc_n = \Trc \vc^0$ and $\zc_t = \Trc \vc^\perp$. To compute their inner product in $\Hbc$ with formula~\eqref{DefHbInnerProd}, we must use their optimal $1$-harmonic lifts in $\Vc_1(\Omega)$. First, because $\Hc^1_0(\Omega) \subset \Ker(\TrT) = \Kc$, taking the orthogonal complement in $\Hc^1(\Omega)$ yields $\Kc^{\perp_{\Hc^1}} \subset \Vc_1(\Omega)$. Therefore, $\vc^\perp$ is already the optimal lift of $\zc_t$. 
Second, let $\wc_n \in \Vc_1(\Omega)$ be the optimal lift of $\zc_n$. Since $\vc^0$ and $\wc_n$ have the same full trace $\zc_n$, their difference satisfies $\vc^0 - \wc_n \in \Hc^1_0(\Omega) \subset \Kc$ by Lemma~\ref{InclH10TrtTrn}. Because $\vc^0 \in \Kc$, we deduce that $\wc_n \in \Kc$ as well.\\
Using definition~\eqref{DefHbInnerProd} of the $\Hbc$ inner product and the orthogonality between $\Kc$ and $\Kc^{\perp_{\Hc^1}}$ in $\Hc^1(\Omega)$, we obtain
\begin{equation*}
    (\zc_n, \zc_t)_{\Hbc} = (\wc_n, \vc^\perp)_{\Hc^1(\Omega)} = 0,
\end{equation*}
which establishes the desired orthogonality $\Hbcn \perp \Hbctan$.

We now prove Point~2.  
For $\vc^\perp \in \Kc^{\perp_{\Hc^1}} \subset \Vc_1(\Omega)$, by isometries~\eqref{EqBnorm},~\eqref{EqIsomTrT} and~\eqref{EqIsomPiT} we find
\begin{equation*}
    \|\Trc\vc^\perp\|_{\Hbc} = \|\vc^\perp\|_{\Hc^1}
    = \|\TrT\vc^\perp\|_{\Hbcp} = \|\piT\vc^\perp\|_{\Hbct},
\end{equation*}
so both $\Phi_T$ and $\Phi_\perp$ are isometries. Surjectivity follows since every $\fc \in \Hbct$ (resp. $\Hbcp$) has an optimal lift $\wc_f \in \Kc^{\perp_{\Hc^1}}$ with $\fc = \piT \wc_f = \Phi_T(\Trc\wc_f)$ (resp. $\fc = \TrT \wc_f = \Phi_\perp(\Trc\wc_f)$). The fact that $\Hbct$ and $\Hbcp$ are two isometric descriptions of $\Hbctan$ is now immediate: $\Phi_T$ and $\Phi_\perp$ are both surjective isometries from $\Hbctan$ onto $\Hbct$ and $\Hbcp$ respectively, so all three spaces are isometrically isomorphic. For the inner products, let $\Trc\wc_f, \Trc\wc_g \in \Hbctan$ with $\wc_f, \wc_g \in \Kc^{\perp_{\Hc^1}}$. By definition~\eqref{EqInnerBtan} and the inner products~\eqref{EqInnerBT} and~\eqref{EqInnerBperp}:
\begin{align*}
    (\Trc\wc_f, \Trc\wc_g)_{\Hbctan}  &= (\wc_f, \wc_g)_{\Hc^1(\Omega)} = (\piT\wc_f, \piT\wc_g)_{\Hbct}\\
    &= (\Phi_T(\Trc\wc_f), \Phi_T(\Trc\wc_g))_{\Hbct}
\end{align*}
and
\begin{align*}
    (\Trc\wc_f, \Trc\wc_g)_{\Hbctan} &= (\wc_f, \wc_g)_{\Hc^1(\Omega)} = (\TrT\wc_f, \TrT\wc_g)_{\Hbcp}\\
    &= (\Phi_\perp(\Trc\wc_f), \Phi_\perp(\Trc\wc_g))_{\Hbcp}.
\end{align*}

Finally, we prove Point~3. For any $\fc \in \Hbct$ with optimal lift $\wc_f \in \Kc^{\perp_{\Hc^1}}$:
\begin{equation*}
    (\Phi_\perp \circ \Phi_T^{-1})(\fc) = \Phi_\perp(\Trc\wc_f) = \TrT\wc_f = R(\fc).\qedhere
\end{equation*}
\end{proof}

\begin{remark}\label{RemLipschitzFrame}
When $\Omega$ is a Lipschitz domain, formulas~(12)--(13) of~\cite{BUFFA-2002} can be recovered directly from our abstract framework. Writing $(\Trc\uc)_T = u_1\tilde{\boldsymbol\tau}_1 + u_2\tilde{\boldsymbol\tau}_2$ in the local orthonormal frame $(\tilde{\boldsymbol\tau}_1,\tilde{\boldsymbol\tau}_2)$, the geometric formulas $\piT : \vc\mapsto\nc\times(\Trc\vc\times\nc)$ and $\TrT : \vc\mapsto\Trc\vc\times\nc$ give:
\begin{align*}
    \piT(\uc) = u_1\tilde{\boldsymbol\tau}_1 + u_2\tilde{\boldsymbol\tau}_2,
    \qquad
    \TrT(\uc) = u_2\tilde{\boldsymbol\tau}_1 - u_1\tilde{\boldsymbol\tau}_2,
\end{align*}
and $R = \Phi_\perp \circ \Phi_T^{-1}$ is the $90^\circ$ rotation $(u_1,u_2)\mapsto(u_2,-u_1)$ in the tangential plane. The key feature of the Lipschitz case is that the local frame $(\tilde{\boldsymbol\tau}_1,\tilde{\boldsymbol\tau}_2)$ allows every element $\Trc\wc \in \Hbc$ to be split in the tangential plane into two independent scalar components $(\Trc\wc\cdot\tilde{\boldsymbol\tau}_1, \Trc\wc\cdot\tilde{\boldsymbol\tau}_2)$, so that $\Hbct$ and $\Hbcp$ are simply two different readings of the same pair $(u_1,u_2)$.

In the fractal case, $\partial\Omega$ has no tangent bundle and no local frame exists, so no pointwise splitting into scalar components $(u_1,u_2)$ is available. The maps $\Phi_T : \Hbctan \to \Hbct$ and $\Phi_\perp : \Hbctan \to \Hbcp$ of Proposition~\ref{PropDecompHbc} are the correct abstract substitutes: rather than two scalar readings of a pair of components, they provide two isometric Hilbert space descriptions of the single abstract space $\Hbctan$, encoding the same algebraic structure without any local geometric structure on $\partial\Omega$.
\end{remark}

To conclude this section, we present in Figure~\ref{fig:tangential_spaces} a diagram summarizing the tangential spaces introduced so far and the relationships between them.

\begin{figure}[htbp]
\centering
\begin{tikzpicture}[
    box/.style   = {draw, rounded corners=4pt, minimum width=3cm,
                    minimum height=1.1cm, align=center, thick},
    teal/.style  = {box, fill=teal!10,   draw=teal!50},
    purp/.style  = {box, fill=violet!10, draw=violet!40},
    coral/.style = {box, fill=orange!10, draw=orange!40},
    sol/.style   = {->, >=stealth, thin},
    das/.style   = {->, >=stealth, thin, dashed},
    lbl/.style   = {font=\small},
]

\node[teal]  (Btan) at (0, 0)    {$\Hbctan$\\{\footnotesize $=\Trc(\Kc^\perp)$}};

\node[coral]  (BT)   at (-3.5,-2.5){$\Hbct$\\{\footnotesize image of $\piT$}};
\node[coral] (Bp)   at ( 3.5,-2.5){$\Hbcp$\\{\footnotesize image of $\TrT$}};

\node[purp]  (BTd)  at (-3.5,-5) {$\Hbctm$\\{\footnotesize dual of $\Hbct$}};
\node[purp]  (Bpd)  at ( 3.5,-5) {$\Hbcpm$\\{\footnotesize dual of $\Hbcp$}};

\draw[sol, bend right=15] (Btan) to node[lbl,above left]  {$\Phi_T$}   (BT);
\draw[sol, bend right=15] (BT)   to node[lbl,left] {$\Phi_T^{-1}$}   (Btan);

\draw[sol, bend left=15]  (Btan) to node[lbl,above right] {$\Phi_\perp$}  (Bp);
\draw[sol, bend left=15]  (Bp)   to node[lbl,right]  {$\Phi_\perp^{-1}$}  (Btan);

\draw[sol, bend left=20]  (BT)  to node[lbl,above] {$R$}      (Bp);
\draw[sol, bend left=20]  (Bp)  to node[lbl,above] {$R^{-1}$} (BT);

\draw[sol, bend right=12] (BT)  to node[lbl,left]  {$\iota_T$}     (BTd);
\draw[sol, bend right=12] (BTd) to node[lbl,right] {$\iota_T^{-1}$}(BT);

\draw[sol, bend left=12]  (Bp)  to node[lbl,right] {$\iota_\perp$}     (Bpd);
\draw[sol, bend left=12]  (Bpd) to node[lbl,left]  {$\iota_\perp^{-1}$}(Bp);

\draw[sol, bend left=20] (Bpd) to
  node[lbl,below] {$R^* = \iota_T \circ R^{-1} \circ \iota_\perp^{-1}$} (BTd);

\draw[sol, bend left=20] (BTd) to
  node[lbl,below] {$(R^*)^{-1}$} (Bpd);

\end{tikzpicture}
\caption{The five abstract tangential trace spaces and all operators between them}
\label{fig:tangential_spaces}
\end{figure}

\subsection{\texorpdfstring{$\Hcurl$-framework for $\widetilde\piT$}{H(curl)-framework for piT tilde}}\label{SubsecPiTHcurl}

Definition~\ref{DefPiTH1} defines $\piT\uc$ for $\uc \in \Hc^1(\Omega)$ using the optimal lift $\wc_f$ as a canonical representative. However, as noted after Definition~\ref{DefPiTH1}, the choice of $\wc_f$ was not mandatory there: property~\eqref{EqNullitybuv} guaranteed that any pre-image $\vc \in \Hc^1(\Omega)$ with $\TrT\vc = \fc$ gives the same value $-b(\uc, \vc)$.

This is no longer true for $\uc \in \Hcurl$, for two related reasons. The vanishing $b(\uc, \vc^0) = 0$ for $\vc^0 \in \Kc$, which in the $\Hc^1(\Omega)$-framework followed from~\eqref{EqNullitybuv} through the trace $\Trc\uc$, is no longer available, since $\Trc\uc$ is not defined for $\uc \in \Hcurl$. The form $b$ remains antisymmetric on $\Hcurl \times \Hcurl$, but this alone does not give $b(\uc, \vc^0) = 0$, and indeed $b(\uc, \vc^0) = \langle \widetilde\TrT\uc, \Trc\vc^0 \rangle_{\Hbcm,\, \Hbc}$ may not vanish, since $\Trc\vc^0 \neq \mathbf{0}$ in general. As a consequence, two different lifts of $\fc$ may yield different values of $b(\uc, \cdot)$, and the formula is ambiguous unless one fixes a specific representative. The optimal lift $\wc_f \in \Kc^{\perp_{\Hc^1}}$ is the canonical choice that removes this ambiguity, and its use is now mandatory rather than merely convenient.

\begin{theorem}\label{ThmPiTHcurl}
Let $\Omega$ be a bounded $H^1$-extension domain of $\R^3$. The operator $\widetilde\piT : \Hcurl \to \Hbcpm$ defined for all $\uc \in \Hcurl$ by
\begin{equation}\label{EqDefPiTHcurl}
    \forall \fc \in \Hbcp, \quad
    \langle \widetilde\piT\uc, \fc \rangle_{\Hbcpm,\, \Hbcp} := -b(\uc, \wc_f),
\end{equation}
where $\wc_f \in \Kc^{\perp_{\Hc^1}}$ is the unique optimal lift of $\fc \in \Hbcp$ from Theorem~\ref{ThBperp}, is well-defined, linear, and continuous with
\begin{equation*}
    \|\widetilde\piT\uc\|_{\Hbcpm} \leq \sqrt{2} \, \|\uc\|_{\Hcurl}.
\end{equation*}
It extends $\piT$: for all $\uc \in \Hc^1(\Omega)$, $\widetilde\piT\uc = \piT\uc$. Moreover, for all $\uc \in \Hcurl$ and $\vc \in \Hc^1(\Omega)$,
\begin{equation}\label{EqIPPbuvCurl}
    \langle \widetilde\piT\uc, \TrT\vc \rangle_{\Hbcpm,\, \Hbcp}
    = -\langle \widetilde\TrT\uc, \piT\vc \rangle_{\Hbctm,\, \Hbct}.
\end{equation}
\end{theorem}

\begin{proof}
Linearity and continuity follow by the same argument as in Proposition~\ref{PropPiTH1}, working with $\|\uc\|_{\Hcurl}$ instead of $\|\uc\|_{\Hc^1(\Omega)}$ in the Cauchy-Schwarz estimate. The extension property $\widetilde\piT|_{\Hc^1(\Omega)} = \piT$ is immediate from the fact that~\eqref{EqDefPiTHcurl} and~\eqref{EqDefPiTH1} are the same formula.

For formula~\eqref{EqIPPbuvCurl}, write $\vc = \vc^\perp + (\vc - \vc^\perp)$ with $\vc^\perp \in \Kc^{\perp_{\Hc^1}}$ the optimal lift of $\TrT\vc$. By linearity of $b$:
\begin{equation*}
    b(\uc, \vc) = -\langle \widetilde\piT\uc, \TrT\vc \rangle_{\Hbcpm,\, \Hbcp}
    + b(\uc, \vc) - b(\uc, \vc^\perp).
\end{equation*}
Canceling $b(\uc, \vc)$ gives $\langle\widetilde\piT\uc, \TrT\vc\rangle_{\Hbcpm,\, \Hbcp} = -b(\uc, \vc^\perp)$. By definition~\eqref{EqDefTrTonBt} applied to $\gc = \piT\vc$ with optimal lift $\vc^\perp$: 
$$b(\uc, \vc^\perp) = \langle\widetilde\TrT\uc, \piT\vc\rangle_{\Hbctm,\, \Hbct},$$ 
which gives~\eqref{EqIPPbuvCurl}.
\end{proof}

The kernel of $\widetilde\piT$ deserves particular attention: does the property $ \Ker(\widetilde\piT) =  \Hcurlz$ still hold for a bounded, Stokes-regular $H^1$-extension domain $\Omega$? 

First, by definition~\eqref{EqDefPiTHcurl}, $\widetilde\piT\uc = \mathbf{0}$ means $b(\uc, \wc) = 0$ for all $\wc \in \Kc^{\perp_{\Hc^1}}$, which by definition~\eqref{EqDefTrTonBt} is equivalent to $\widetilde\TrT\uc = 0$ in $\Hbctm$. Under Assumption~\ref{AssumpHcurlz3D}, we know that $\Ker(\widetilde\TrT) = \Hcurlz$ by Proposition~\ref{propDensKerTrT}, with $\widetilde\TrT\,:\,\Hcurl\to\Hbcm$. The condition $\widetilde\TrT\uc = 0$ in $\Hbctm$ is a priori weaker than $\widetilde\TrT\uc = 0$ in $\Hbcm$: the former only tests $\widetilde\TrT\uc$ against elements of $\Hbctan = \Trc(\Kc^{\perp_{\Hc^1}})$, while the latter tests against all elements of $\Hbc = \Trc(\Hc^1(\Omega))$. Hence, one can only infer that $\Hcurlsz \subset \Ker(\widetilde\piT)$. 

Then, to prove the converse inclusion, one has to check that
$$\forall \vc^0 \in \Kc, \quad b(\uc, \vc^0) = 0.$$ In the $\Hc^1(\Omega)$-framework, this condition followed from the antisymmetry of $b$ and the fact that $\Trc\uc$ was available. However, for $\uc \in \Hcurl$, $\Trc\uc$ is no longer defined and this argument breaks down. The key to resolving this difficulty is the density of $\Hc^1(\Omega)$ in $\Hcurl$ under Assumption~\ref{AssumpHcurlz3D}. This is made precise in the following lemma.

\begin{lemma}\label{LembuvHcurl}
Let $\Omega$ be a bounded, Stokes-regular $H^1$-extension domain of $\R^3$.
Then for all $\uc \in \Hcurl$ and all $\vc^0 \in \Kc$:
\begin{equation}\label{EqbuvHcurl}
    b(\uc, \vc^0) = 0.
\end{equation}
\end{lemma}

\begin{proof}
Let $\vc^0 \in \Kc = \Ker(\TrT)$. By the Stokes formula~\eqref{EqStokesTrT}:
\begin{equation*}
    \forall \wc \in \Hc^1(\Omega), \quad b(\vc^0, \wc) = \langle \TrT\vc^0, \Trc\wc \rangle_{\Hbcm, \Hbc} = 0.
\end{equation*}
Since $\vc^0 \in \Hc^1(\Omega) \subset \Hcurl$, the map $\uc \mapsto b(\vc^0, \uc)$ is continuous on $\Hcurl$ by Theorem~\ref{ThGStokesF}. Under Assumption~\ref{AssumpHcurlz3D}, Proposition~\ref{DensHrot} implies that $\Hc^1(\Omega)$ is dense in $\Hcurl$. Furthermore, equation~\eqref{EqNullitybuv} gives
\[
\forall \wc \in \Hc^1(\Omega), \quad b(\vc^0,\wc)=0.
\]
By continuity, it follows that
\[
\forall \uc \in \Hcurl, \quad b(\vc^0,\uc)=0.
\]
Finally, using the antisymmetry of $b$, we obtain
\[
b(\uc,\vc^0) = -b(\vc^0,\uc)=0.\qedhere
\]
\end{proof}

As a straightforward consequence of this lemma, we can
characterize the kernel of $\widetilde{\piT}$ as follows.

\begin{corollary}\label{propDensKerpiT}
Let $\Omega$ be a bounded, Stokes-regular $H^1$-extension domain of $\R^3$. Then
\begin{equation}\label{HcurlzKerpiT}
  \Hcurlz = \Ker(\widetilde{\piT}).
\end{equation}
\end{corollary}

\section{Tangential operators}\label{SecTangOp}

In this section, we aim to generalize the tangential operators $\nabla_T$, $\divT$, $\curlcT$ and $\curlT$ from the Lipschitz case (see~\cite{BUFFA-2002} and~\cite{Ciarlet_2024} for details). By the identity $\rotv \nabla \cdot = 0$, we know that for any $v\in H^1(\Omega)$, the gradient $\nabla v$ belongs to $\Hcurl$. This motivates the use of the $\Hcurl$ version of the trace operators in a first step. Then, in a similar way as in~\cite[Section 3]{BUFFA-2002}, we define alternative versions of $\nabla_T$ and $\curlcT$ by restriction, allowing for a more precise definition of their adjoints.

\subsection{\texorpdfstring{$\Hcurl$-framework}{H(curl,Omega)-framework}}

\subsubsection{Tangential gradient and tangential divergence}\label{subsec_Hcurl_tang_grad+div}

To define the tangential gradient $\widetilde{\nabla_T}$, we adapt~\cite[Theorem 2.22]{Ciarlet_2024}. We recall that the tilde notation indicates the $\Hcurl$-based versions of the tangential operators, namely $\widetilde{\TrT}$ and $\widetilde{\piT}$; see Theorems~\ref{ThmTrTtoBtm} and~\ref{ThmPiTHcurl}.

\begin{theorem}\label{ThmNablaTtilde}
Let $\Omega$ be a bounded $H^1$-extension domain of $\R^3$. Then the tangential gradient operator
\begin{align*}
    \widetilde{\nabla_T} \colon \Hb &\to \Hbcpm\\
    \Tr v & \mapsto \widetilde{\nabla_T}(\Tr v) = \widetilde{\piT}(\nabla v)
\end{align*}
is a well-defined, linear and continuous mapping.
\end{theorem}

\begin{proof}
Let us show that this is well-defined. First, using the identity $\rotv \nabla \cdot = 0$, we observe that for any $v\in H^1(\Omega)$, $\nabla v \in \Hcurl$, so $\widetilde{\piT}(\nabla v)$ is well-defined. We now need to show that this definition does not depend on the choice of representative. Let $v_1, v_2 \in H^1(\Omega)$ be such that $\Tr v_1 = \Tr v_2$, and define $v := v_1 - v_2 \in H^1_0(\Omega)$. We have for any $\fc \in \Hbcp$:
\begin{align*}
    \langle \widetilde{\nabla_T}(\Tr v), \fc \rangle_{\Hbcpm,\,\Hbcp} &= \langle \widetilde{\piT}(\nabla v), \fc \rangle_{\Hbcpm,\,\Hbcp}\\
    &= - \int_\Omega \nabla v \cdot \rotv \wc_f \dx \text{ by formula~\eqref{EqDefPiTHcurl}}\\
    &= \int_\Omega v \, \divnesp(\rotv \wc_f) \dx
    \text{ by formula~\eqref{EqGreenIntBis}} \\
    &= 0,
\end{align*}
so $\widetilde{\nabla_T}(\Tr v_1) = \widetilde{\nabla_T}(\Tr v_2)$.\\
The linearity is an immediate consequence of the linearity of $\widetilde{\piT}$ and $\nabla$. Thus, it remains to prove   continuity. By Theorem~\ref{ThmPiTHcurl}, we have
\begin{equation*}
    \forall u \in H^1(\Omega), \|\widetilde{\piT}(\nabla u)\|_{\Hbcpm} \leq \sqrt{2} \|\nabla u\|_{\Hcurl} = \sqrt{2} \|\nabla u\|_{\Lc^2(\Omega)} \leq \sqrt{2} \|u\|_{H^1(\Omega)}.
\end{equation*}
Hence, for all $u \in H^1(\Omega)$, we have
\begin{equation*}
    \|\widetilde{\nabla_T} (\Tr u)\|_{\Hbcpm} \leq \sqrt{2} \|u\|_{H^1(\Omega)},
\end{equation*}
which implies by~\eqref{normTri} that
\begin{equation}
    \|\widetilde{\nabla_T} (\Tr u)\|_{\Hbcpm} \leq \sqrt{2} \|\Tr u\|_{\Hb}.
\end{equation}
This proves the continuity of $\widetilde{\nabla_T}$.
\end{proof}

Since $\widetilde{\nabla_T}$ is a bounded linear operator, we can define its adjoint, which is also linear and bounded.
A similar construction appears in~\cite[Proposition 3.6]{BUFFA-2002}.

\begin{definition}\label{DefDivTtilde}
Let $\Omega$ be a bounded $H^1$-extension domain of $\R^3$. We define the tangential divergence operator $\widetilde{\divT} : \Hbcp \to \Hbm$ as the adjoint of $-\widetilde{\nabla_T}$. It is a linear and bounded operator satisfying
\begin{equation*}
    \forall \fc\in \Hbcp,\, \forall z \in \Hb, \langle \widetilde\divT\fc, z\rangle_{\Hbm,\, \Hb} = -\langle \widetilde{\nabla_T} z, \fc\rangle_{\Hbcpm,\, \Hbcp}.
\end{equation*}
\end{definition}

We now provide a formula linking the normal trace and tangential divergence operators for elements of $\Hc^1(\Omega)$, as in~\cite[Theorem 2.26]{Ciarlet_2024}.

\begin{proposition}\label{PropTrnCurlDivTtilde}
Let $\Omega$ be a bounded $H^1$-extension domain of $\R^3$. For all $\vc \in \Hc^1(\Omega)$, the following equality holds in $\Hbm$:
\begin{equation}\label{EqTrnCurlDivTtilde}
    \trn (\rotv \vc) = \widetilde{\divT} (\TrT \vc). 
\end{equation}
\end{proposition}

\begin{proof}
Let $\vc \in \Hc^1(\Omega)$ and $z = \Tr u \in \Hb$.
Since $\rotv\vc \in \Hdiv$ (as $\div\rotv\vc = 0$), the Green
formula~\eqref{EqItByPn} gives
\begin{equation*}
    \langle \trn(\rotv\vc), z \rangle_{\Hbm,\, \Hb}
    = \int_\Omega \rotv\vc \cdot \nabla u \dx
    = b(\nabla u, \vc),
\end{equation*}
where the last equality uses $\rotvnesp(\nabla u) = \mathbf{0}$.

Decompose $\vc = \vc^0 + \vc^\perp$ with $\vc^0 \in \Kc$ and $\vc^\perp \in \Kc^{\perp_{\Hc^1}}$ the optimal lift of $\TrT\vc$. We claim that $b(\nabla u, \vc^0) = 0$. Indeed, for any $u' \in H^2(\Omega)$, one has $\nabla u' \in \Hc^1(\Omega)$, so property~\eqref{EqNullitybuv} gives $b(\nabla u', \vc^0) = 0$, i.e.
$$\langle \trn(\rotv\vc^0), \Tr u' \rangle_{\Hbm,\, \Hb} = 0.$$ 
Since $\{\Tr u' : u' \in H^2(\Omega)\}$ is dense in $\Hb$ (as $\Omega$ is an $H^1$-extension domain), this gives $\trn(\rotv\vc^0) = 0$ in $\Hbm$, and therefore for all $u \in H^1(\Omega)$ we have 
$$b(\nabla u, \vc^0) = \langle \trn(\rotv\vc^0), \Tr u \rangle_{\Hbm,\, \Hb} = 0.$$
Hence $b(\nabla u, \vc) = b(\nabla u, \vc^\perp)$. First, by definition~\eqref{EqDefPiTHcurl} of $\widetilde\piT$, then by definition of $\widetilde{\nabla_T}$, resp. $\widetilde\divT$:
\begin{align*}
    b(\nabla u, \vc^\perp)
    &= -\langle \widetilde\piT(\nabla u), \TrT\vc \rangle_{\Hbcpm,\, \Hbcp}\\
    &= -\langle \widetilde\nabla_T z, \TrT\vc \rangle_{\Hbcpm,\, \Hbcp}\\
    &= \langle \widetilde\divT(\TrT\vc), z\rangle_{\Hbm,\, \Hb}. \tag*{\qedhere}
\end{align*}
\end{proof}

\begin{remark}\label{RemExtensionHcurl}
For a general field $\boldsymbol v\in \Hcurl$, Theorem 5.7 gives $\widetilde{\TrT}\boldsymbol v\in\boldsymbol B_T'(\partial\Omega)$. In general this functional need not belong to $\boldsymbol B_\perp(\partial\Omega)$, which is the domain of $\widetilde{\divT}$. Hence the expression $\widetilde{\divT}(\TrT\boldsymbol v)$ in~\eqref{EqTrnCurlDivTtilde} is defined a priori only for $\boldsymbol v\in\boldsymbol H^1(\Omega)$. Under the Stokes-regularity assumption, Proposition~\ref{DensHrot} implies that $\boldsymbol H^1(\Omega)$ is dense in $\Hcurl$. Therefore the map
\[
\boldsymbol v\longmapsto
\widetilde{\divT}(\TrT\boldsymbol v),
\qquad \boldsymbol v\in\boldsymbol H^1(\Omega),
\]
has the unique continuous extension
\[
\boldsymbol v\longmapsto
\trn(\rotv\boldsymbol v)
\]
to $\Hcurl$.
\end{remark}

\subsubsection{Tangential curl operators}

The definitions of $\widetilde\curlcT$ and $\widetilde\curlT$ are analogous to those of $\widetilde{\nabla_T}$ and $\widetilde\divT$, but use the tangential trace $\widetilde{\TrT}$ instead of $\widetilde{\piT}$.

\begin{theorem}\label{ThmCurlcTtilde}
Let $\Omega$ be a bounded $H^1$-extension domain of $\R^3$. Then the tangential curl operator
\begin{align*}
    \widetilde\curlcT \colon \Hb &\to \Hbctm\\
    \Tr v & \mapsto \widetilde\curlcT(\Tr v) = \widetilde{\TrT}(\nabla v)
\end{align*}
is a well-defined, linear and continuous mapping.
\end{theorem}

\begin{proof}
The proof follows the same lines as that of Theorem~\ref{ThmNablaTtilde}, with $\Hbctm$ in place of $\Hbcpm$, and relies on formula~\eqref{EqDefTrTonBt} together with Theorem~\ref{ThmTrTtoBtm}.
\end{proof}

Since $\widetilde\curlcT$ is a bounded linear operator, we can define its adjoint, which is also linear and bounded. A similar construction is presented in~\cite[Proposition 3.6]{BUFFA-2002}.

\begin{definition}\label{DefCurlTtilde}
Let $\Omega$ be a bounded $H^1$-extension domain of $\R^3$. We define the scalar tangential curl operator $\widetilde\curlT : \Hbct \to \Hbm$ as the adjoint of $\widetilde\curlcT$. It is a linear and bounded operator satisfying
\begin{equation*}
    \forall \fc \in \Hbct,\, \forall z \in \Hb, \langle\widetilde\curlT\fc, z\rangle_{\Hbm,\, \Hb} = \langle\widetilde\curlcT z, \fc\rangle_{\Hbctm,\, \Hbct}.
\end{equation*}
\end{definition}

We now provide a formula linking the normal trace and the scalar tangential curl operators. Here again, we must consider an element of $\Hc^1(\Omega)$.

\begin{proposition}\label{PropTrnCurlCurlTtilde}
Let $\Omega$ be a bounded $H^1$-extension domain of $\R^3$. For all $\vc \in \Hc^1(\Omega)$, the following equality holds in $\Hbm$:
\begin{equation}\label{EqTrnCurlCurlTtilde}
    \trn (\rotv \vc) = \widetilde\curlT (\piT \vc). 
\end{equation}
\end{proposition}

\begin{proof}
The proof is similar to that of Proposition~\ref{PropTrnCurlDivTtilde}, using equation~\eqref{EqDefTrTonBt} instead of~\eqref{EqDefPiTHcurl}.
\end{proof}

\subsection{\texorpdfstring{$\Hc^1(\Omega)$-framework}{H1(Omega)-framework}}

Following the approach of~\cite[Proposition 3.4]{BUFFA-2002} (see also~\cite[Remark 2.23]{Ciarlet_2024}), we now consider appropriate restrictions of the tangential operators $\widetilde{\nabla_T}$, $\widetilde\curlcT$ and their adjoints. To this end, throughout this section we assume that $\Omega$ is an $H^1,H^2$-extension domain, by which we mean that $\Omega$ is both an $H^1$-extension domain and an $H^2$-extension domain,   see~\cite[Theorem~5]{HAJLASZ-2008}) for a characterization. In particular, there exists a bounded linear extension operator $E_2 : H^2(\Omega) \longrightarrow H^2(\R^n)$. This assumption is satisfied, for instance, by all $(\varepsilon,\delta)$-domains (see~\cite[Theorem~1]{JONES-1981} and~\cite[Theorem~8]{ROGERS-2006}). Under this hypothesis, we can identify $H^2(\Omega)=W^{2,2}(\Omega)$ with the space $H^{2,2}(\Omega)$ introduced in~\cite[Section~2]{HINZ-2023}. We then introduce the following space:
\begin{equation}\label{DefCpartialOmega}
    \Cb := \Tr[H^2(\Omega)],
\end{equation}
which is a Hilbert space endowed with the norm
\begin{equation}\label{normTrH2}
\left\|f\right\|_{\Cb}:=\min\{ \left\|v\right\|_{H^2(\Omega)}\mid f=\mathrm{Tr}\:v\}
\end{equation}
(see~\cite[Section 2]{HINZ-2023} for details).

\subsubsection{Tangential gradient and tangential divergence}

Referring back to the introductory paragraph of Section~\ref{subsec_Hcurl_tang_grad+div}, we observe that using the $\Hc^1(\Omega)$ version of $\piT$ requires $\nabla v$ in $\Hc^1(\Omega)$. This motivates restricting $\widetilde{\nabla_T}$ to $\Cb$. This approach is consistent with the regular case, where restrictions from $H^{3/2}(\partial\Omega)$, i.e. the space of traces of elements of $H^2(\Omega)$, are typically considered. Since we now use the $\Hc^1(\Omega)$ version of $\piT$, we can directly restrict the codomain to $\Hbct$.

\begin{theorem}\label{ThmNablaT}
Let $\Omega$ be a bounded $H^1,H^2$-extension domain of $\R^3$. Then the restricted tangential gradient operator
\begin{align*}
    \nabla_T \colon \Cb &\to \Hbct\\
    \Tr v & \mapsto \nabla_T(\Tr v) = \piT(\nabla v)
\end{align*}
is a well-defined, linear and continuous mapping.
\end{theorem}

\begin{proof}
Let $v \in H^2(\Omega)$. Since $\nabla v \in \Hc^1(\Omega)$, $\piT(\nabla v)$ is well-defined and belongs to $\Hbct$ by Definition~\ref{DefBT}. Independence of the representative is proved exactly as in Theorem~\ref{ThmNablaTtilde}, which applies since $H^2(\Omega) \subset H^1(\Omega)$ and $\Hbct \subset \Hbcpm$ (see (\ref{EqDefBT})).
For continuity, the quotient norm~\eqref{EqQuotNormBT} gives
\begin{equation*}
    \|\nabla_T(\Tr v)\|_{\Hbct} = \|\piT(\nabla v)\|_{\Hbct} \leq \|\nabla v\|_{\Hc^1(\Omega)} \leq \|v\|_{H^2(\Omega)},
\end{equation*}
by~\eqref{normTrH2}, $\|\nabla_T(\Tr v)\|_{\Hbct} \leq \|\Tr v\|_{\Cb}$.
\end{proof}

Once again, since $\nabla_T$ is a bounded linear operator, we can define its adjoint, which is also linear and bounded. This gives the $\Hc^1(\Omega)$ version (or restricted version) of the tangential divergence.

\begin{definition}\label{DefDivT}
Let $\Omega$ be a bounded $H^1,H^2$-extension domain of $\R^3$. We define the restricted tangential divergence operator $\divT : \Hbctm \to \Cbm$ as the adjoint of $-\nabla_T$. It is a linear and bounded operator satisfying
\begin{multline}
    \forall \fc \in \Hbctm,\, \forall z \in \Cb,\\
    \langle \divT\fc, z \rangle_{\Cbm,\, \Cb} = -\langle \fc, \nabla_T z \rangle_{\Hbctm,\, \Hbct}.
\end{multline}
\end{definition}

With this definition, we provide an alternative version of Proposition~\ref{PropTrnCurlDivTtilde}.

\begin{proposition}\label{PropTrnCurlDivT}
Let $\Omega$ be a bounded $H^1,H^2$-extension domain of $\R^3$. For all $\vc \in \Hc^1(\Omega)$, the following equality holds in $\Cbm$:
\begin{equation}\label{EqTrnCurlDivT}
    \trn (\rotv \vc) = \divT (\TrT \vc). 
\end{equation}
\end{proposition}

\begin{proof}
Let $\vc \in \Hc^1(\Omega)$ and $z = \Tr u \in \Cb$, so $u \in H^2(\Omega) \subset H^1(\Omega)$. Since $\div\rotv\vc = 0$, we have $\rotv\vc \in \Hdiv$ and $\trn(\rotv\vc)$ is well-defined. By the Green formula~\eqref{EqItByPn}:
\begin{equation*}
    \langle \trn(\rotv\vc), z \rangle_{\Hbm,\, \Hb}
     = b(\nabla u, \vc).
\end{equation*}
By definition of $\divT$ and $\nabla_T$, and by equation~\eqref{EqbuvTrTpiT} ,
one has
\begin{align*}
    \langle \divT(\TrT\vc), z \rangle_{\Cbm,\, \Cb}
    &= -\langle \TrT\vc, \nabla_T z \rangle_{\Hbctm,\, \Hbct}\\
    &= -\langle \TrT\vc, \piT(\nabla u) \rangle_{\Hbctm,\, \Hbct}\\
    &= b(\nabla u, \vc).
\end{align*}
Since $\Cb \subset \Hb$, both pairings give the same value, which gives~\eqref{EqTrnCurlDivT}.
\end{proof}

\subsubsection{Tangential curl operators}

In the same way as above, we consider restricted versions of the tangential curl operators.

\begin{theorem}\label{ThmCurlcT}
Let $\Omega$ be a bounded $H^1,H^2$-extension domain of $\R^3$. Then the restricted tangential curl operator
\begin{align*}
    \curlcT \colon \Cb &\to \Hbcp \\
    \Tr v & \mapsto \curlcT(\Tr v) = {\TrT}(\nabla v)
\end{align*}
is a well-defined, linear and continuous mapping.
\end{theorem}

\begin{proof}
For $v \in H^2(\Omega)$, $\nabla v \in \Hc^1(\Omega)$, so $\TrT(\nabla v) \in \Hbcp$ by Definition~\ref{DefBperp}. Independence of the representative and continuity follow exactly as in Theorem~\ref{ThmNablaT}, replacing $\piT$ by $\TrT$ and $\Hbct$ by $\Hbcp$.
\end{proof}

Since $\curlcT$ is a bounded linear operator, we can define its adjoint, which is also linear and bounded. This gives the $\Hc^1(\Omega)$ version of the scalar tangential curl.

\begin{definition}\label{DefCurlT}
Let $\Omega$ be a bounded $H^1,H^2$-extension domain of $\R^3$. We define the restricted scalar tangential curl operator $\curlT : \Hbcpm \to \Cbm$ as the adjoint of $\curlcT$. It is a linear and bounded operator satisfying
\begin{multline}
    \forall \fc \in \Hbcpm,\, \forall z \in \Cb,\\
    \langle \curlT\fc, z \rangle_{\Cbm,\, \Cb} = \langle \fc, \curlcT z \rangle_{\Hbcpm,\, \Hbcp}.
\end{multline}
\end{definition}

With this definition, we finally provide an alternative version of Proposition~\ref{PropTrnCurlCurlTtilde}.

\begin{proposition}\label{PropTrnCurlCurlT}
Let $\Omega$ be a bounded $H^1,H^2$-extension domain of $\R^3$. For all $\vc \in \Hc^1(\Omega)$, the following equality holds in $\Cbm$:
\begin{equation}\label{EqTrnCurlCurlT}
    \trn (\rotv \vc) = \curlT (\piT\vc). 
\end{equation}
\end{proposition}

\begin{proof}
The proof is identical to that of Proposition~\ref{PropTrnCurlDivT}, using formula~\eqref{EqPiTuTrTv} in place of~\eqref{EqbuvTrTpiT}.
\end{proof}

\section{Hodge-Dirac operator}\label{SecHodgeDirac}

As an illustration of the trace theory developed in this article, we finally show that the Green formula for the Hodge-Dirac operator established in~\cite{Schulz_2022} for Lipschitz domains extends to bounded $H^1$-extension domains. This extension is a direct consequence of the normal and tangential trace operators defined in Sections~\ref{sec:normal_trace},~\ref{sec:tangential_trace} and~\ref{sec:tangential_component}, together with the integration by parts formula~\eqref{EqbuvTrTpiT} of Theorem~\ref{ThmTrTtoBtm}.

We introduce the exterior derivative $\boldsymbol d$ and codifferential operator $\boldsymbol \delta$:
\[
\boldsymbol{d} := 
\begin{pmatrix}
0 & \boldsymbol{0}^\top & \boldsymbol{0}^\top & 0 \\
\nabla & \boldsymbol{0}_{3 \times 3} & \boldsymbol{0}_{3 \times 3} & \boldsymbol{0} \\
\boldsymbol{0} & \textbf{curl} & \boldsymbol{0}_{3 \times 3} & \boldsymbol{0} \\
0 & \boldsymbol{0}^\top & \text{div} & 0
\end{pmatrix}
\quad \text{and} \quad
\boldsymbol{\delta} := 
\begin{pmatrix}
0 & -\text{div} & \boldsymbol{0}^\top & 0 \\
0 & \boldsymbol{0}_{3 \times 3} & \textbf{curl} & \boldsymbol{0} \\
\boldsymbol{0} & \boldsymbol{0}_{3 \times 3} & \boldsymbol{0}_{3 \times 3} & -\nabla \\
0 & \boldsymbol{0}^\top & \boldsymbol{0}^\top & 0
\end{pmatrix}.
\]

Let $D := \boldsymbol d + \boldsymbol\delta$ denote the Hodge–Dirac operator. Although the generalized operators $\widetilde{\TrT}$ and $\widetilde\piT$ are defined on $\Hcurl$, the boundary pairing used below in the Hodge--Dirac Green formula has been established here for the $\Hc^1(\Omega)$-based traces $\TrT$ and $\piT$. We therefore work on $(H^1(\Omega))^8$ rather than on the full energy space $\Hc(D,\Omega)$. The case of $\Hc(D,\Omega)$ was considered for the Lipschitz case in~\cite{Schulz_2022}.
We keep the convention of using bold symbols for elements in $\mathbb{R}^3$, and use an overhead arrow to denote elements in $\mathbb{R}^8$: for example, $\Hconevec$ stands for $\left(H^1(\Omega)\right)^8$.
\\

\noindent For $\vec{\boldsymbol U} = 
\begin{pmatrix}
U_0 \\
\boldsymbol U_1 \\
\boldsymbol U_2 \\
U_3
\end{pmatrix} 
\in \Hconevec$, we define the following two trace vectors:
\begin{align*}
\gamma_T(\vec{\boldsymbol{U}}) &:= 
\begin{pmatrix}
\Tr U_0 \\
\piT \boldsymbol{U}_1 \\
\trn \boldsymbol{U}_2
\end{pmatrix}
\quad \text{and} \quad
\gamma_R(\vec{\boldsymbol{U}}) :=
\begin{pmatrix}
\trn \boldsymbol{U}_1 \\
\TrT \boldsymbol{U}_2 \\
\Tr U_3
\end{pmatrix}
\end{align*}

We are now able to prove the following Green formulas for the Hodge-Dirac operator.

\begin{theorem}
Let $\Omega$ be a bounded $H^1$-extension domain of $\mathbb{R}^3$. We have the following two formulas:
\begin{enumerate}
\item[(i)] $\forall\, \vec{\boldsymbol U},\vec{\boldsymbol V} \in \Hconevec,$
\begin{equation*}
\int_\Omega \boldsymbol d \vec{\boldsymbol U} \cdot \vec{\boldsymbol V}\dx = \int_\Omega \vec{\boldsymbol U} \cdot \boldsymbol \delta \vec{\boldsymbol V} \dx + \llangle \gamma_T \vec{\boldsymbol U}, \gamma_R \vec{\boldsymbol V} \rrangle_{\partial\Omega},
\end{equation*}
\item[(ii)] $\forall \, \vec{\boldsymbol U},\vec{\boldsymbol V} \in \Hconevec,$
\begin{equation*}
\int_\Omega D\vec{\boldsymbol U} \cdot \vec{\boldsymbol V}\dx = \int_\Omega \vec{\boldsymbol U} \cdot D\vec{\boldsymbol V} \dx + \llangle \gamma_T \vec{\boldsymbol U}, \gamma_R \vec{\boldsymbol V} \rrangle_{\partial\Omega} - \llangle \gamma_T \vec{\boldsymbol V}, \gamma_R \vec{\boldsymbol U} \rrangle_{\partial\Omega},
\end{equation*}
\end{enumerate}
where
\begin{align*}
\llangle \gamma_T \vec{\boldsymbol U}, \gamma_R \vec{\boldsymbol V} \rrangle_{\partial\Omega} 
&= \left\langle\trn \boldsymbol V_1,\Tr U_0\right\rangle_{\Hbm,\,\Hb}  + \left\langle \TrT \boldsymbol V_2, \piT \boldsymbol U_1 \right\rangle_{\Hbctm,\, \Hbct} \notag \\
&\quad + \left\langle\trn \boldsymbol U_2,\Tr V_3\right\rangle_{\Hbm,\,\Hb}.
\end{align*}
\end{theorem}

\begin{proof}

Expanding $\boldsymbol d\vec{\boldsymbol U}$ and $\boldsymbol\delta\vec{\boldsymbol V}$ componentwise gives
\begin{align*}
&\int_\Omega \boldsymbol d\vec{\boldsymbol U}\cdot\vec{\boldsymbol V}\,dx
-\int_\Omega \vec{\boldsymbol U}\cdot\boldsymbol\delta\vec{\boldsymbol V}\,dx\\
&\quad=\int_\Omega\bigl(\nabla U_0\cdot\boldsymbol V_1+U_0\operatorname{div}\boldsymbol V_1\bigr)\,dx\\
&\qquad+\int_\Omega\bigl((\rotv\boldsymbol U_1)\cdot\boldsymbol V_2-\boldsymbol U_1\cdot\rotv\boldsymbol V_2\bigr)\,dx\\
&\qquad+\int_\Omega\bigl((\operatorname{div}\boldsymbol U_2)V_3+\boldsymbol U_2\cdot\nabla V_3\bigr)\,dx.
\end{align*}
Apply the normal Green formula~\eqref{EqItByPn} to the first and third lines. For the middle line, apply the tangential integration-by-parts formula~\eqref{EqbuvTrTpiT} with $\boldsymbol u=\boldsymbol V_2$ and $\boldsymbol v=\boldsymbol U_1$. The three boundary terms are exactly $\langle\!\langle\gamma_T\vec{\boldsymbol U},\gamma_R\vec{\boldsymbol V}\rangle\!\rangle_{\partial\Omega}$, which proves (i). Interchanging $\vec{\boldsymbol U}$ and $\vec{\boldsymbol V}$ in (i), and using $D=\boldsymbol d+\boldsymbol\delta$, gives (ii).
\end{proof}

\begin{remark}
In~\cite{Schulz_2022}, the Green formula is established for $\vec{\boldsymbol{U}}, \vec{\boldsymbol{V}} \in \Hc(D,\Omega)$, the full energy space of the Dirac operator, using the surjectivity of the trace operators on $\Hcurl$ and $\Hdiv$. In our setting, restricting to $\Hconevec$ allows us to use the tangential trace operators defined in this article, at the cost of working in a smaller function space. While the broader tangential trace $\widetilde{\TrT}$ is already defined on $\Hcurl$ for $H^1$-extension domains (Theorem~\ref{ThmTrTtoBtm}), extending the Green formula to the full energy space $\Hc(D,\Omega)$ would require establishing the surjectivity of $\widetilde{\TrT}$ and $\widetilde{\piT}$ from $\Hcurl$ onto appropriate boundary spaces, together with an extension of formula~\eqref{EqbuvTrTpiT} to $\uc,\vc\in\Hcurl$.
\end{remark}

\backmatter
 \section*{Acknowledgments}
The authors thank C. Verret for her initial joint work with R. Cervera on the subject under the supervision of A. Rozanova-Pierrat.
A. Rozanova-Pierrat thanks J.~Tomasik, F.~Magoul\`es, X.~Claeys and M.~Hinz for their interest in this work and for helpful discussions.
The first author acknowledges the opportunity provided by the Recherche teaching program at CentraleSup\'elec, during which this work was initiated.
 
\section*{Statements and Declarations}

\bmhead{Funding}
A. Rozanova-Pierrat is partially supported by COST Action 24122 mSPACE, supported
by COST (European Cooperation in Science and Technology), \url{www.cost.eu.} 
A. Teplyaev was supported in part by the National Science Foundation, the Fulbright Specialist Program, and the Simons Foundation Collaboration Grant; the project was conducted during a stay at CentraleSup\'elec, Universit\'e Paris-Saclay, whose kind hospitality is gratefully acknowledged.

\bmhead{Conflict of interest}
The authors declare that they  have no conflict of interest.

\bmhead{Data availability}
Data sharing is not applicable to this article as no datasets were generated or analyzed during the current study.

\bibliography{biblio}


\begin{thebibliography}{47}
\ifx \bisbn   \undefined \def \bisbn  #1{ISBN #1}\fi
\ifx \binits  \undefined \def \binits#1{#1}\fi
\ifx \bauthor  \undefined \def \bauthor#1{#1}\fi
\ifx \batitle  \undefined \def \batitle#1{#1}\fi
\ifx \bjtitle  \undefined \def \bjtitle#1{#1}\fi
\ifx \bvolume  \undefined \def \bvolume#1{\textbf{#1}}\fi
\ifx \byear  \undefined \def \byear#1{#1}\fi
\ifx \bissue  \undefined \def \bissue#1{#1}\fi
\ifx \bfpage  \undefined \def \bfpage#1{#1}\fi
\ifx \blpage  \undefined \def \blpage #1{#1}\fi
\ifx \burl  \undefined \def \burl#1{\textsf{#1}}\fi
\ifx \doiurl  \undefined \def \doiurl#1{\url{https://doi.org/#1}}\fi
\ifx \betal  \undefined \def \betal{\textit{et al.}}\fi
\ifx \binstitute  \undefined \def \binstitute#1{#1}\fi
\ifx \binstitutionaled  \undefined \def \binstitutionaled#1{#1}\fi
\ifx \bctitle  \undefined \def \bctitle#1{#1}\fi
\ifx \beditor  \undefined \def \beditor#1{#1}\fi
\ifx \bpublisher  \undefined \def \bpublisher#1{#1}\fi
\ifx \bbtitle  \undefined \def \bbtitle#1{#1}\fi
\ifx \bedition  \undefined \def \bedition#1{#1}\fi
\ifx \bseriesno  \undefined \def \bseriesno#1{#1}\fi
\ifx \blocation  \undefined \def \blocation#1{#1}\fi
\ifx \bsertitle  \undefined \def \bsertitle#1{#1}\fi
\ifx \bsnm \undefined \def \bsnm#1{#1}\fi
\ifx \bsuffix \undefined \def \bsuffix#1{#1}\fi
\ifx \bparticle \undefined \def \bparticle#1{#1}\fi
\ifx \barticle \undefined \def \barticle#1{#1}\fi
\bibcommenthead
\ifx \bconfdate \undefined \def \bconfdate #1{#1}\fi
\ifx \botherref \undefined \def \botherref #1{#1}\fi
\ifx \url \undefined \def \url#1{\textsf{#1}}\fi
\ifx \bchapter \undefined \def \bchapter#1{#1}\fi
\ifx \bbook \undefined \def \bbook#1{#1}\fi
\ifx \bcomment \undefined \def \bcomment#1{#1}\fi
\ifx \oauthor \undefined \def \oauthor#1{#1}\fi
\ifx \citeauthoryear \undefined \def \citeauthoryear#1{#1}\fi
\ifx \endbibitem  \undefined \def \endbibitem {}\fi
\ifx \bconflocation  \undefined \def \bconflocation#1{#1}\fi
\ifx \arxivurl  \undefined \def \arxivurl#1{\textsf{#1}}\fi
\csname PreBibitemsHook\endcsname

\bibitem[\protect\citeauthoryear{Assous et~al.}{2018}]{Assous_2018}
\begin{bbook}
\bauthor{\bsnm{Assous}, \binits{F.}},
\bauthor{\bsnm{Ciarlet}, \binits{P.}},
\bauthor{\bsnm{Labrunie}, \binits{S.}}:
\bbtitle{Mathematical Foundations of Computational Electromagnetism},
\bedition{1}st edn.
\bsertitle{Applied mathematical sciences}.
\bpublisher{Springer},
\blocation{Cham, Switzerland}
(\byear{2018}).
\doiurl{10.1007/978-3-319-70842-3}
\end{bbook}
\endbibitem

\bibitem[\protect\citeauthoryear{Monk}{2003}]{Monk_2003}
\begin{bbook}
\bauthor{\bsnm{Monk}, \binits{P.}}:
\bbtitle{Finite Element Methods for Maxwell's Equations}.
\bpublisher{Oxford University Press},
\blocation{Oxford}
(\byear{2003}).
\doiurl{10.1093/acprof:oso/9780198508885.001.0001}
\end{bbook}
\endbibitem

\bibitem[\protect\citeauthoryear{{C}hicaud et~al.}{2021}]{CHICAUD-2021}
\begin{barticle}
\bauthor{\bsnm{{C}hicaud}, \binits{D.}},
\bauthor{\bsnm{{C}iarlet}, \binits{P.}},
\bauthor{\bsnm{{M}odave}, \binits{A.}}:
\batitle{{A}nalysis of {V}ariational {F}ormulations and {L}ow-regularity
  {S}olutions for {T}ime-harmonic {E}lectromagnetic {P}roblems in {C}omplex
  {A}nisotropic {M}edia}.
\bjtitle{{S}{I}{A}{M} {J}ournal on {M}athematical {A}nalysis}
\bvolume{53}(\bissue{3}),
\bfpage{2691}--\blpage{2717}
(\byear{2021})
\doiurl{10.1137/20M1344111}
\end{barticle}
\endbibitem

\bibitem[\protect\citeauthoryear{Chicaud}{2021}]{CHICAUD-2021-1}
\begin{botherref}
\oauthor{\bsnm{Chicaud}, \binits{D.}}:
{Analysis of time-harmonic electromagnetic problems in elliptic anisotropic
  media}.
Theses,
{Institut Polytechnique de Paris}
(December 2021).
\doiurl{10.70675/967c1a06z9a2cz4303zb664zb57f05c43987}
\end{botherref}
\endbibitem

\bibitem[\protect\citeauthoryear{Chicaud and Ciarlet}{2023}]{CHICAUD-2023}
\begin{barticle}
\bauthor{\bsnm{Chicaud}, \binits{D.}},
\bauthor{\bsnm{Ciarlet}, \binits{P.}}:
\batitle{Analysis of {T}ime-{H}armonic {M}axwell {I}mpedance {P}roblems in
  {A}nisotropic {M}edia}.
\bjtitle{SIAM J. Math. Anal.}
\bvolume{55}(\bissue{3}),
\bfpage{1969}--\blpage{2000}
(\byear{2023})
\doiurl{10.1137/22M1485413}
\end{barticle}
\endbibitem

\bibitem[\protect\citeauthoryear{Calder\'on}{1961}]{CALDERON-1961}
\begin{bchapter}
\bauthor{\bsnm{Calder\'on}, \binits{A.P.}}:
\bctitle{Lebesgue spaces of differentiable functions and distributions}.
In: \beditor{\bsnm{Morrey}, \binits{C.B.} \bsuffix{Jr.}} (ed.)
\bbtitle{Partial Differential Equations}.
\bsertitle{Proceedings of Symposia in Pure Mathematics},
vol. \bseriesno{4},
pp. \bfpage{33}--\blpage{49}.
\bpublisher{American Mathematical Society},
\blocation{Providence, RI}
(\byear{1961}).
\doiurl{10.1090/pspum/004/0143037} .
\bcomment{MR 0143037}
\end{bchapter}
\endbibitem

\bibitem[\protect\citeauthoryear{Stein}{1970}]{STEIN-1970}
\begin{bbook}
\bauthor{\bsnm{Stein}, \binits{E.M.}}:
\bbtitle{Singular Integrals and Differentiability Properties of Functions}.
\bsertitle{Princeton Mathematical Series},
vol. \bseriesno{30}.
\bpublisher{Princeton University Press},
\blocation{Princeton, NJ}
(\byear{1970}).
\doiurl{10.1515/9781400883882}
\end{bbook}
\endbibitem

\bibitem[\protect\citeauthoryear{{J}ones}{1981}]{JONES-1981}
\begin{barticle}
\bauthor{\bsnm{{J}ones}, \binits{P.W.}}:
\batitle{{Q}uasiconformal mappings and extendability of functions in {S}obolev
  spaces}.
\bjtitle{{A}cta {M}athematica}
\bvolume{147}(\bissue{1}),
\bfpage{71}--\blpage{88}
(\byear{1981})
\doiurl{10.1007/BF02392869}
\end{barticle}
\endbibitem

\bibitem[\protect\citeauthoryear{{H}aj{\l}asz et~al.}{2008}]{HAJLASZ-2008}
\begin{barticle}
\bauthor{\bsnm{{H}aj{\l}asz}, \binits{P.}},
\bauthor{\bsnm{{K}oskela}, \binits{P.}},
\bauthor{\bsnm{{T}uominen}, \binits{H.}}:
\batitle{{S}obolev embeddings, extensions and measure density condition}.
\bjtitle{{J}ournal of {F}unctional {A}nalysis}
\bvolume{254}(\bissue{5}),
\bfpage{1217}--\blpage{1234}
(\byear{2008})
\doiurl{10.1016/j.jfa.2007.11.020}
\end{barticle}
\endbibitem

\bibitem[\protect\citeauthoryear{{R}ogers}{2006}]{ROGERS-2006}
\begin{barticle}
\bauthor{\bsnm{{R}ogers}, \binits{L.G.}}:
\batitle{{D}egree-independent {S}obolev extension on locally uniform domains}.
\bjtitle{{J}ournal of {F}unctional {A}nalysis}
\bvolume{235}(\bissue{2}),
\bfpage{619}--\blpage{665}
(\byear{2006})
\doiurl{10.1016/j.jfa.2005.11.013}
\end{barticle}
\endbibitem

\bibitem[\protect\citeauthoryear{{H}inz et~al.}{2021}]{HINZ-2021-1}
\begin{barticle}
\bauthor{\bsnm{{H}inz}, \binits{M.}},
\bauthor{\bsnm{{R}ozanova{-}{P}ierrat}, \binits{A.}},
\bauthor{\bsnm{{T}eplyaev}, \binits{A.}}:
\batitle{{N}on-{L}ipschitz {U}niform {D}omain {S}hape {O}ptimization in
  {L}inear {A}coustics}.
\bjtitle{{S}{I}{A}{M} {J}ournal on {C}ontrol and {O}ptimization}
\bvolume{59}(\bissue{2}),
\bfpage{1007}--\blpage{1032}
(\byear{2021})
\doiurl{10.1137/20M1361687}
\end{barticle}
\endbibitem

\bibitem[\protect\citeauthoryear{{H}inz et~al.}{2023}]{HINZ-2023}
\begin{barticle}
\bauthor{\bsnm{{H}inz}, \binits{M.}},
\bauthor{\bsnm{{R}ozanova{-}{P}ierrat}, \binits{A.}},
\bauthor{\bsnm{{T}eplyaev}, \binits{A.}}:
\batitle{{B}oundary value problems on non-{L}ipschitz uniform domains:
  stability, compactness and the existence of optimal shapes}.
\bjtitle{{A}symptotic {A}nalysis}
\bvolume{134}(\bissue{1-2}),
\bfpage{25}--\blpage{61}
(\byear{2023})
\doiurl{10.3233/ASY-231825}
\end{barticle}
\endbibitem

\bibitem[\protect\citeauthoryear{Claret et~al.}{2026}]{Claret_2026}
\begin{barticle}
\bauthor{\bsnm{Claret}, \binits{G.}},
\bauthor{\bsnm{Hinz}, \binits{M.}},
\bauthor{\bsnm{Rozanova-Pierrat}, \binits{A.}},
\bauthor{\bsnm{Teplyaev}, \binits{A.}}:
\batitle{Layer potential operators for transmission problems on extension
  domains}.
\bjtitle{Journal de Math{\'e}matiques Pures et Appliqu{\'e}es}
\bvolume{212},
\bfpage{103888}
(\byear{2026})
\doiurl{10.1016/j.matpur.2026.103888}
\end{barticle}
\endbibitem

\bibitem[\protect\citeauthoryear{Claret et~al.}{2024}]{Claret_2024}
\begin{botherref}
\oauthor{\bsnm{Claret}, \binits{G.}},
\oauthor{\bsnm{Rozanova-Pierrat}, \binits{A.}},
\oauthor{\bsnm{Teplyaev}, \binits{A.}}:
Convergence of layer potentials and {R}iemann-{H}ilbert problem on extension
  domains.
arXiv
(2024).
\doiurl{10.48550/ARXIV.2411.03767} .
\url{https://arxiv.org/abs/2411.03767}
\end{botherref}
\endbibitem

\bibitem[\protect\citeauthoryear{{H}inz et~al.}{2021}]{HINZ-2021}
\begin{barticle}
\bauthor{\bsnm{{H}inz}, \binits{M.}},
\bauthor{\bsnm{{M}agoul{\`e}s}, \binits{F.}},
\bauthor{\bsnm{{R}ozanova{-}{P}ierrat}, \binits{A.}},
\bauthor{\bsnm{{R}ynkovskaya}, \binits{M.}},
\bauthor{\bsnm{{T}eplyaev}, \binits{A.}}:
\batitle{{O}n the existence of optimal shapes in architecture}.
\bjtitle{{A}pplied {M}athematical {M}odelling}
\bvolume{94},
\bfpage{676}--\blpage{687}
(\byear{2021})
\doiurl{10.1016/j.apm.2021.01.041}
\end{barticle}
\endbibitem

\bibitem[\protect\citeauthoryear{{C}laret and
  {R}ozanova{-}{P}ierrat}{2025}]{CLARET-2025}
\begin{bchapter}
\bauthor{\bsnm{{C}laret}, \binits{G.}},
\bauthor{\bsnm{{R}ozanova{-}{P}ierrat}, \binits{A.}}:
\bctitle{Existence of optimal shapes for heat diffusions across irregular
  interfaces}.
In: \bbtitle{Fractal Geometry in Pure and Applied Mathematics}.
\bsertitle{Contemporary Mathematics},
vol. \bseriesno{828},
pp. \bfpage{15}--\blpage{40}.
\bpublisher{American Mathematical Society},
\blocation{Providence, RI}
(\byear{2025})
\end{bchapter}
\endbibitem

\bibitem[\protect\citeauthoryear{{R}ozanova{-}{P}ierrat}{2021}]{ROZANOVA-PIERRAT-2021}
\begin{bchapter}
\bauthor{\bsnm{{R}ozanova{-}{P}ierrat}, \binits{A.}}:
\bctitle{Generalization of {Rellich--Kondrachov} theorem and trace compactness
  for fractal boundaries}.
In: \beditor{\bsnm{Lancia}, \binits{M.R.}},
\beditor{\bsnm{Rozanova-Pierrat}, \binits{A.}} (eds.)
\bbtitle{Fractals in Engineering: Theoretical Aspects and Numerical
  Approximations}.
\bsertitle{SEMA SIMAI Springer Series},
vol. \bseriesno{8},
pp. \bfpage{155}--\blpage{173}.
\bpublisher{Springer},
\blocation{Cham}
(\byear{2021}).
\doiurl{10.1007/978-3-030-61803-2_7}
\end{bchapter}
\endbibitem

\bibitem[\protect\citeauthoryear{{R}ozanova{-}{P}ierrat}{2026}]{ROZANOVA-PIERRAT-2026-1}
\begin{bchapter}
\bauthor{\bsnm{{R}ozanova{-}{P}ierrat}, \binits{A.}}:
\bctitle{{O}n analysis of problems of mathematical physics with non-{L}ipschitz
  boundaries}.
In: \beditor{\bsnm{Freiberg}, \binits{U.}},
\beditor{\bsnm{Hinz}, \binits{M.}},
\beditor{\bsnm{J{\"a}rvenp{\"a}{\"a}}, \binits{M.}},
\beditor{\bsnm{Winter}, \binits{S.}} (eds.)
\bbtitle{Fractal Geometry and Stochastics}.
\bsertitle{Progress in Probability}.
\bpublisher{Birkh{\"a}user},
\blocation{Cham, Switzerland}
(\byear{2026}).
\bcomment{Preprint available as arXiv:2509.02176 (2025)}.
\burl{https://arxiv.org/abs/2509.02176}
\end{bchapter}
\endbibitem

\bibitem[\protect\citeauthoryear{{A}rfi and
  {R}ozanova{-}{P}ierrat}{2019}]{ARFI-2019}
\begin{barticle}
\bauthor{\bsnm{{A}rfi}, \binits{K.}},
\bauthor{\bsnm{{R}ozanova{-}{P}ierrat}, \binits{A.}}:
\batitle{{Dirichlet}-to-{Neumann} or {Poincar\'e--Steklov} operator on fractals
  described by {$d$}-sets}.
\bjtitle{{D}iscrete \& {C}ontinuous {D}ynamical {S}ystems - {S}}
\bvolume{12}(\bissue{1}),
\bfpage{1}--\blpage{26}
(\byear{2019})
\doiurl{10.3934/dcdss.2019001}
\end{barticle}
\endbibitem

\bibitem[\protect\citeauthoryear{{R}ozanova
  {P}ierrat}{2027}]{ROZANOVA-PIERRAT-2026}
\begin{barticle}
\bauthor{\bsnm{{R}ozanova-{P}ierrat}, \binits{A.}}:
\batitle{Erratum for ``{Dirichlet}-to-{Neumann} or {Poincar\'e--Steklov}
  operator on fractals described by {$d$}-sets''}.
\bjtitle{Discrete and Continuous Dynamical Systems -- S}
\bvolume{32},
\bfpage{613}--\blpage{615}
(\byear{2027})
\doiurl{10.3934/dcdss.2026136} .
\bcomment{Published online 29 April 2026}
\end{barticle}
\endbibitem

\bibitem[\protect\citeauthoryear{{D}ekkers et~al.}{2022}]{DEKKERS-2022}
\begin{barticle}
\bauthor{\bsnm{{D}ekkers}, \binits{A.}},
\bauthor{\bsnm{{R}ozanova{-}{P}ierrat}, \binits{A.}},
\bauthor{\bsnm{{T}eplyaev}, \binits{A.}}:
\batitle{{M}ixed boundary valued problems for linear and nonlinear wave
  equations in domains with fractal boundaries}.
\bjtitle{{C}alculus of {V}ariations and {P}artial {D}ifferential {E}quations}
\bvolume{61}(\bissue{2}),
\bfpage{75}
(\byear{2022})
\doiurl{10.1007/s00526-021-02159-3}
\end{barticle}
\endbibitem

\bibitem[\protect\citeauthoryear{{D}ekkers and
  {R}ozanova{-}{P}ierrat}{2022}]{DEKKERS-2022-1}
\begin{barticle}
\bauthor{\bsnm{{D}ekkers}, \binits{A.}},
\bauthor{\bsnm{{R}ozanova{-}{P}ierrat}, \binits{A.}}:
\batitle{{D}irichlet boundary valued problems for linear and nonlinear wave
  equations on arbitrary and fractal domains}.
\bjtitle{{J}ournal of {M}athematical {A}nalysis and {A}pplications}
\bvolume{512}(\bissue{1}),
\bfpage{126089}
(\byear{2022})
\doiurl{10.1016/j.jmaa.2022.126089}
\end{barticle}
\endbibitem

\bibitem[\protect\citeauthoryear{{L}ancia}{2002}]{LANCIA-2002}
\begin{barticle}
\bauthor{\bsnm{{L}ancia}, \binits{M.R.}}:
\batitle{{A} {T}ransmission {P}roblem with a {F}ractal {I}nterface}.
\bjtitle{{Z}eitschrift f{\"u}r {A}nalysis und ihre {A}nwendungen}
\bvolume{21}(\bissue{1}),
\bfpage{113}--\blpage{133}
(\byear{2002})
\doiurl{10.4171/ZAA/1067}
\end{barticle}
\endbibitem

\bibitem[\protect\citeauthoryear{{C}apitanelli and
  {V}ivaldi}{2011}]{CAPITANELLI-2011}
\begin{barticle}
\bauthor{\bsnm{{C}apitanelli}, \binits{R.}},
\bauthor{\bsnm{{V}ivaldi}, \binits{M.A.}}:
\batitle{{I}nsulating layers and {R}obin {P}roblems on {K}och mixtures}.
\bjtitle{{J}ournal of {D}ifferential {E}quations}
\bvolume{251}(\bissue{4-5}),
\bfpage{1332}--\blpage{1353}
(\byear{2011})
\doiurl{10.1016/j.jde.2011.02.003}
\end{barticle}
\endbibitem

\bibitem[\protect\citeauthoryear{{B}iegert}{2009}]{BIEGERT-2009}
\begin{barticle}
\bauthor{\bsnm{{B}iegert}, \binits{M.}}:
\batitle{{O}n traces of {S}obolev functions on the boundary of extension
  domains}.
\bjtitle{{P}roceedings of the {A}merican {M}athematical {S}ociety}
\bvolume{137}(\bissue{12}),
\bfpage{4169}--\blpage{4176}
(\byear{2009})
\doiurl{10.1090/S0002-9939-09-10045-X}
\end{barticle}
\endbibitem

\bibitem[\protect\citeauthoryear{Girault and Raviart}{1986}]{GIRAULT-1986}
\begin{bbook}
\bauthor{\bsnm{Girault}, \binits{V.}},
\bauthor{\bsnm{Raviart}, \binits{P.-A.}}:
\bbtitle{Finite Element Methods for Navier-Stokes Equations}.
\bpublisher{Springer},
\blocation{Berlin, Heidelberg}
(\byear{1986}).
\doiurl{10.1007/978-3-642-61623-5}
\end{bbook}
\endbibitem

\bibitem[\protect\citeauthoryear{N\'ed\'elec}{2001}]{Nedelec2001}
\begin{bbook}
\bauthor{\bsnm{N\'ed\'elec}, \binits{J.-C.}}:
\bbtitle{Acoustic and Electromagnetic Equations: Integral Representations for
  Harmonic Problems}.
\bsertitle{Applied Mathematical Sciences},
vol. \bseriesno{144}.
\bpublisher{Springer},
\blocation{New York, NY}
(\byear{2001}).
\doiurl{10.1007/978-1-4757-4393-7}
\end{bbook}
\endbibitem

\bibitem[\protect\citeauthoryear{Buffa et~al.}{2002}]{BUFFA-2002}
\begin{barticle}
\bauthor{\bsnm{Buffa}, \binits{A.}},
\bauthor{\bsnm{Costabel}, \binits{M.}},
\bauthor{\bsnm{Sheen}, \binits{D.}}:
\batitle{{O}n traces for {$\mathbf{H}(\mathbf{curl},\Omega)$} in {L}ipschitz
  domains}.
\bjtitle{Journal of Mathematical Analysis and Applications}
\bvolume{276}(\bissue{2}),
\bfpage{845}--\blpage{867}
(\byear{2002})
\doiurl{10.1016/S0022-247X(02)00455-9}
\end{barticle}
\endbibitem

\bibitem[\protect\citeauthoryear{Costabel}{1991}]{COSTABEL-1991}
\begin{barticle}
\bauthor{\bsnm{Costabel}, \binits{M.}}:
\batitle{{A} coercive bilinear form for {M}axwell's equations}.
\bjtitle{Journal of Mathematical Analysis and Applications}
\bvolume{157}(\bissue{2}),
\bfpage{527}--\blpage{541}
(\byear{1991})
\doiurl{10.1016/0022-247X(91)90104-8}
\end{barticle}
\endbibitem

\bibitem[\protect\citeauthoryear{Amrouche et~al.}{1998}]{Amrouche_1998}
\begin{barticle}
\bauthor{\bsnm{Amrouche}, \binits{C.}},
\bauthor{\bsnm{Bernardi}, \binits{C.}},
\bauthor{\bsnm{Dauge}, \binits{M.}},
\bauthor{\bsnm{Girault}, \binits{V.}}:
\batitle{Vector potentials in three-dimensional non-smooth domains}.
\bjtitle{Mathematical Methods in the Applied Sciences}
\bvolume{21}(\bissue{9}),
\bfpage{823}--\blpage{864}
(\byear{1998})
\doiurl{10.1002/(SICI)1099-1476(199806)21:9<823::AID-MMA976>3.0.CO;2-B}
\end{barticle}
\endbibitem

\bibitem[\protect\citeauthoryear{Costabel and Dauge}{2000}]{Costabel_2000}
\begin{barticle}
\bauthor{\bsnm{Costabel}, \binits{M.}},
\bauthor{\bsnm{Dauge}, \binits{M.}}:
\batitle{Singularities of {E}lectromagnetic {F}ields in {P}olyhedral
  {D}omains}.
\bjtitle{Archive for Rational Mechanics and Analysis}
\bvolume{151}(\bissue{3}),
\bfpage{221}--\blpage{276}
(\byear{2000})
\doiurl{10.1007/s002050050197}
\end{barticle}
\endbibitem

\bibitem[\protect\citeauthoryear{Creo et~al.}{2020}]{Creo_2018}
\begin{bchapter}
\bauthor{\bsnm{Creo}, \binits{S.}},
\bauthor{\bsnm{Lancia}, \binits{M.R.}},
\bauthor{\bsnm{Vernole}, \binits{P.}},
\bauthor{\bsnm{Hinz}, \binits{M.}},
\bauthor{\bsnm{Teplyaev}, \binits{A.}}:
\bctitle{Magnetostatic problems in fractal domains}.
In: \beditor{\bsnm{Alonso~Ruiz}, \binits{P.}},
\beditor{\bsnm{Chen}, \binits{J.P.}},
\beditor{\bsnm{Rogers}, \binits{L.G.}},
\beditor{\bsnm{Strichartz}, \binits{R.S.}},
\beditor{\bsnm{Teplyaev}, \binits{A.}} (eds.)
\bbtitle{Analysis, Probability and Mathematical Physics on Fractals}.
\bsertitle{Fractals and Dynamics in Mathematics, Science, and the Arts: Theory
  and Applications},
vol. \bseriesno{5},
pp. \bfpage{477}--\blpage{502}.
\bpublisher{World Scientific},
\blocation{Singapore}
(\byear{2020}).
\doiurl{10.1142/9789811215537_0015}
\end{bchapter}
\endbibitem

\bibitem[\protect\citeauthoryear{Claeys and Hiptmair}{2012}]{claeys-2012}
\begin{barticle}
\bauthor{\bsnm{Claeys}, \binits{X.}},
\bauthor{\bsnm{Hiptmair}, \binits{R.}}:
\batitle{{Electromagnetic scattering at composite objects : a novel multi-trace
  boundary integral formulation}}.
\bjtitle{ESAIM Mathematical Modelling and Numerical Analysis}
\bvolume{46}(\bissue{6}),
\bfpage{1421}--\blpage{1445}
(\byear{2012})
\doiurl{10.1051/m2an/2012011}
\end{barticle}
\endbibitem

\bibitem[\protect\citeauthoryear{Schulz and Hiptmair}{2022}]{Schulz_2022}
\begin{barticle}
\bauthor{\bsnm{Schulz}, \binits{E.}},
\bauthor{\bsnm{Hiptmair}, \binits{R.}}:
\batitle{First-kind boundary integral equations for the {D}irac operator in
  3-dimensional {L}ipschitz domains}.
\bjtitle{SIAM Journal on Mathematical Analysis}
\bvolume{54}(\bissue{1}),
\bfpage{616}--\blpage{648}
(\byear{2022})
\doiurl{10.1137/20m1389224}
\end{barticle}
\endbibitem

\bibitem[\protect\citeauthoryear{Jonsson and Wallin}{1984}]{JONSSON-1984}
\begin{bbook}
\bauthor{\bsnm{Jonsson}, \binits{A.}},
\bauthor{\bsnm{Wallin}, \binits{H.}}:
\bbtitle{Function Spaces on Subsets of $\mathbb{R}^n$}.
\bsertitle{Mathematical Reports},
vol. \bseriesno{2}.
\bpublisher{Harwood Academic Publishers},
\blocation{London}
(\byear{1984})
\end{bbook}
\endbibitem

\bibitem[\protect\citeauthoryear{{A}chdou et~al.}{2014}]{ACHDOU-2014}
\begin{barticle}
\bauthor{\bsnm{{A}chdou}, \binits{Y.}},
\bauthor{\bsnm{{D}eheuvels}, \binits{T.}},
\bauthor{\bsnm{{T}chou}, \binits{N.}}:
\batitle{{C}omparison of {D}ifferent {D}efinitions of {T}races for a {C}lass of
  {R}amified {D}omains with {S}elf-{S}imilar {F}ractal {B}oundaries}.
\bjtitle{{P}otential {A}nalysis}
\bvolume{40}(\bissue{4}),
\bfpage{345}--\blpage{362}
(\byear{2014})
\doiurl{10.1007/s11118-013-9352-y}
\end{barticle}
\endbibitem

\bibitem[\protect\citeauthoryear{{W}allin}{1991}]{WALLIN-1991}
\begin{barticle}
\bauthor{\bsnm{{W}allin}, \binits{H.}}:
\batitle{{T}he trace to the boundary of {S}obolev spaces on a snowflake}.
\bjtitle{{M}anuscripta {M}ath}
\bvolume{73}(\bissue{1}),
\bfpage{117}--\blpage{125}
(\byear{1991})
\doiurl{10.1007/BF02567633}
\end{barticle}
\endbibitem

\bibitem[\protect\citeauthoryear{{N}ystr{\"o}m}{1996}]{NYSTROM-1996}
\begin{barticle}
\bauthor{\bsnm{{N}ystr{\"o}m}, \binits{K.}}:
\batitle{{I}ntegrability of {G}reen potentials in fractal domains}.
\bjtitle{{A}rkiv f{\"o}r {M}atematik}
\bvolume{34}(\bissue{2}),
\bfpage{335}--\blpage{381}
(\byear{1996})
\doiurl{10.1007/BF02559551}
\end{barticle}
\endbibitem

\bibitem[\protect\citeauthoryear{Fukushima et~al.}{2011}]{Fukushima_2011}
\begin{bbook}
\bauthor{\bsnm{Fukushima}, \binits{M.}},
\bauthor{\bsnm{Oshima}, \binits{Y.}},
\bauthor{\bsnm{Takeda}, \binits{M.}}:
\bbtitle{Dirichlet Forms and Symmetric Markov Processes},
\bedition{2nd} edn.
\bsertitle{De Gruyter Studies in Mathematics},
vol. \bseriesno{19}.
\bpublisher{Walter de Gruyter},
\blocation{Berlin, Boston}
(\byear{2011}).
\doiurl{10.1515/9783110218091}
\end{bbook}
\endbibitem

\bibitem[\protect\citeauthoryear{Grisvard}{1992}]{Grisvard_1992}
\begin{bbook}
\bauthor{\bsnm{Grisvard}, \binits{P.}}:
\bbtitle{Singularities in Boundary Value Problems}.
\bsertitle{Recherches en Math\'ematiques Appliqu\'ees},
vol. \bseriesno{22}.
\bpublisher{Masson},
\blocation{Paris}
(\byear{1992})
\end{bbook}
\endbibitem

\bibitem[\protect\citeauthoryear{Doob}{1962}]{Doob_1962}
\begin{barticle}
\bauthor{\bsnm{Doob}, \binits{J.L.}}:
\batitle{Boundary properties of functions with finite {Dirichlet} integrals}.
\bjtitle{Annales de l'Institut Fourier}
\bvolume{12},
\bfpage{573}--\blpage{621}
(\byear{1962})
\doiurl{10.5802/aif.126}
\end{barticle}
\endbibitem

\bibitem[\protect\citeauthoryear{Constantinescu and
  Cornea}{1963}]{Constantinescu_1963}
\begin{bbook}
\bauthor{\bsnm{Constantinescu}, \binits{C.}},
\bauthor{\bsnm{Cornea}, \binits{A.}}:
\bbtitle{Ideale R\"ander Riemannscher Fl\"achen}.
\bsertitle{Ergebnisse der Mathematik und ihrer Grenzgebiete},
vol. \bseriesno{32}.
\bpublisher{Springer},
\blocation{Berlin, Heidelberg}
(\byear{1963}).
\doiurl{10.1007/978-3-642-87031-6}
\end{bbook}
\endbibitem

\bibitem[\protect\citeauthoryear{Maeda}{1964}]{Maeda_1964}
\begin{barticle}
\bauthor{\bsnm{Maeda}, \binits{F.-Y.}}:
\batitle{{Normal derivatives on an ideal boundary}}.
\bjtitle{Journal of Science of the Hiroshima University, Series A-I
  (Mathematics)}
\bvolume{28}(\bissue{2}),
\bfpage{113}--\blpage{131}
(\byear{1964})
\doiurl{10.32917/hmj/1206139392}
\end{barticle}
\endbibitem

\bibitem[\protect\citeauthoryear{Maeda}{1980}]{Maeda_1980}
\begin{bbook}
\bauthor{\bsnm{Maeda}, \binits{F.-Y.}}:
\bbtitle{Dirichlet Integrals on Harmonic Spaces}.
\bsertitle{Lecture Notes in Mathematics},
vol. \bseriesno{803}.
\bpublisher{Springer},
\blocation{Berlin}
(\byear{1980}).
\doiurl{10.1007/BFb0093923}
\end{bbook}
\endbibitem

\bibitem[\protect\citeauthoryear{Adams and Fournier}{2003}]{ADAMS-2003}
\begin{bbook}
\bauthor{\bsnm{Adams}, \binits{R.A.}},
\bauthor{\bsnm{Fournier}, \binits{J.J.F.}}:
\bbtitle{Sobolev Spaces},
\bedition{2}nd edn.
\bsertitle{Pure and Applied Mathematics},
vol. \bseriesno{140}.
\bpublisher{Academic Press},
\blocation{Amsterdam}
(\byear{2003}).
\doiurl{10.1016/S0079-8169(03)80012-0}
\end{bbook}
\endbibitem

\bibitem[\protect\citeauthoryear{Ciarlet}{2024}]{Ciarlet_2024}
\begin{botherref}
\oauthor{\bsnm{Ciarlet}, \binits{P.}}:
Notes de cours sur les {\'e}quations de {M}axwell et leur approximation.
{HAL} open archive. \url{https://hal.science/hal-03153780v4}
(2024)
\end{botherref}
\endbibitem

\bibitem[\protect\citeauthoryear{Creo and Lancia}{2024}]{Creo_2024}
\begin{barticle}
\bauthor{\bsnm{Creo}, \binits{S.}},
\bauthor{\bsnm{Lancia}, \binits{M.R.}}:
\batitle{The p-curl system in extension domains}.
\bjtitle{Discrete and Continuous Dynamical Systems - S}
\bvolume{17}(\bissue{5\&6}),
\bfpage{2208}--\blpage{2223}
(\byear{2024})
\doiurl{10.3934/dcdss.2023113}
\end{barticle}
\endbibitem

\end{thebibliography}

\end{document}